\documentclass[10pt]{article}

\usepackage{Header}

\title{
    Quantum Differential Equation for Slices\\
    of the Affine Grassmannian
}
\author{Ivan Danilenko}
\date{}

\begin{document}
\maketitle


\begin{abstract}
    The affine Grassmannian associated to a reductive group $\G$ is an affine analogue of the usual flag varieties. It is a rich source of Poisson varieties and their symplectic resolutions. These spaces are examples of conical symplectic resolutions dual to the Nakajima quiver varieties. In this work, we study their quantum connection. We use the stable envelopes of D. Maulik and A. Okounkov\cites{MOk} to write an explicit formula for this connection. The classical part of the multiplication comes from \cites{Da}. The computation of the purely quantum part is done based on the deformation approach of A. Braverman, D. Maulik and A. Okounkov\cites{BMO}. For the case of simply-laced $\G$, we identify the quantum connection with the trigonometric Knizhnik-Zamolodchikov equation for the Langlands dual group $\GLangDual$.
\end{abstract}

\section{Introduction}
\label{SecIntroduction}

\subsection{Overview}

The central object in this work is the moduli space called the affine Grassmannian $\Gr$. It is associated to a complex connected reductive group $\G$, and one can think of it as an analogue of flag varieties for a corresponding Kac-Moody group. It is known to have deep connections with representation theory and Langlands duality \cites{G,MVi1,MVi2}. More precisely, the geometric objects we study are transversal slices in affine Grassmannians. To give an idea of what these slices are, recall that the affine Grassmannian $\Gr$ has a cell structure similar to the Schubert cells in the ordinary flag variety. A transversal slice $\Gr^\lambda_\mu$ describes how one orbit is attached to another. The slices $\Gr^\lambda_\mu$ first appeared in mathematical physics as Coulomb branches in supersymmetic theories \cite{SW}. They are naturally algebraic Poisson varieties and, in some cases, admit a smooth symplectic resolution $\Gr^{\underline{\lambda}}_\mu \to \Gr^\lambda_\mu$, which is a notion of independent mathematical interest \cites{K1, K2, BK}. The goal of this work is to study the quantum cohomologies of these resolutions.

One of our motivations is 3d mirror symmetry. It is known to exchange the Coulomb ($X$) and the Higgs branch ($X^\vee$) of vacuua in a 3d supersymmetric field theory \cites{IS,SW}. In \cite{MOk} D. Maulik and A. Okounkov studied the quantum cohomology for Nakajima quiver varieties \cites{N2,N3,N4}, symplectic resolutions which naturally arise as the Higgs branch of vacua $X^\vee$. Mirror symmetry is a source of conjectural relations between the enumerative geometry of $X$ and $X^\vee$, and understanding the Coulomb side makes it possible to verify this relation directly. 

In this project the main algebraic structure which captures the enumerative geometry of $X$ is the (equivariant) quantum cohomology $\QCoHlgy{\T}{X}$. It is a deformation of usual (equivariant) cohomology by curve counts called Gromov-Witten invariants. The multiplication in $\QCoHlgy{\T}{X}$ for symplectic resolutions $X\to X_0$ gives rise to an interesting flat connection with base $\CoHlgyk{X}{2}$ and fiber $\EqCoHlgy{\T}{X}$, called the quantum connection. It appears as a differential equation on important generating functions in enumerative geometry. The quantum connection was identified with well-known representation-theoretic connections for some symplectic resolutions. For the case of Nakajima quiver varieties this connection was computed in \cite{MOk}, for the Hilbert scheme or $A_n$ resolution in \cite{MOb}, for the cotangent bundle of a flag variety in \cite{Su}. Here we identify the quantum connection in the case of resolutions of slices of the affine Grassmannian $\Gr^{\underline{\lambda}}_\mu$, with the trigonometric Knizhnik-Zamolodchikov connection coming from the representation theory of a reductive group $\G^\vee$. Moreover, the group $\G$ in the definition of $\Gr$ and $\G^\vee$ in the Knizhnik-Zamolodchikov equations turn out to be Langlands-dual, which is natural because of Langlands nature of the geometry of $\Gr$.

The identification of the quantum connection with the Knizhnik-Zamolodchikov equations recently plays an important role in the project of M. Aganagic \cites{A1,A2}. One of the goals of this project is to provide a new approach to categorified knot invariants. The appropriate category on the B-side of mirror symmentry is the derived category of coherent sheaves on $\Gr^{\underline{\lambda}}_\mu$, and the K-theoretic vertex functions satisfy the quantum differential equation. If we know that it coincides with the Knizhnik-Zamolodchikov equation, as we show in this paper, then we are sure that the monodromies act on K-theory of $\Gr^{\underline{\lambda}}_\mu$ exactly as R-matrices in quantum link invariants \cites{RT,W}. Hence, the categorified braid group action on the derived category of $\Gr^{\underline{\lambda}}_\mu$ specializes under decategorification to the well-known link invariants.

\subsection{Overview of the main result}

Quantum (equivariant) cohomology is a deformation of ordinary equivariant cohomology by three-point Gromov-Witten invariants. We define it by the following formula 
\begin{equation*}
	\left\langle 
		\gamma_1, 
		\gamma_2 
		*
		\gamma_3 
	\right\rangle 
	= 
	\left\langle 
		\gamma_1, 
		\gamma_2 
		\cup
		\gamma_3 
	\right\rangle  
	+
	\sum_
	{
		\dd\in \HlgyEffk{ X, \mathbb{Z} }{2}
	}
	\left\langle 
		\gamma_1, 
		\gamma_2,
		\gamma_3 
	\right\rangle_{0,3,\dd} q^\dd
\end{equation*}
for any $\gamma_1, \gamma_2, \gamma_3  \in \EqCoHlgy{\T}{X}$. Here $\left\langle -, -, - \right\rangle^X_{0,3,d}$ are genus zero three-point degree $d$ equivariant Gromov-Witten invariants and $q$ is a formal parameter. We refer to the extra terms in this formula as purely quantum part.

In this work the spaces of interest are resolutions of slices of the affine Grassmannian $X = \Gr^{\ucolambda}_{\comu}$ which are parametrized by a reductive group $\G$, a sequence of minuscule coweights $\ucolambda = \left( \colambda_1, \dots, \colambda_l \right)$ and a dominant coweight $\comu$.

The quantum connection requires the quantum multiplication of special kind, when $\gamma_3$ is a divisor, i.e. $\gamma_3=D\in \EqCoHlgyk{\T}{X}{2}$. To simplify the problem, we can apply the divisor equation \cite{HKK+}:

\begin{equation*}
	\left\langle 
		\gamma_1, 
		\gamma_2, 
		D 
	\right\rangle^X_{0,3,\dd} 
	= 
	\left\langle 
		D, \dd 
	\right\rangle
	\left\langle 
		\gamma_1, 
		\gamma_2
	\right\rangle^X_{0,2,\dd}.
\end{equation*}

We compute the purely quantum part in this paper. The main idea follows \cite{BMO}, it's the reduction to the wall cases. Let us outline what we use to do it.

\subsubsection{Stable basis and purely quantum multiplication}

The first step is to construct a basis in the equivariant cohomology $\EqCoHlgy{\T}{X}$ which is easy to work with. One of the choices naturally comes after localization, i.e. passing to a related algebra
\begin{equation*}
	\EqCoHlgy{\T}{X}_{loc} 
	= 
	\EqCoHlgy{\T}{X} 
	\otimes_{\EqCoHlgy{\T}{\pt}} 
	\Frac \EqCoHlgy{\T}{\pt},
\end{equation*}
where $\Frac \EqCoHlgy{\T}{\pt}$ is the fraction field.

In all the spaces $X$ we consider, the fixed locus $X^\T$ is a finite set of points. This allows us to define a basis by pushing forward the units $1_p\in\EqCoHlgy{\T}{p}$, $p\in X^\T$ under inclusions $i_p\colon p\hookrightarrow X$. The localization Theorem of Atiyah-Bott\cite{AB} implies that this is a basis, called the fixed point basis.

The fixed point basis was used by M. Viscardi in his PhD thesis \cite{V} to compute the purely quantum multiplication in the first nontrivial case, i.e. when $\lambda$ is a sum of two minuscule coweights, and only in the leading order. This computation shows the limitation of this approach. Viscardi computed only the leading term in the purely quantum part of the multiplication and finding the correction terms looks extremely hard if even possible.

The way how we overcome this obstacle and protect ourselves from corrections is to use non-localized classes and rely heavily on dimensional arguments to have an automatic vanishing of corrections. Using pushforward classes $i_{p,*}1$ is not enough because of their high degree, $2 \dim X$. Such a degree still allows our formulas to have barely computable corrections. This is why we use the stable basis $\StabCEpt{\C}{\epsilon}{p}$ introduced in \cite{MOk}, which has degree $\dim X$. It is also given by the images of $1_p\in\EqCoHlgy{\T}{p}$, but under a different map, going in the "wrong" way
\begin{equation*}
	\StabCE{\C}{\epsilon} 
	\colon 
	\EqCoHlgy{\T}{X^\A} 
	\to 
	\EqCoHlgy{\T}{X}.
\end{equation*}
The main choice one has to make is the choice of attracting directions in a torus, which is denoted here by $\C$ as a Weyl chamber. Informally, the classes $\StabCEpt{\C}{\epsilon}{p}$ are "corrected" versions of the (Poincar\'{e} dual) fundamental classes of the attracting varieties to $p$. The notion of attracting variety clearly depends on a choice of $\C$.

We use the classes $\StabCEpt{\C}{\epsilon}{p}$ for our computations. Strictly speaking, these form a basis only after localization, but we still refer to these classes as the stable basis.

These classes behave well for symplectic resolutions. If we have a ($\T$-equivariant) symplectic resolution $\left( X, \omega \right)$
\begin{equation*}
	X
	\to
	X_0
\end{equation*}
of an affine Poisson variety $X_0$, then it is known\cite{KV} to have a universal family of deformations of a pair $\left( X, \omega \right)$
\begin{equation} \label{EqUniversalFamily}
	\begin{tikzcd}
		X \arrow[r,hook] \arrow[d]
		&
		\widetilde{X}^u \arrow[d]
		\\
		0 \arrow[r,hook]
		&
		\CoHlgyk{X}{2}
	\end{tikzcd}
\end{equation}
where the base can be identified with the second cohomology of $X$. Then the stable envelopes naturally extend to this family. In a certain sense they are invariant under such deformations.

If $X$ is a conical symplectic resolution, then we have two simplifications: 

\begin{itemize}
\item
	The two-point Gromov-Witten invariants are divisible by $\hbar$:
	\begin{equation*}
		\left\langle 
			\gamma_1, 
			\gamma_2
		\right\rangle^X_{0,2,\dd}
		=
		0
		\mod \hbar.
	\end{equation*}
\item
	The off-diagonal restrictions of the stable envelopes are divisible by $\hbar$:
	\begin{equation*}
		\StabCEptpt{\C}{\epsilon}{p}{q}
		=
		0
		\mod \hbar.
	\end{equation*}
\end{itemize}

If we take $\gamma_1 = \StabCEpt{-\C}{\DualPol{\epsilon}}{q}$ and $\gamma_2 = \StabCEpt{\C}{\epsilon}{p}$ then by dimension argument
\begin{equation*}
	\left\langle 
		\StabCEpt{-\C}{\DualPol{\epsilon}}{q}, 
		\StabCEpt{\C}{\epsilon}{p}
	\right\rangle^X_{0,2,\dd}
	\in \hbar \mathbb{Q}
\end{equation*}

Then one can reduce the computation of these invariants to the localization with respect to subtorus $\A = \ker \hbar$. As said above, almost all restrictions of stable envelopes vanish and the computation becomes similar to the computation of the leading term in the fixed point basis.

In \refSec{SecQuantum} we do this computation following the method of A. Braverman, D. Maulik and A. Okounkov\cite{BMO} of reducing the enumerative problem to contributions of walls in K\"{a}hler parameters, that is, in the base of the family \refEq{EqUniversalFamily}.

The following explicit formula
\begin{align*}
	\ChernE{\T}{i} \qdot \StabCEpt{\C}{\epsilon}{p}
	=
	&
	-\hbar
	\sum\limits_{i<j}
	\StabCE{\C}{\epsilon}
	\left[
		\dfrac{q^{e_i - e_j}}{1-q^{e_i - e_j}}
		\OmegaOperatorTilde{ij}{\mathrm{st},\epsilon}
		\left(
			p
		\right)
		+
		\dfrac{q^{2(e_i - e_j)}}{1-q^{2(e_i - e_j)}}
		\OmegaOperatorTilde{ij}{\mathrm{lg},\epsilon}
		\left(
			p
		\right)
	\right]
	\\
	&+\hbar
	\sum\limits_{j<i}
	\StabCE{\C}{\epsilon}
	\left[
		\dfrac{q^{e_j - e_i}}{1-q^{e_j - e_i}}
		\OmegaOperatorTilde{ji}{\mathrm{st},\epsilon}
		\left(
			p
		\right)
		+
		\dfrac{q^{2(e_j - e_i)}}{1-q^{2(e_j - e_i)}}
		\OmegaOperatorTilde{ji}{\mathrm{lg},\epsilon}
		\left(
			p
		\right)
	\right]
\end{align*}
is the main result of \refSec{SecQuantum}. We refer reader to this section for the notation.

\subsubsection{Identification}

Combining the results on purely quantum multiplication with the classical multiplication in \cite{Da}, we get for a simply-laced $\G$
\begin{align}
	\notag
	\ChernE{\T}{i} * \StabCEpt{\C}{\epsilon}{p}
	&=
	\StabCE{\C}{\epsilon}
	\left[
		\Hterm{i}
		\left(
			p
		\right)
		+
		\hbar
		\sum\limits_{j \neq i}
		\dfrac
		{
			q^{e_i}
			\OmegaOperator{ij}{\C,\epsilon}
			\left(
				p
			\right) 
			+
			q^{e_j}
			\OmegaOperator{ij}{-\C,\epsilon}
			\left(
				p
			\right)
		}
		{
			q^{e_i} - q^{e_j}
		}
	\right]
	\\
	&-
	\hbar
	\StabCE{\C}{\epsilon}
	\left[
		\sum\limits_{j < i}
		\dfrac
		{
			q^{e_j}
			\Kterm{ij}
			\left(
				p
			\right)
		}
		{
			q^{e_i} - q^{e_j}
		}
		+
		\sum\limits_{j > i}
		\dfrac
		{
			q^{e_i}
			\Kterm{ij}
			\left(
				p
			\right)
		}
		{
			q^{e_i} - q^{e_j}
		}
	\right],
	\label{EqQuantMultFormula}
\end{align}
This formula is a power series via the expansion in the region chosen with respect to the effective cone $\HlgyEffk{X}{2} \subset \Hlgyk{X}{2}$
\begin{equation*}
	q^{e_1}
	\ll 
	q^{e_2} 
	\ll 
	\dots 
	\ll 
	q^{e_l}.
\end{equation*}

The quantum connection is a connection on a vector bundle over an open part of $\EqCoHlgyk{\T}{X}{2}$ with a fiber $\EqCoHlgy{\T}{X}$. The of sections of these bundle in a formal neighbourhood can be identified with $\QCoHlgy{\T}{X}$. Then the quantum connection is given by the following formula
\begin{equation*}
	\nabla^Q_D
	=
	\partial_D - D *
\end{equation*}
for the tangent vector in the base $D \in \EqCoHlgyk{\T}{X}{2}$.

Taking into account the formula \refEq{EqQuantMultFormula}, we see that it looks similar to the trigonometric Knizhnik-Zamolodchikov connection (see \cite{EFK}, for example). Explicitly one can write it as a family of commuting differential operators
\begin{equation*}
	\nabla^{KZ}_i
	=
	z_i\dfrac{\partial \phantom{z_i}}{\partial z_i}
	-
	H^i
	- 
	\hbar
	\sum_{j\neq i} 
	\dfrac
	{
		z_i \Omega_{\C}^{ij}+z_j \Omega_{-\C}^{ij}
	}
	{
		z_i - z_j
	}
\end{equation*}
on the space
\begin{equation*}
	\mathbb{C}[\hbar](z_1,\dots,z_l)
	\otimes
	V_1 
	\otimes
	\dots 
	\otimes 
	V_l
\end{equation*}
or, more geometrically, a flat connection on sections of a certain vector bundle over a base with coordinates $z_i$.

The trigonometric Knizhnik-Zamolodchikov connection appeared in conformal field theory \cite{KZ} and representation theory of affine Lie groups \cite{EFK}.

The terms with $\Kterm{ij}$ in \refEq{EqQuantMultFormula} can be removed by applying a gauge transform. Then the main thing remaining in the identification of the operators $\OmegaOperator{ij}{\C,\epsilon}$ and $\Hterm{i}$ with $\Omega^{ij}_{\C}$ and $H^i$. To do it we first have to identify variables, spaces and parameters on both sides.

To make the identification with the quantum connection of $\Gr^{\ucolambda}_{\comu}$, we take the representations $V_i$ to be $V_{\colambda_i}$, the irreducible finite dimensional representations with highest weight $\colambda_i$. Moreover, we restrict the tensor product to the subspace of weight $\comu$
\begin{equation*}
	\mathbb{C}[\hbar](z_1,\dots,z_l)
	\otimes
	V_{\colambda_1}
	\otimes
	\dots 
	\otimes 
	V_{\colambda_l}.
\end{equation*}

Then we identify similar objects on two sides:

\begin{center}
	\begin{tabular}{ | c | c | }
	\hline
		Trigonometric KZ
		&
		Quantum
	\\
	\hline
		$z_i$ 
		& 
		$q^{e_i}$
	\\
	\hline
		Expansion region $|z_i|\ll|z_j|$ 
		& 
		Effective cone $e_i-e_j>0$
	\\
	\hline
		Choice of $\C$ in $\Omega_{\C}$ 
		& 
		Choice of $\C$ in $\Stab_{\C}$
	\\
	\hline
		Basis $v_{\mu_1}\otimes\dots\otimes v_{\mu_l}$ 
		&
		Stable envelopes $\StabCpt{\C}{p}$\\
	\hline
	\end{tabular}
\end{center}

These identifications are the main content of \refSec{SecRepTheory}. Finally, the \refThm{ThmKZequalsQ} we prove that the trigonometric Knizhnik-Zamolodchikov connection matches with the quantum connection of $\Gr^{\ucolambda}_{\comu}$ for a simply-laced $\G$.

\subsection{Structure}

The paper is organized as follows. In \refSec{SecSlices} we recall the fundamental results about the slices in the affine Grassmannian. In \refSec{SecQuantum} we do the main computation, the quantum multiplication by a divisor. In \refSec{SecRepTheory} we show how the result from \refSec{SecQuantum} matches with the trigonometric Knizhnik-Zamolodchikov equation if one introduces appropriate identification with objects from representation theory. Finally, in Appendix \ref{AppendixA} and Appendix \ref{AppendixB} we summarize the results about $\T$-invariant curves and families of deformations that we use in the main text.

\subsection*{Acknowledgments}

The author is thankful to Mina Aganagic, Joel Kamnitzer, Henry Liu, Davesh Maulik, Andrei Negut, Andrei Okounkov, Andrey Smirnov, Changjian Su, and Shuai Wang for discussions related to this paper.

\section{Slices of the affine Grassmannian}
\label{SecSlices}

In this chapter we recall some facts about slices of the affine Grassmannians.

\subsection{Representation-theoretic notation}

Let us present notation we use for common representation-theoretic objects. One major difference from standard notation is that we use non-checked notation for coobjects (coweights, coroots, etc) and checked for ordinary ones (weighs, roots, etc). This is common in the literature on affine Grassmannian since it simplifies notation. And unexpected side effect is that weights in equivariant cohomology will have checks. We hope this won't cause confusion.

\begin{itemize}
\item
	$\G$ is a connected simple complex group unless stated otherwise,
\item
	$\A\subset \G$ is a maximal torus,
\item
	$\B\supset \A$ is a Borel subgroup,
\item
	$\mathfrak{g}$ is the Lie algebra of $\G$
\item
	$\CocharLattice{\A} = \Hom(\Gm, \A)$ is the cocharacter lattice,
\item
	$\DominantCocharLattice{\A} \subset \CocharLattice{\A}$ the submonoid of dominant cocharacters.
\item
	$\CoweightLattice \supset \CocharLattice{\A}$ is the coweight lattice, $\CoweightLattice = \CocharLattice{\A}$ if $\G$ is adjoint type.
\item
	$W = N(\A)/\A$ is the Weyl group of $\G$,
\item
	$\mathfrak{g} = \mathfrak{h} \oplus \bigoplus\limits_{\wtalpha} \grootsub{\wtalpha}$ is the root decomposition of $\mathfrak{g}$.
\item
	$\wtrho$ is the halfsum of positive roots.
\item
	$\CorootScalar{\bullet}{\bullet}$ Weyl-invariant scalar product on coweight space $\CoweightLattice \otimes_{Z} \mathbb{R}$, normalized in such a way that for the shortest coroot the length squared is $2$ (equivalently, the length squared of the longest root is $2$).
\end{itemize}

We identify all weights with Lie algebra weights (in particular we use additive notation for the weight of a tensor product of weight subspaces).

\subsection{Affine Grassmannian}
Let $\OO = \mathbb{C}[[t]]$ and $\K = \mathbb{C}((t))$. We will refer to $\Disk = \Spec \OO$ and $\PDisk = \Spec \K$ as the formal disk and the formal punctured disk respectively.

Let $G$ be a connected reductive algebraic group over $\mathbb{C}$. The \textbf{affine Grassmannian} $\Gr_\G$ is the moduli space of

\begin{equation} \label{EqGrAsModuli}
	\left\lbrace
		\left( 
			\mathcal{P}, 
			\varphi 
		\right)
		\left\vert 
			\begin{matrix}
				\mathcal{P}
				\text{ is a }
				\G
				\text{-principle bundle over }
				\Disk,
				\\
				\varphi
				\colon 
				\mathcal{P}_0\vert_{\PDisk} 
				\xrightarrow{\sim} 
				\mathcal{P}\vert_{\PDisk}
				\text{ is a trivialization of }
				\mathcal{P}
				\\
				\text{ over the punctured disk }
				\PDisk
			\end{matrix}
		\right.
	\right\rbrace
\end{equation}
where $\mathcal{P}_0$ is the trivial principal $\G$-bundle over $\Disk$. It is representable by an ind-scheme.

\medskip

One can give a less geometric definition of the $\mathbb{C}$-points of the affine Grassmannian
\begin{equation} \label{EqGrAsCoset}
	\Gr_\G(\mathbb{C}) 
	= 
	\GK / \GO
\end{equation}
where we use notation $\GRing{R}$ for $R$-points of the scheme $\G$ given a $\mathbb{C}$-algebra $R$. We will use this point of view when we talk about points in $\Gr_\G$.

The see that these are exactly the $\mathbb{C}$-points of $\Gr_\G$ (i.e. the pairs $\left( \mathcal{P}, \varphi \right)$ as in \refEq{EqGrAsModuli}), note that on $\Disk$ any $\G$-principle bundle is trivializable, so take one trivialization $\psi\colon \mathcal{P}_0 \xrightarrow{\sim} \mathcal{P}$. This gives a composition of isomorphisms
\begin{equation} \label{EqGrEquivalenceOfDefs}
	\mathcal{P}_0\vert_{\PDisk} 
	\xrightarrow{\psi\vert_{\PDisk} } 
	\mathcal{P}\vert_{\PDisk} 
	\xrightarrow{\varphi^{-1}} 
	\mathcal{P}_0\vert_{\PDisk}
\end{equation}
which is a section of $\mathcal{P}_0\vert_{\PDisk}$, i.e. an element of $\GK$. Change of the trivialization $\psi$ precomposes with a section of $\mathcal{P}_0$, in other words, precomposes with an element of $\GO$. This gives a quotient by $\GO$ from the left.

\medskip

We restrict ourselves to the case of connected simple (possibly not simply-connected) $\G$ in what follows without loosing much generality. Restating the results to allow arbitrary connected complex reductive group is straightforward.

From now on we fix a connected simple complex group $\G$. Since it can't cause confusion, we omit later $\G$ in the notation $\Gr_\G$ and just write $\Gr$.

\medskip

The group $\GK\rtimes \Gm$ naturally acts on $\Gr$. In the coset formulation the $\GK$-action is given by left multiplication and $\Gm$-acts by scaling variable $t$ in $\K$ with weight $1$. In the moduli space description $\GK$ acts by changes of section $g \cdot \left( \mathcal{P}, \varphi \right) = \left( \mathcal{P}, \varphi g^{-1} \right)$ (these two actions are the same action if one takes into account identification \refEq{EqGrEquivalenceOfDefs}), $\Gm$ scales $\Disk$ such that the coordinate $t$ has weight $1$. We will later denote this $\Gm$ as $\LoopGm$ when we want to emphasize that we use this algebraic group and its natural action.

We will need actions by subgroups of this group, namely $\A \subset \G \subset \GO \subset \GK$. It will also be useful to consider extended torus $\T = \A \times \LoopGm$, where $\LoopGm$ part comes from the second term in the product $\GK\rtimes \LoopGm$.

The canonical projection $\T\to\LoopGm$ gives a character of $\T$ which we call $\hbar$ (this explains the subscript $\hbar$ in the notation $\LoopGm$). Then the weight of coordinate $t$ on $\Disk$ is $\hbar$ by the construction of the $\LoopGm$-action.

\begin{remark}
We call the maximal torus of $\G$ by $\A$ to have notation similar to \cite{MOk}, i.e. $\T$ is the maximal torus acting on the variety, and $\A$ is the subtorus of $\T$ preserving the symplectic form.
\end{remark}

\medskip

The partial flag variety $\G/\PP$ (where $\PP$ is parabolic) has a well-known decomposition by orbits of $\B$-action called Schubert cells. The affine Grassmannian has a similar feature.

One can explicitly construct fixed points of the $\A$-action. Given a cocharacter ${\colambda} \colon \Gm \to \T$, one can construct a map using natural inclusions
\begin{equation*}
	\PDisk 
	\hookrightarrow 
	\Gm 
	\xrightarrow{\colambda} 
	\T
	\hookrightarrow 
	\G
\end{equation*}
i.e. an element of $\GK$. Projecting it naturally to $\Gr$ we get an element $\Tfixed{\colambda} \in \Gr$. These points are $\A$-fixed and even the following stronger statements are true.

\begin{proposition} \label{PropGrFixedLocus}
	\leavevmode
	\begin{enumerate}
	\item 
		$\Gr^\A = \bigsqcup\limits_{{\colambda} \in \CocharLattice{\A}} \lbrace \Tfixed{\colambda} \rbrace $
	\item 
		Moreover, $\Gr^\T = \Gr^\A$
	\end{enumerate}
\end{proposition}

Using these elements we define
\begin{equation*}
	\Gr^{\colambda} 
	= 
	\GO \cdot \Tfixed{\colambda}
\end{equation*}
as their orbits.
These have the following properties

\begin{proposition} \label{PropGrCells}
	\leavevmode
	\begin{enumerate}
	\item
		For any $w\in W$ one has $\Gr^{w{\colambda}} = \Gr^{\colambda}$.
	\item 
		$\Gr = \bigsqcup\limits_{{\colambda} \in \DominantCocharLattice{\A} } \Gr^{\colambda}$
	\item 
		$\overline{\Gr^{\colambda}} = \bigsqcup\limits_{\genfrac{}{}{0pt}{}{{\comu} \leq {\colambda}}{{\comu} \in \DominantCocharLattice{\A} }} \Gr^{\comu} $
	\item
		As a $\G$-variety, $\Gr^{\colambda}$ is isomorphic to a the total space of a $\G$-equivariant vector bundle over a partial flag variety $\G/\PP^-_{\colambda}$ with parabolic $\PP^-_{\colambda}$ whose Lie algebra contains all roots $\wtalpha$ such that $\left\langle \wtalpha, \colambda \right\rangle \leq 0$.
	\item
		$\Disk$-scaling $\LoopGm$ in $\T$ scales the fibers of this vector bundle over $\G/\PP^-_{\colambda}$ with no zero weights.
	\item
		$\Gr^{\colambda}$ is the smooth part of $\overline{\Gr^{\colambda}}$. In particular, $\overline{\Gr^{\colambda}}$ is smooth iff ${\colambda}$ is a minuscule coweight or zero.
	\end{enumerate}
\end{proposition}

\medskip

\begin{corollary} \label{CorTCellSuructure}
	$\Gr$ admits a $\T$-invariant cell structure.
\end{corollary}

\begin{proof}
	Each $\G/\PP^-_{\colambda}$ admits an $\A$-invariant cell structure, namely Schubert cells. $\A$-equivariant vector bundles over these cells give $\A$-invariant cells for each of $\Gr^{\colambda}$. Hence we have an $\A$-invariant cell structure for $\Gr$.

	Moreover this cell structure is invariant under $\Disk$-scaling $\LoopGm$-action, because it just scales the fibers of the vector bundles. Finally, we get that this cell decomposition $\T$-invariant.
\end{proof}

\begin{corollary} \label{CorOrbitsInGrGr}
	Orbits of diagonal $\GK$-action on $\Gr\times\Gr$ are in bijection with dominant cocharacters.
\end{corollary}

\begin{proof}
	First recall that for any group $H$ and a subgroup $K$ one has a bijection
	\begin{equation*}
		\left\lbrace
			\begin{matrix}
				H\text{-diagonal orbits of}\\
				H/K \times H/K
			\end{matrix}
		\right\rbrace
		\simeq
		K\backslash H / K
		\simeq
		\left\lbrace
			\begin{matrix}
				K\text{-orbits of}\\
				H/K
			\end{matrix}
		\right\rbrace
	\end{equation*}
	Applying this for $H = \GK$ and $K = \GO$ and that the rightmost set is in bijection with dominant cocharacters by \refProp{PropGrCells} we get the statement.
\end{proof}

We write $L_1\xrightarrow{\colambda} L_2$ if $(L_1,L_2)\in \Gr \times \Gr$ is in the $\GK$-orbit indexed by a dominant ${\comu} \leq {\colambda}$.

\begin{remark}
	Explicitly, one says that $L_1\xrightarrow{\colambda} L_2$ iff picking representatives $L_1 = g_1 \GO$, $L_2 = g_2 \GO$ with $g_1,g_2 \in \GK$ we have
	\begin{equation*}
		\GO g^{-1}_1 g_2 \GO 
		\subset 
		\GO \Tfixed{{\comu}} 
		\text{ for } 
		{\comu} \leq {\colambda}
	\end{equation*}
	or, equivalently,
	\begin{equation*}
		g^{-1}_1 g_2 \GO 
		\in 
		\overline{ \GO \cdot \Tfixed{\colambda} }
	\end{equation*}
	Note that this independent on the choice of the representatives of $L_1$, $L_2$.
\end{remark}

\subsection{Transversal slices}
As one can see, elements $g(t)$ of $\GO \subset \G$ can be characterized as the elements of $g(t)\in\GK$ such that $\lim\limits_{t\to 0} g(t) \in \G \subset \GK$.

We can use a similar way to define a subgroup transversal to $\GO$. Let
\begin{equation*}
	\GOne 
	= 
	\left\lbrace 
		g(t)\in\G 
		\bigg\vert \; 
		\lim_{t\to \infty} g(t) 
		= 
		\UnitG 
	\right\rbrace
\end{equation*}
This is a subgroup of $\GK$.

$\GO$ and $\GOne$ are transversal in the sense that
\begin{equation*}
	T_\UnitG \GK 
	= 
	T_\UnitG \GO 
	\oplus 
	T_\UnitG \GOne
\end{equation*}
because
\begin{align*}
	&T_\UnitG \GK = \gK{t} \\
	&T_\UnitG \GO = \gO{t} \\
	&T_\UnitG \GOne = \gOne{t}
\end{align*}

One uses $\GOne$ to analyze how $\GO$-orbits in $\Gr$ are attached to each other. The ind-varieties
\begin{equation*}
	\Gr_{\comu} 
	= 
	\GOne \cdot \Tfixed{{\comu}}
\end{equation*}
is the slice transversal to $\Gr^{\comu}$ at $\Tfixed{{\comu}}$. We will use use the convention that ${\comu}$ is dominant.

The main object of interest for us are the following varieties
\begin{equation*}
	\Gr^{\colambda}_{\comu} 
	= 
	\overline{\Gr^{\colambda}} 
	\cap 
	\Gr_{\comu}
\end{equation*}
These are non-empty iff ${\comu} \leq {\colambda}$.

\begin{proposition} \label{PropSlicesProperties}
	\leavevmode
	\begin{enumerate}
	\item 
		$\Gr^{\colambda}_{\comu}$ is a $\T$-invariant subscheme of $\Gr$;

	\item 
		The only $\A$-fixed (and $\T$-fixed) point of $\Gr^{\colambda}_{\comu}$ is $\Tfixed{{\comu}}$;

	\item 
		The $\Disk$-scaling $\LoopGm$-action contracts $\Gr^{\colambda}_{\comu}$ to $\Tfixed{{\comu}}$.
	
	\item 
		$\Gr^{\colambda}_{\comu}$ is a normal affine variety of dimension $\langle 2\wtrho , {\colambda} - {\comu} \rangle$.
\end{enumerate}
\end{proposition}

\subsection{Poisson structure}

The affine Grassmannian has a natural Poisson structure. Let us describe it using the Manin triples introduced by V. Drinfeld\cites{Dr1,Dr2}. First recall that $\GK$ is a Poisson ind-group. It is given by a Manin triple
\begin{equation*}
	\left(
		\gO{t},
		\gOne{t},
		\gK{t}
	\right),
\end{equation*}
if we equip $\gK{t}$ with the standard $\G$-invariant scalar product
\begin{equation} \label{EqPairing}
	\left(
		x t^n,
		y t^m
	\right)
	=
	\CorootScalar{x}{y}
	\delta_{n+m+1,0},
\end{equation}
where $\delta_{\bullet,\bullet}$ is the Kronecker delta. Then the subalgebras $\gO{t}$, $\gOne{t}$ are isotropic and the pairing between them is non-degenerate. Thus $\gK{t}$ is a Lie bialgebra and $\gO{t}$ is a Lie subbialgebra.

We have that $\GK$ is a Poisson ind-group and $\GO$ is its Poisson subgroup. By theorem of Drinfeld\cite{Dr2} the quotient $\Gr = \GK/\GO$ is Poisson. The explicit construction shows that since the pairing \refEq{EqPairing} has weight $\hbar$, we have the following important property

\begin{proposition}
    The Poisson bivector on $\Gr$ is a $\T$-eigenvector in $\Hlgyk{\Gr, \Lambda^2 \Gr}{0}$ with weight $-\hbar$.
\end{proposition}
\begin{proof}
    The pairing is invariant with respect to $\A$, and is of degree $-1$ in $t$, so the $\T$ torus acts with weigh $-\hbar$.
\end{proof}

The slices $\Gr^{\colambda}_{\comu} \subset \Gr$ are Poisson subvarieties. Moreover, their smooth parts are symplectic leaves, hovewer not all symplectic leaves have this form: they might be shifted by the $\G$-action.

\subsection{Resolutions of slices}

One of the reasons what motivates us to study varieties $\Gr^{\colambda}_{\comu}$ is that they provide a family of symplectic resolutions. Let us recall some basics on them. See \cite{K1} for more details.

A \textbf{symplectic resolution} is a smooth algebraic variety $X$ equipped with a closed non-degenerate 2-form $\omega$ such that the canonical map
\begin{equation*}
	X \to X_0 = \Spec \CoHlgyk{X,\OO_X}{0}
\end{equation*}
is a birational projective map.

Well-known examples of the symplectic resolutions are Du Val singularities, the Hilbert scheme points on a plane (or a Du Val singularity), cotangent bundles to flag varieties, Nakajima quiver varieties. 

$X_0$ carries a natural structure of (possible singular) Poisson affine scheme. One can pose a question if affine Poisson varieties $\Gr^{\colambda}_{\comu}$ we study admit a symplectic resolution, i.e. a smooth symplectic $X$ with $X_0 = \Gr^{\colambda}_{\comu}$ such that the assumptions above hold. Unfortunately, these spaces do not always possess a symplectic resolution. We'll follow \cite{KWWY} to construct resolution in special cases.

For construction of resolutions it is convenient to assume that the cocharacter lattice is as large as possible, that is, equal to the coweight lattice. This happens when $\G$ is of adjoint type (the other extreme from being semi-simple). This assumption doesn't change the space we are allowed to consider: If $\widetilde{\G} \to \G$ is a finite cover, then there is a natural inclusion $\Gr_{\G} \hookrightarrow \Gr_{\widetilde{\G}}$ with the image being a collection of connected components of $\Gr_{\widetilde{\G}}$. Since the slices are connected, they belong only to one component and can be thought as defined for $\widetilde{\G}$. Last but not least the $\T$-equivariant structures are compatible if one takes into account that the torus for $\G$ is a finite cover over the torus for $\widetilde{\G}$.

From now on we assume that $\G$ is a connected simple group of adjoint type (so the center is trivial) and make no distinction between cocharacters and coweights unless overwise is stated.

Consider a sequence of dominant weights $\ucolambda = \left( \colambda_1, \dots \colambda_l  \right)$ Define the (closed) convolution product:

\begin{equation*}
	\Gr^{ \ucolambda } =
	\left\lbrace 
		\left(
			L_1,\dots,L_l
		\right) 
		\in \Gr^{\times l}
		\left\vert
			\UnitG \cdot \GO
			\xrightarrow{\colambda_1}
			L_1 
			\xrightarrow{\colambda_2} 
			\dots 
			\xrightarrow{\colambda_{l-1}}
			L_{l-1} 
			\xrightarrow{\colambda_{l}}
			L_l
		\right.
	\right\rbrace
\end{equation*}

There's a map
\begin{align*}
	m'_{ \ucolambda } 
	\colon
	\Gr^{ \ucolambda } 
	&\to 
	\Gr,
	\\
	(L_1,\dots,L_l)
	&\mapsto 
	L_l.
\end{align*}

One can show that the image of $m'_{ \ucolambda }$ is inside $\Gr^{\overline{\colambda}} $, where ${\colambda} = \sum_i \colambda_i$, so we can actually make a map
\begin{equation*}
	m_{ \ucolambda } 
	\colon
	\Gr^{ \ucolambda }
	\to
	\overline{\Gr^{\colambda}}.
\end{equation*}

Then let us define
\begin{equation*}
	\Gr^{\ucolambda}_\mu 
	= 
	m_{ \ucolambda }^{-1}
	\left( 
		\Gr^{\colambda}_{\comu}
	\right).
\end{equation*}

These spaces and maps are useful to construct symplectic resolutions of singular affine $\Gr^{\colambda}_{\comu}$. Let all $\colambda_i$ in $\ucolambda$ be fundamental coweights (which means this is a maximal sequence which splits ${\colambda} = \sum_i \colambda_i$ into nonzero dominant coweights). We will also need the corresponding irreducible highest weight representations $V(\colambda_k)$ of the Langlands dual group $\GLangDual $. Then there's the following result (see \cite{KWWY}, Theorem 2.9)

\begin{theorem} \label{ThmSymplecticResolutionExistence}
	The following are equivalent
	\begin{enumerate}
	\item
		$\Gr^{\colambda}_{\comu}$ possesses a symplectic resolution;
	\item
		$\Gr^{ \ucolambda }_{\comu}$ is smooth and thus $m_{ \ucolambda }$ gives a symplectic resolution of singularities of $\Gr^{\colambda}_{\comu}$.
	\item
		There do not exist coweights $\conu_1, \dots, \conu_l$ such that $\sum_i \conu_i = {\comu}$, for all $k$, $\conu_k $ is a weight of $V(\colambda_k)$ and for some $k$ , $\conu_k$ is not an extremal weight of $V(\colambda_k)$ (that is, not in the Weyl orbit of $\colambda_k$).
	\end{enumerate}
\end{theorem}

\begin{corollary}
	If all $\colambda_i$ are minuscule, then all three statements of the \refThm{ThmSymplecticResolutionExistence} hold.
\end{corollary}

\begin{corollary}
	If $G = \PSL{n}$, then all fundamental coweights are minuscule and for all dominant coweights ${\comu} \leq {\colambda}$ all three statements of the \refThm{ThmSymplecticResolutionExistence} hold.
\end{corollary}

\begin{corollary}
	If $G$ is not $\PSL{n}$, then there exist dominant coweights ${\comu} \leq {\colambda}$ such that all three statements of the \refThm{ThmSymplecticResolutionExistence} \textbf{are not} satisfied.
\end{corollary}

From this time on we assume that the pair of dominant coweights ${\comu} \leq {\colambda}$ is such that the statements of the \refThm{ThmSymplecticResolutionExistence} hold.

There are important special cases of these symplectic resolutions

\begin{example}
	Let $G=\PSL{2}$, then the cocharacter and coroot lattice are character (or, equivalently, weight) and root lattices of $\SL{2}$ respectively. For any integer $n\geq 1$ by setting $\colambda = (n+1)\coomega$ and $\comu = (n-1)\coomega$ we get that $\Gr^{\colambda}_{\comu}$ is $\mathcal{A}_n$-singularity and the map $\Gr^{\left(\coomega, \dots, \coomega \right)}_{\comu} \to \Gr^{\colambda}_{\comu}$ is a symplectic resolution of singularity.
\end{example}

\begin{example}
	Let $\conu$ be a minuscule coweight of $\G$ and $\iota$ is the Cartan involution and set $\colambda = \conu + \iota\conu $, $\comu = 0$. Then we get a resolution $\Gr^{\left(\conu, \iota\conu \right)}_0 \to \Gr^{\conu+ \iota\conu}_0$ is the Springer resolution. Here $\Gr^{\left(\conu, \iota\conu \right)}_0 \cong T^*\left( \G/\PP^-_\conu \right)$ (the cotangent bundle of a partial flag variety), where $\PP^-_\conu$ is the maximal parabolic corresponding to $\conu$ (i.e. the Lie algebra of $\PP^-_\conu$ has all roots $\wtalpha$ of $\G$ satisfying $\left\langle \conu, \wtalpha \right\rangle \leq 0$). $\Gr^{\conu+ \iota\conu}_0$ is the affinization.

\end{example}

\begin{example}
	Let $G = \PSL{n}$, $\coomega_1$ be highest weight of the defining representation of $\SL{n}$ (so it is a $\PSL{n}$-coweight). Set also $\colambda = n \coomega_1$, $\underline{\colambda} = \left( \coomega_1, \dots, \coomega_1 \right)$ and ${\comu} = 0$. Then $\Gr^{\left(\coomega_1, \dots, \coomega_1 \right)}_0 \cong T^*\left( \G/\B \right)$ is the cotangent bundle of full flag variety, $\Gr^{n \coomega_1}_0$ is the nilpotent cone and $\Gr^{\left(\coomega_1, \dots, \coomega_1 \right)}_0 \to \Gr^{n \coomega_1}_0$ is the Springer resolution.
\end{example}

\medskip

Recall that $\GK\rtimes\LoopGm$ acts on $\Gr$, so it also acts diagonally on $\Gr^{\times l}$ for any $l$ (see e.g. \cites{Da}). 

\begin{proposition}
	\leavevmode
	\begin{enumerate}
	\item
		For any sequence $\ucolambda$ the subscheme $\Gr^{ \ucolambda } \subset \Gr^{\times l}$ is $\GO\rtimes\LoopGm$-invariant
	\item
		The resolution of singularities $m_{ \ucolambda } \colon \Gr^{ \ucolambda }_{\comu} \to \Gr^{\colambda}_{\comu}$ is $\T$-equivariant.
	\end{enumerate}
\end{proposition}

\subsection{Equivariant geometry}

Here we review some facts about $\T$-action on the slices. One can find the proofs in \cites{Da}.

\begin{proposition} \label{PropSliceFixedLocus}
	The fixed loci $X^\T$ and $X^\A$ are equal to the following disjoint union of finitely many points
	\begin{equation*}
		X^\T 
            =
            X^\A
		= 
		\left\lbrace 
			\left( 
				\Tfixed{\comu_1}, 
				\dots, 
				\Tfixed{\comu_l} 
			\right) 
			\in 
			\Gr^{\times l} 
			\left\vert
				\begin{matrix}
					\;\forall i \;\: 
					\comu_i - \comu_{i-1} 
					\text{ is a weight of } 
					V(\colambda_i)
					\\
					\comu_l = {\comu}, 
					\quad 
					\comu_0 = 0
				\end{matrix}
			\right.
		\right\rbrace.
	\end{equation*}
\end{proposition}

\begin{proof}
	Since the $\T$-action is diagonal, we have that $\left(L_1,\dots,L_l\right)\in X^\T$ iff $\forall i$ $L_i$ has form $\Tfixed{\comu_i}$ for some coweight $\comu_i$. Next, conditions for such points to be in $X$ are
	\begin{enumerate}
	\item[a]
		A condition to be in $\Gr^{\ucolambda}$. That is, $\forall i$ $L_{i-1}\xrightarrow{\colambda_i} L_i $ which gives that $\comu_i-\comu_{i-1}$ is a weight of $V(\colambda_i)$ with convention $\comu_0 = 0$ to make $L_0=\Tfixed{0}=\UnitG\cdot\GO$.
	\item[b]
A condition to be in $m_{\ucolambda }^{-1}\left( \Gr_{\comu} \right)$. That is, $L_l = m_{\ucolambda } (L_1,\dots, L_l) = \Tfixed{{\comu}}$, which is equivalent to $\comu_l = {\comu} $.
	\end{enumerate}
	This gives the description of $X^\T$ in the proposition.
        Using $\Gr^\A = \Gr^\T$, a similar argument shows the same for $X^\A$.
\end{proof}

\begin{remark}
	If $\colambda_i$ is minuscule, the constraint of $\conu_i$ to be a weight of $V(\colambda_i)$ just means that $\comu_i - \comu_{i-1}  \in W\colambda_i$.
\end{remark}

\begin{remark}
	If all $\colambda_i$ are minuscule, the fixed points are in bijection with a basis of the weight ${\comu}$ subspace of the $\GLangDual$-representation $V(\colambda_1)\otimes\dots\otimes V(\colambda_l)$.
\end{remark}

\begin{remark}
    The tight connection between the cohomology of sheaves on the resolution of slices for $\G$ and the representation theory of the Langlands dual $\GLangDual$ is known as geometric Satake correspondence \cites{G,MVi1,MVi2}. A well-known spectial case is
    \begin{equation*}
	\CoHlgy
	{
		\Gr^
		{
			\left( 
				\colambda_1, 
				\dots, 
				\colambda_l 
			\right)
		}
		_{\comu}
	}
	\simeq
	V_{\colambda_1}
	\otimes
	\dots
	\otimes
	V_{\colambda_l}
	\left[
		\comu
	\right]
    \end{equation*}
    Since the resolution of slices are $\T$-equivariantly formal, this also follows from $\T$-equivariant cohomology and the combinatorics of the fixed points.
\end{remark}

\begin{notation}
	For a fixed point $p\in X^{\T}$ there are two related sequences of coweights. Let 
	\begin{equation*}
		p =
		\left(
			\Tfixed{0},
			\Tfixed{\comu_1}, 
			\dots, 
			\Tfixed{\comu_{l-1}}, 
			\Tfixed{\comu}
		\right).
	\end{equation*}

	Then for all $i$, $0 < i < l$, we define
	\begin{equation*}
		\sigmapi{p}{i} = \comu_i,
	\end{equation*}
	and for coherent notation we set
	\begin{align*}
		\sigmapi{p}{0} = 0,\\
		\sigmapi{p}{l} = \comu.
	\end{align*}

	This gives a sequence
	\begin{equation*}
		\sigmap{p} =
		\left(
			\sigmapi{p}{0}, \dots, \sigmapi{p}{l}
		\right).
	\end{equation*}

	We also define for all $i$, $1\leq i \leq l$
	\begin{equation*}
		\deltapi{p}{i} = \sigmapi{p}{i} - \sigmapi{p}{i-1}.
	\end{equation*}

	One forms the second sequence from these coweights
	\begin{equation*}
		\deltap{p} =
		\left(
			\deltapi{p}{1}, \dots, \deltapi{p}{l}
		\right).
	\end{equation*}

	Informally speaking, $\deltap{p}$ is the sequence of increments, so $\sigmap{p}$ is the sequence of their partial sums. This explains the notation.
\end{notation}

\begin{figure}
	\centering
	\includegraphics{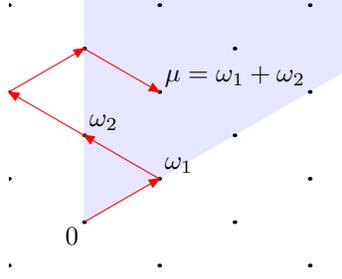}
	\caption{A typical path associated to a fixed point.} \label{FigPathExample}
\end{figure}

One useful way to think about the fixed points of $\Gr^{ \ucolambda }_{\comu}$ is that they are one-to-one certain piecewise-linear paths in the coweight lattice. Let $P_p$ be the piecewise linear path connecting points 
\begin{equation*}
	0 
	= 
	\sigmapi{p}{0}, 
	\sigmapi{p}{1},
	\dots, 
	\sigmapi{p}{l-1}, 
	\sigmapi{p}{l} 
	= 
	\comu
\end{equation*} 
in the coweight lattice.

The paths $P_p$ are the ones that have $l$ segments, start at the origin, and end at ${\comu}$. The $i$th segment must be a weight of $V(\colambda_i)$ (as said earlier, this is equivalent to being a Weyl reflection of $\colambda_i$ if $\colambda_i$ is minuscule). See \refFig{FigPathExample} for an example of a typical path in the case $\G = \PSL{3}$ and $\Gr^{ \left(\coomega_1, \coomega_1, \coomega_1, \coomega_1, \coomega_2\right) }_{\coomega_1+\coomega_2}$.

Using the paths $P_p$ is useful for a combinatorial description of the tangent weights at the fixed points. We postpone the proof until the appendix.

\begin{figure}
	\centering
	\begin{subfigure}{0.45\textwidth}
		\centering
		\includegraphics{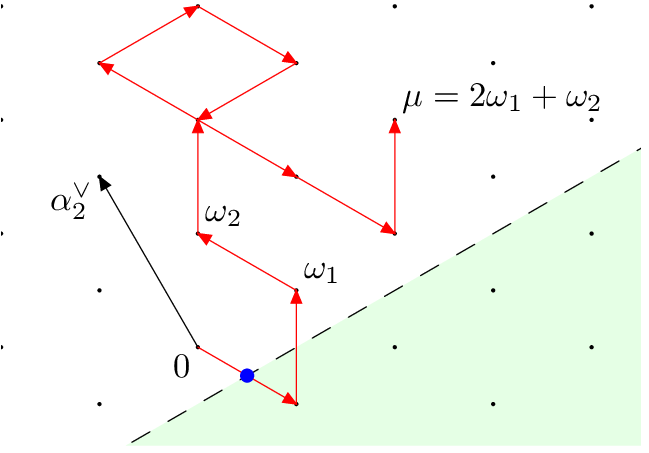}
		\caption{Weight multiplicities of $\wtalpha_2-\hbar$ and $-\wtalpha_2$ are $1$. }\label{FigLatticeMultiplicities1}
	\end{subfigure}
	~
	\begin{subfigure}{0.45\textwidth}
		\centering
		\includegraphics{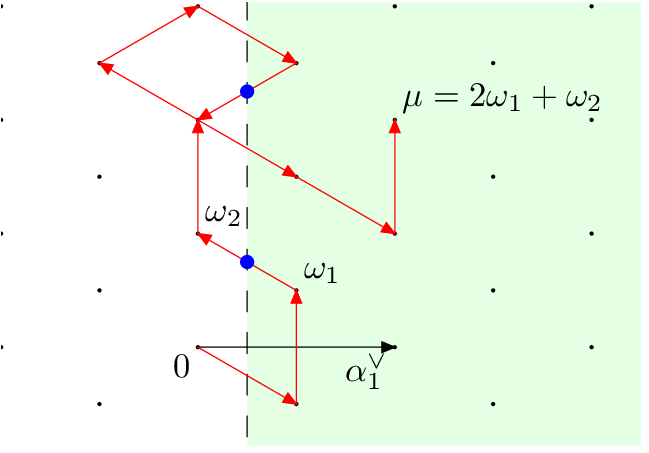}
		\caption{Weight multiplicities of $\wtalpha_1$ and $-\wtalpha_2-\hbar$ are $2$.}\label{FigLatticeMultiplicities2}
	\end{subfigure}
\caption{Weight multiplicities in $T_p X$ for $p$ represented by the red path $P_p$.}\label{FigLatticeMultiplicities}
\end{figure}

\begin{theorem} \label{ThmLatticeMultiplicities}
The only weights in $T_p X$ are of form $\wtalpha+n\hbar$ for a root $\wtalpha$ and $n\in\mathbb{Z}$. Moreover, the multiplicity of weight $\wtalpha+n\hbar$ is the number of crossings of $P_p$ with the hyperplane 
\begin{equation*}
\left\langle \bullet, \wtalpha \right\rangle - \left( n + \dfrac{1}{2} \right) = 0
\end{equation*}
in the direction of the halfspace containing $0$.
\end{theorem}

An illustration of the statement of \refThm{ThmLatticeMultiplicities} is shown on \refFig{FigLatticeMultiplicities}. The dashed lines and green halfphanes correspond to pairs of weights, the blue dots show the intersections under consideration. We used the standard identification of (short) roots $\wtalpha_1$ and $\wtalpha_2$ with vectors of length squared 2 via a properly normalized invariant inner product to visualize them.

The following straightforward corollary is useful in further computations

\begin{corollary} \label{CorAMultiplisities}
	Let $p \in X^{\A}=X^{\T}$. Then
	\begin{enumerate}
	\item
		A weight $\wtalpha$ can have non-zero multiplicity in $T_p X$ only if $\wtalpha$ is a root,
	\item
		If $\wtalpha$ is a positive root, then the multiplicity of $\wtalpha$ in $T_p X$ is the number of such $i$ that
		\begin{equation*}
			\left\langle
				\deltapi{p}{i},
				\wtalpha
			\right\rangle
			= 
			-1.
		\end{equation*}
		Equivalently, this is
		\begin{equation*}
			N -
			\left\langle
				\comu,
				\wtalpha
			\right\rangle,
		\end{equation*}
		where $N$ is the number of such $i$ that
		\begin{equation*}
			\left\langle
				\deltapi{p}{i},
				\wtalpha
			\right\rangle
			= 
			+1.
		\end{equation*}
	\item
		If $\wtalpha$ is a negative root, then the multiplicity of $\wtalpha$ in $T_p X$ is equal to the multiplicity of $-\wtalpha$ in $T_p X$.
	\end{enumerate}
\end{corollary}

\bigskip

The slices have a natural collection line bundles (see \cites{Da})
The line bundle $\OOO{1}$ gives a collection of line bundles on $\Gr^{ \ucolambda }_{\comu}$ via coordinate-wise pullbacks:

\begin{equation*}
	\LBundle{i} = \pi_i^* \OOO{1},
\end{equation*}
where
\begin{align*}
	\pi_i
	\colon 
	\Gr^{ \ucolambda }_{\comu} 
	&\to 
	\Gr, 
	\quad 
	0 \leq i \leq l 
	\\
	(L_0,L_1,\dots,L_l)
	&\mapsto 
	L_i.
\end{align*}

They turn out to be $\T$-equivariant. For our goal it's more relevant to consider the following related collection of line bundles

\begin{equation} \label{EqEBundleDefinition}
	\EBundle{i} = \LBundle{i}/\LBundle{i-1} \text{ for } 1 \leq i \leq l.
\end{equation}

We know that these line bundles are sufficient for our purposes, more precisely the following holds. More precisely, the following statements follow from similar statements for $\LBundle{i}$ in \cites{Da}

\begin{proposition} \label{PropSecondCohomologyGeneration} 
    \leavevmode
    \begin{enumerate}
        \item 
            $\EqCoHlgyk{\T}{X}{2}$, is generated as a vector space by $\ChernE{\T}{i}$ $(1 \leq i < l)$ and constants $\EqCoHlgyk{\T}{\pt}{2}$.
        \item 
            $\CoHlgyk{X}{2}$ is generated as a vector space by  $\ChernE{}{i}$ $(1 \leq i < l)$.
        \item 
            $\Pic{X} \otimes_{\mathbb{Z}} \mathbb{Q}$ is generated as a vector space by $\EBundle{i}$ $(1 \leq i < l)$.
    \end{enumerate}
.
\end{proposition}

So, to study quantum multiplication by elements of $\EqCoHlgyk{\T}{X}{2}$ and the quantum connection, it's enough to study quantum multiplication by $\ChernE{\T}{i}$.

\medskip

We know the action of $\ChernE{\T}{i}$ in the fixed point basis. It's diagonal with eigenvalues given by restrictions to the fixed points. For the purposes of this paper it's enough to look only $\A$-equivariant weights (set $\hbar = 0$ in the formulas from \cites{Da}).

\begin{proposition} \label{PropLWeights}
\leavevmode
    \begin{enumerate}
        \item
            The $\A$-weight of a line bundle $\LBundle{i}$ at a fixed point $p \in X^{\A}$ is 
            \begin{equation*}
    	   \CorootScalar{\sigmapi{p}{i}}{\bullet}
            \end{equation*}
        \item
            The $\A$-weight of a line bundle $\EBundle{i}$ at a fixed point $p \in X^{\A}$ is 
            \begin{equation*}
    	   \CorootScalar{\deltapi{p}{i}}{\bullet}
            \end{equation*}
    \end{enumerate}
\end{proposition}

\subsection{Equivariant cohomology and stable envelopes}

From now on until the appendix the base ring for all cohomology is the field $\mathbb{Q}$, unless stated otherwise. All the results generalize straightforwardly to the case of any field of characteristic $0$.

Computations in equivariant cohomology heavily rely on the well-known localization Theorem\cite{AB}.

\begin{theorem}
(Atiyah-Bott)
The restriction map in the localized equivariant cohomology
\begin{equation*}
	\EqCoHlgy{\T}{X}_{loc} 
	\to 
	\EqCoHlgy{\T}{X^{\T}}_{loc}
\end{equation*}
is an isomorphism.
\end{theorem}

Here and later we use notation $\EqCoHlgy{\T}{X}_{loc} := \EqCoHlgy{\T}{X}\otimes_{\EqCoHlgy{\T}{\pt}} \Frac \EqCoHlgy{\T}{\pt}$.

If $X$ is proper then there is the fundamental class $\left[ X \right] \in \EqHlgyk{\A}{X}{2\dim X}$. By this gives a pushforward to a point homomorphism
\begin{align*}
	\EqCoHlgy{\T}{X}
	&\to
	\EqCoHlgy{\T}{\pt}
	\\
	\gamma
	&\mapsto
	\int\limits_X
	\gamma 
\end{align*}
quite often called integration by analogy with de Rham cohomology. Then the localization theorem allows one to compute this using only restrictions to the fixed locus.

\begin{equation} \label{EqLocalizationFormula}
	\int\limits_X
	\gamma
	=
	\sum_{Z\subset X^{\T}}
	\int\limits_Z
	\dfrac
	{
		\resZ
		{
			\gamma
		}
		{
			Z
		}
	}
	{
		\Euler{\T}{N_{Z/X}}
	}
\end{equation}
where the sum is taken over all components $Z \subset X^{\T}$, $\Euler{\T}{\bullet}$ is the equivariant Euler class, $N_{Z/X}$ is the normal bundle to $Z$ in $X$.

If $X$ is not proper the fundamental class lies in Borel-Moore (equivariant) homology $\left[ X \right] \in \EqBMHlgyk{\T}{X}{2\dim X}$. This pairs naturally with the equivariant cohomology with compact support $\EqCompCoHlgy{\T}{X}$ and gives a pushforward map
\begin{equation*}
	\EqCompCoHlgy{\T}{X}
	\to 
	\EqCoHlgy{\T}{\pt}.
\end{equation*}
In general, we have the following natural maps
\begin{equation} \label{EqCDCohomology}
	\begin{tikzcd}
		\EqCompCoHlgy{\T}{X} \arrow[d,"\iota^*"] \arrow[r]
		&
		\EqCoHlgy{\T}{X} \arrow[d,"\iota^*"]
		\\
		\EqCompCoHlgy{\T}{X^{\T}} \arrow[r]
		&
		\EqCoHlgy{\T}{X^{\T}}
	\end{tikzcd}
\end{equation}
The vertical arrows are restrictions. $X^{\T}$ is closed subvariety, so the inclusion $\iota\colon X^{\A} \to X$ is proper and the pullback is well-defined in the cohomology with compact support). After the localization the vertical arrows in \refEq{EqCDCohomology} become isomorphisms.

Assume now that $X^{\T}$ is proper. Then the bottom arrow in \refEq{EqCDCohomology} becomes an isomorphism because the compact support condition becomes vacuous. Then we get an isomorphism
\begin{equation*}
	\EqCompCoHlgy{\T}{X}
	\xrightarrow{\sim}
	\EqCoHlgy{\T}{X} 
\end{equation*}
which gives the pushforward
\begin{equation*}
	\EqCoHlgy{\T}{X}_{loc}
	\to 
	\EqCoHlgy{\T}{\pt}_{loc}
\end{equation*}
We still denote it by $\int\limits_X \gamma$. The relation \refEq{EqLocalizationFormula} still holds. Some authors use it to define $\int\limits_X \gamma$. As a side affect, a priori, the answer one gets by this definition is in $\EqCoHlgy{\T}{\pt}_{loc} = \Frac \EqCoHlgy{\T}{\pt}$ because of the division by Euler classes. However, we see that the image of non-localized compactly supported classes $\EqCompCoHlgy{\T}{X}$ is inside $\EqCoHlgy{\T}{\pt}$.

This defines on $\EqCoHlgy{\T}{X}_{loc}$ the structure of a Frobenius algebra over $\EqCoHlgy{\T}{pt}_{loc}$. In particular, we often use the natural pairing

\begin{equation*}
	\left\langle
		\gamma_1,
		\gamma_2
	\right\rangle_X
	=
	\int\limits_X
	\gamma_1
	\cup
	\gamma_2.
\end{equation*}

If $X$ is clear we omit it in the notation.

We think of $\EqCoHlgy{\T}{\pt}$ as the ground ring. We usually write the multiplication in $\EqCoHlgy{\T}{\pt}$ by $\cdot$ to simplify notation.

The multiplication in $\EqCoHlgy{\T}{X}$ is still denoted by $\cup$. We also refer to it as \textbf{classical multiplication} as opposed to its deformation, quantum multiplication. 

\subsection{Stable Envelopes}

For details we refer the reader to \cites{MOk}. Under certain choices (namely, chamber $\C$ and a polarization $\epsilon$ on $X^{\A}$) one can define a map going the "wrong way".

\begin{theorem} \label{ThmStableEnvelopesDefinition}
	(Maulik-Okounkov)
	There exists a unique map of $\EqCoHlgy{\T}{pt}$-modules 
		\begin{equation*}
			\StabCE{\C}{\epsilon} 
			\colon 
			\EqCoHlgy{\T}{X^\A} 
			\to 
			\EqCoHlgy{\T}{X}
		\end{equation*}
	such that for any component $Z \subset X^\A$ and $\gamma \in \EqCoHlgy{\T/\A}{Z}$, the stable envelope $\Gamma = \StabCEpt{\C}{\epsilon}{\gamma}$ satisfies
	\begin{enumerate}
	\item 
		$\supp \Gamma \subset \Attr^f_\C \left( Z \right)$,
	\item 
		$\resZ{\Gamma}{Z} = \pm \Euler{\A}{N^{-\C}_{p}}\cup \gamma$, where the sign is fixed by the polarization: $\resA{\resZ{\Gamma}{Z}}{A} = \resPol{\epsilon}{Z} \cup \resA{\gamma}{A}$,
	\item 
		$\deg_\A \resZ{\Gamma}{Z'} < \frac{1}{2}\dim X$ for any $Z' \lC{\C} Z$.
	\end{enumerate}
\end{theorem}

In the spaces of our interest $X^{\A}$ is finite discrete, so $\EqCoHlgy{\T}{X^\A} $ has a basis $\CohUnit{p}$ of classes those only non-zero restriction is $1$ at $p\in X^{\A}$. To simplify notation we write
\begin{equation*}
	\StabCEpt{\C}{\epsilon}{p}
\end{equation*}
instead of
\begin{equation*}
	\StabCEpt{\C}{\epsilon}{\CohUnit{p}}.
\end{equation*}

The classes $\StabCEpt{\C}{\epsilon}{p}$ may be thought of as a refined version of (the Poincar\'{e}-dual class to) the fundamental cycle $\pm\left[ \overline{\Attr_\C (p)} \right]$. This explains the first two conditions in the \refThm{ThmStableEnvelopesDefinition}. The last condition is required to ensure these classes behave well under deformations of the symplectic variety.

For symplectic resolutions, if one restricts $\StabCEpt{\C}{\epsilon}{p}$ to the torus $\A$ (sets $\hbar = 0$), it has a non-zero coefficient only at $p$ \cites{MOk}. Since the fixed point is a basis in the localized cohomology $\EqCoHlgy{\T}{X}_{loc}$, stable envelopes also form a basis in the localized cohomology. This is why we sometimes refer to $\StabCEpt{\C}{\epsilon}{p}$ as the stable basis.

\subsection{Classical multiplication in stable basis}

Let us first discuss in general relation between classical multiplication of a stable envelope by a divisor and restrictions of stable envelopes.

First we define two families of diagonal operators
\begin{align}
	\label{EqHOperDef}
	\Hterm{i}
	\colon
	\EqCoHlgy{\A}{X^{\A}}
	&\to
	\EqCoHlgy{\A}{X^{\A}}
	\\
	\notag
	\CohUnit{p}
	&\mapsto
	\left[
		\CorootScalar{\deltapi{p}{i}}{\bullet}
		+
		\dfrac{\hbar}{2}
		\CorootScalar{\deltapi{p}{i}}{\comu}
	\right]
	\CohUnit{p}
	\\
	\notag
	\\
	\label{EqOmega0OperDef}
	\OmegaOperator{ij}{0}
	\colon
	\EqCoHlgy{\A}{X^{\A}}
	&\to
	\EqCoHlgy{\A}{X^{\A}}\\
	\notag
	\CohUnit{p}
	&\mapsto
	\CorootScalar
	{\deltapi{p}{i}}
	{\deltapi{p}{j}}
	\CohUnit{p}
\end{align}

Then we introduce a family of strictly triangular operators. Let $i,j$ be such that $0\leq i, j \leq l$, $\wtalpha$ be a root and $\epsilon$ be a polarization. We define
\begin{equation}
	\label{EqOmegaAlphaOperDef}
	\OmegaOperator{ij}{-\alpha,\epsilon}
	\colon
	\EqCoHlgy{\A}{X^{\A}}
	\to
	\EqCoHlgy{\A}{X^{\A}}
\end{equation}
by the following property: for any $p \in X^{\A}$
\begin{itemize}
    \item 
        If there exists $q \in X^{\A}$ satisfying conditions 
        \begin{align}
            \deltapi{q}{i} 
            &=
            \deltapi{p}{i} - \coalpha,
            \notag
            \\
            \deltapi{q}{j} 
            &=
            \deltapi{p}{j} + \coalpha,
            \label{EqQIsAdjacentToP}
            \\
            \deltapi{q}{k}
            &=
            \deltapi{p}{k}
            \text{ for all }
            k,\; k \neq i,\; k \neq j,
            \notag
        \end{align}	
        then
        \begin{equation*}
        	\OmegaOperator{ij}{-\alpha,\epsilon} \left( \CohUnit{p} \right)
        	=
        	\sigma^{\epsilon}_{p,q}
        	\dfrac{\CorootScalar{\coalpha}{\coalpha}}{2}
        	\CohUnit{q}.
        \end{equation*}
        The signs $\sigma^{\epsilon}_{p,q} \in \signset$ are the following product of signs
        \begin{equation} \label{EqSignsInOmega}
        	\sigma^{\epsilon}_{p,q}
        	=
        	\sign{\tilde{\C}}{\epsilon}{q}{X}
        	\sign{\tilde{\C}}{\epsilon}{p}{X}
        	=
        	\dfrac
        	{
        		\resPol{\epsilon}{p}
        	}
        	{
        		\Euler{\A}{N^{-\widetilde{\C}}_{p/X}}
        	}
        	\dfrac
        	{
        		\Euler{\A}{N^{-\widetilde{\C}}_{q/X}}
        	}
        	{
        		\resPol{\epsilon}{q}
        	},
        	\end{equation}
        where $\widetilde{\C}$ is any Weyl chamber adjacent to the wall $\wtalpha = 0$. The sign does not depend on this choice as shown in \cites{Da}.
    \item
        If there's no such $q$, we set
        \begin{equation*}
        	\OmegaOperator{ij}{-\alpha,\epsilon}
        	\left( \CohUnit{p} \right)
        	=
        	0.
        \end{equation*}
\end{itemize}

\begin{remark}
	By the conditions on $\deltap{p}$'s from the minuscule weight restriction, the existence for such $q\in X^{\A}$ is equivalent to the following two conditions on $p$
	\begin{align}
	\label{EqExistenceOfQ1}
	\left\langle
		\deltapi{p}{i},
		\wtalpha
	\right\rangle
	&= 1,
	\\
	\label{EqExistenceOfQ2}
	\left\langle
		\deltapi{p}{j},
		\wtalpha
	\right\rangle
	&= -1.
	\end{align}
\end{remark}

Now for any Weyl chamber we define the following operator
\begin{equation}
	\label{EqOmegaCOperDef}
	\OmegaOperator{ij}{\C,\epsilon}
	= 
	\dfrac{1}{2}
	\OmegaOperator{ij}{0}
	+
	\sum\limits_{\alpha \gC{\C} 0}
	\OmegaOperator{ij}{\alpha,\epsilon}.
\end{equation}

\begin{remark}
	The notation $\OmegaOperator{ij}{\C,\epsilon}$ suggests that these operators are related to Casimir operators in representation theory. We'll make this statement more precise later. The minus sign for $-\alpha$ in the definition of $\OmegaOperator{ij}{-\alpha,\epsilon}$ is introduced to make this comparison easier.
\end{remark}

\begin{remark}
	If one chooses polarization $\resPol{\epsilon}{p} = \Euler{\A}{N^{-\C}_{p/X}}$, then the signs in $\OmegaOperator{ij}{\alpha,\epsilon}$ with a simple $\alpha$ (with respect to $\C$) are all $+1$. We will need it later to choose the right basis in representation-theoretic language.
\end{remark}

Similar to stable envelopes, for a simpler notation we write
\begin{equation*}
	\Hterm{i}
	\left( p \right)
	\text{ and }
	\OmegaOperator{ij}{\C,\epsilon}
	\left( p \right)
\end{equation*}
instead of
\begin{equation*}
	\Hterm{i}
	\left( \CohUnit{p} \right)
	\text{ and }
	\OmegaOperator{ij}{\C,\epsilon}
	\left( \CohUnit{p} \right).
\end{equation*}

The main result from \cites{Da} on classical multiplication in the stable basis is the following theorem.

\begin{theorem} \label{ThmReformulatedClassicalMultiplication}
	The classical multiplication is given by the following formula
	\begin{equation*}
		\OpChernE{\T}{i}
		\StabCEpt{\C}{\epsilon}{p}
		=
		\StabCE{\C}{\epsilon}
		\left[
			\Hterm{i}
			\left(
				p
			\right)
			+
			\hbar
			\sum\limits_{j < i}
			\OmegaOperator{ji}{-\C,\epsilon}
			\left(
				p
			\right)
			-
			\hbar
			\sum\limits_{i < j}
			\OmegaOperator{ij}{-\C,\epsilon}
			\left(
				p
			\right)
		\right].
	\end{equation*}
	
\end{theorem}

\section{Quantum Cohomology}
\label{SecQuantum}

In this chapter we compute the multiplication by a divisor in the quantum cohomology. This gives rise to the quantum differential equation, which we match in the next paper with certain representation-theoretic object, called trigonometric Knizhnik-Zamolodchikov connection.

Our computation is based on the idea of the reduction to Picard rank $1$, appearing in \cite{BMO} for the case of the cotangent bundle of the full flag variety. We follow this paper quite closely for the definition of the reduced class and the application of a deformation family to pass to the wall cases.

The first partial result on the purely quantum cohomology was done in M. Viscardi's thesis \cite{V}. Let us point out some issues with this computation
\begin{enumerate}
	\item
		The computation was done only for the simplest case of $\ucolambda = \left( \colambda_1, \colambda_2 \right)$, it was only mentioned that the general case can be done by reduction to Picard rank $1$, but no actual computation,
	\item
		The choice of basis was the fixed point basis, which allowed only to compute the leading part of the purely quantum multiplication. It was completely unprotected from corrections in $\hbar$,
	\item
		The classical part was not computed. Moreover, even if it's done (for example, by a version of \refProp{PropLWeights} with $\T$-weight, see \cites{Da}) in the fixed point basis it must be diagonal, which makes it hard to identify with asymmectic trigonometric Knizhnik-Zamolodchikov connection. As we see from the asymmetry of the final answer this means that the corrections to the purely quantum part must be assymmetric (in particular, the corrections do appear!),
	\item
		The diagonal part was explicitly omitted, but never restored, which lead to appearance of "truncated" Casimir operators, missing exactly the Cartan (diagonal!) part.
\end{enumerate}

We don't have these problems in our computation.

\subsection{Preliminaries on quantum cohomology}

Quantum cohomology is a deformation of ordinary equivariant cohomology by three-point Gromov-Witten invariants (by a formal parameter $q$). We define it by the following formula 
\begin{equation*}
	\left\langle 
		\gamma_1, \gamma_2 * \gamma_3 
	\right\rangle_X
	= 
	\left\langle 
		\gamma_1, \gamma_2 \cup \gamma_3 
	\right\rangle_X
	+
	\sum_
	{
		\dd\in \HlgyEffk{X, \mathbb{Z}}{2}
	}
	\left\langle
		\gamma_1, \gamma_2, \gamma_3
	\right\rangle^X_{0,3,\dd} q^\dd
\end{equation*}
for any $\gamma_1, \gamma_2, \gamma_3 \in \EqCoHlgy{\T}{X}$. Here $\left\langle \bullet, \bullet, \bullet \right\rangle^X_{0,3,\dd}$ are genus zero three-point degree $\dd$ equivariant Gromov-Witten invariants.

The three-point Gromov-Witten invariants are the integrals over a virtual fundamental cycle
\begin{equation*}
	\left\langle
		\gamma_1, \gamma_2, \gamma_3
	\right\rangle^X_{0,3,\dd}
	=
	\int\limits_
	{
		\left[
			\MapModuli{0,3}{X,\dd}
		\right]^{vir}
	}
	\ev_1^*
	\gamma_1
	\cup
	\ev_2^*
	\gamma_2
	\cup
	\ev_3^*
	\gamma_3
\end{equation*}
of the Deligne-Mumford stack of stable genus $0$ degree $\dd$ maps to $X$ with $3$ marked points.

The integration in this formula in the pushforward map
\begin{equation*}
	\EqCoHlgy{\T}{\MapModuli{0,3}{X,\dd}}
	\to
	\EqCoHlgy{\T}{\pt}
\end{equation*}
which is given by the virtual fundamental cycle $\left[ \MapModuli{0,3}{X,\dd} \right]^{vir} \in \EqHlgy{\T}{\MapModuli{0,3}{X,\dd}}$.

For the spaces of our interest $\MapModuli{0,3}{X,\dd}$ are not proper in general and the virtual fundamental class is only in Borel-Moore (eqivariant) homology $\left[ \MapModuli{0,3}{X,\dd} \right]^{vir} \in \EqBMHlgy{\T}{\MapModuli{0,3}{X,\dd}}$. However, let us assume that there is proper $\T$-equivariant map to an affine $\T$-variety $X_0$
\begin{equation*}
	\pi
	\colon
	X
	\to
	X_0
\end{equation*}
such that $X_0$ is contracted to a point $x \in X_0$ by some cocharacter $\colambda \colon \Gm \to \T$. We know that for the spaces $X = \Gr^{\ucolambda}_{\comu}$ it holds.

Under this assumption the $\T$-fixed locus $\MapModuliFixed{\T}{0,3}{X,\dd}$ is proper since all $\T$-invariant mas have the image in $\pi^{-1} \left( x \right)$ which is proper. Then the integration (or pushforward to a point)
\begin{equation*}
	\EqCoHlgy{\T}{\MapModuli{0,3}{X,\dd}}_{loc}
	\to
	\EqCoHlgy{\T}{\pt}_{loc}
\end{equation*}
can be defined as mentioned in the previous chapter and computed via the localization formula.

\medskip

Let us denote by $\QuantumPowerSeries$ the algebra of formal power series in $q^\dd$, where $\dd \in \HlgyEffk{X, \mathbb{Z}}{2}$ are the effective curve classes.

The quantum multiplication gives the quantum equivariant cohomology
\begin{equation*}
	\QCoHlgy{\T}{X} 
	= 
	\EqCoHlgy{\T}{X} 
	\otimes_\mathbb{Q}
	\QuantumPowerSeries
\end{equation*}
a structure of a supercommutative associative algebra. This is a deformation of $\EqCoHlgy{\T}{X}$.

We will split the quantum multiplication into classical part and purely quantum part
\begin{equation}
	\gamma_1 * \gamma_2 
	= 
	\gamma_1 \cup \gamma_2
	+
	\gamma_1 \qdot \gamma_2.
\end{equation}

We are interested in the case when $\gamma_3$ is a divisor, that is $\gamma_3 = D \in \EqCoHlgyk{\A}{X}{2}$. Then we can apply the divisor equation to reduce three-point invariants to two-point invariants
\begin{equation} \label{EqDivisorEquation}
	\left\langle
		\gamma_1, \gamma_2, D
	\right\rangle^X_{0,3,\dd}
	=
	\left\langle 
		D,
		\dd
	\right\rangle
	\left\langle
		\gamma_1, \gamma_2
	\right\rangle^X_{0,2,\dd},
\end{equation}
where $\left\langle D, \dd \right\rangle$ is the pairing of a homology and cohomology class under the natural map $\EqCoHlgy{\T}{X} \to \CoHlgy{X}$.

\subsection{Reduction to Picard rank 1}

\subsubsection{Two approaches to the reduced virtual class}

One can consider non-equivaraint counts instead of $\T$-equivariant ones. There are two related reasons why these counts must vanish for a symplectic resolution $X$ and $\dd \neq 0$. Let us briefly outline them.

The first one comes from a surjective cosection in the obstruction bundle. See \cite{KL} for a more general notion of cosection localization.

The cosection appears as follows. We work with the tangent-obstruction theory relative to the prestable curves, so let us show how the construction work on a fixed prestable curve $C$. Let $\MapModuliC{C}{X,\dd}$ be the moduli space of stable maps from $C$ to $X$ with the degree class $\dd\neq 0$. The standard relative tangent-obstruction theory for $\MapModuliC{C}{X,\dd}$ is defined by the natural morphism
\begin{equation*}
	R\pi_* 
	\left(
		\ev^* TX
	\right)^\vee 
	\to 
	L{\MapModuliC{C}{X,\dd}},
\end{equation*}
where
\begin{align*}
	\ev
	&\colon
	C \times \MapModuliC{C}{X,\dd}
	\to
	X,
	\\
	\pi
	&\colon
	C \times \MapModuliC{C}{X,\dd}
	\to
	\MapModuliC{C}{X,\dd}
\end{align*}
are the natural maps and $L{\MapModuliC{C}{X,\dd}}$ is the cotangent complex of $\MapModuliC{C}{X,\dd}$.

Let $\omega_\pi$ be the relative dualizing sheaf. Then the pairing with the symplectic form $\omega$ and the pullback of forms induces a map
\begin{equation*}
	\ev^*
	\left(
		TX
	\right)
	\to
	\omega_\pi
	\otimes
	\left(
		\mathbb{C} \omega
	\right)^*.
\end{equation*}
This gives a map of complexes
\begin{equation*}
	R\pi_*
	\left(
		\omega_\pi
	\right)^\vee
	\otimes
	\mathbb{C}
	\omega
	\to
	R\pi_*
	\left(
		\ev^*
		\left(
			TX
		\right)^\vee
	\right).
\end{equation*}
Then we truncate
\begin{equation*}
	\tau_{\leq -1}
	R\pi_*
	\left(
		\omega_\pi
	\right)^\vee
	\otimes
	\mathbb{C}
	\omega
	\to
	R\pi_*
	\left(
		\ev^*
		\left(
			TX
		\right)^\vee
	\right).
\end{equation*}
This truncation is a trivial bundle. It carries non-trivial equivariant structure as long as the torus acts non-trivially on $\omega$. This map gives a surjective cosection in the obstructions.

The second reason to vanish is the deformations of $X$. On one hand, the virtual fundamental class is deforamtion invariant. On the other hand, we know that $X$ has a large family of universal deformations and, moreover, a generic deformation has \textbf{no} projective curves. This means that $\left[ \MapModuli{0,3}{X',\dd} \right]^{vir} = 0$ at a generic symplectic deformation $X'$ of $X$ for $\dd\neq 0$. Combining these two facts, we get that for $\dd\neq 0$ the virtual fundamental class $\left[ \MapModuli{0,3}{X,\dd} \right]^{vir} = 0$ for $X$ itself.
 
\begin{remark}
	Same vanishings happen in $\A$-eqivariant count instead of $\T$-equivariant one. The symplectic form has the zero $\A$-weight, so we would still have a non-vanishing cosection of weight zero in the obstruction bundle. Similarly, in the universal deformation the $\A$-action on $X$ extends to the fiberwise $\A$-action in the universal deformation family.
\end{remark}

\begin{remark}
	Now we see how the $\T$-equivariant counts destroy the vanishing constructions. In the first approach the cosection gets a non-zero weight $-\hbar$. The second approach doesn't work because $\T$  scales the base of the deformation with weight $\hbar$ and the zero fiber becomes distinct.
\end{remark}

Each of these reasons suggest how one may modify the counting problem to avoid such vanishing. More precisely, how to track the part which vanishes when one passes to the non-equivariant case, and define the remaining part, which is called the \textbf{reduced fundamental class}. 

First, let us show how one modifies the tangent-obstruction bundle. Above we showed that there is a map
\begin{equation*}
	\iota
	\colon
	\tau_{\leq -1}
	R\pi_*
	\left(
		\omega_\pi
	\right)^\vee
	\otimes
	\mathbb{C}
	\omega
	\to
	R\pi_*
	\left(
		\ev^*
		\left(
			TX
		\right)^\vee
	\right).
\end{equation*}
which is responsible for the appearance of a surjective cosection in the obstruction bundle. Now we can "kill" it by considering the mapping cone of $\iota$, $C(\iota)$. Then the induced map
\begin{equation*}
	C(\iota)
	\to
	{\MapModuliC{C}{X,\dd}}
\end{equation*}
give a perfect obstruction theory. This is the reduced obstruction theory, and the associated fundamental class is called the reduced virtual fundamental class.

The second approach comes from one-parametric families of deformations. Let $\widetilde{X} \to \mathcal{B}$ be a symplectic deformation of $\left( X, \omega \right)$ over a smooth curve $\mathcal{B}$. Let us denote the point in the base over which $\left( X, \omega \right)$ lies as $b\in \mathcal{B}$. Pick $\dd \in \Hlgyk{X,\mathbb{Z}}{2} \subset \Hlgyk{\widetilde{X},\mathbb{Z}}{2}$. We assume that the deformation $\widetilde{X}$ and are in general position, that is
\begin{enumerate}
	\item
		The only fiber of $\widetilde{X} \to \mathcal{B}$ where there are curves of degree $d$ is $\left( X, \omega \right)$ over $b$.
	\item
		The following composition is an isomorphism
		\begin{equation*}
			T_b \mathcal{B}
			\to
			\CoHlgyk{X,TX}{2}
			\xrightarrow{\sim}
			\CoHlgyk{X,T^*X}{2}
			\xrightarrow{d\cap}
			\mathbb{C}
		\end{equation*}
		where the first map is the Kodaira-Spencer map, the second map is induced by the symplectic structure, the last map is the pairing with the homology class $\dd$.
\end{enumerate}

\begin{remark}
	If $X$ is a symplectic resolution and we have the universal family of the deformations, then the first condition says that the image of $\mathcal{B}$ intersects exactly once the hyperplane (or intersection of hyperplanes) associated to $\dd$. The second condition means that this intersection is transversal.
\end{remark}

The following result is the Theorem 1 of \cite{MP}.

\begin{proposition} \label{PropRedClassViaFamily}
	In the setting above, the embedding
	\begin{equation*}
		\MapModuli{0,n}{X,\dd} 
		\to 
		\MapModuli{0,n}{\widetilde{X},\dd}
	\end{equation*}
	is an isomorphism of stacks. Under this isomorphism the virtual classes
	\begin{equation*}
		\left[
			\MapModuli{0,n}{X,\dd}
		\right]^{red}
		=
		\left[
			\MapModuli{0,n}{\widetilde{X},\dd}
		\right]^{vir}
	\end{equation*}
	where the right-hand side is the usual virtual fundamental class associated to $\widetilde{X}$.
\end{proposition}

From both points we see that the virtual dimension of the reduced class
\begin{equation*}
	\left[ 
		\MapModuli{0,n}{X,\dd} 
	\right]^{red}
\end{equation*}
is the virtual dimension of
\begin{equation*}
	\left[ 
		\MapModuli{0,n}{X,\dd} 
	\right]^{vir}
\end{equation*}
plus $1$.

\subsubsection{Choice of basis}

Let $\left( X, \omega \right)$ be a symplectic variety with a torus $\T$ acting on it in such a way that the symplectic form $\omega$ is scaled by $\T$-weight $\hbar$.

Then from the explicit description of the reduced tangent-obstruction theory we see that it misses one trivial line subbundle in the obstructions with weight $-\hbar$. This gives
\begin{equation*}
	\left[
		\MapModuli{0,n}{X,\dd}
	\right]^{vir}
	=
	-\hbar
	\left[
		\MapModuli{0,n}{X,\dd}
	\right]^{red}
\end{equation*}
in $\EqBMHlgy{\T}{\MapModuli{0,n}{X,\dd}}$, $\dd \neq 0$.

This allows us to write for $\dd \neq 0$
\begin{equation*}
	\left\langle
		\gamma_1, \gamma_2
	\right\rangle^{X,vir}_{0,2,\dd}
	=
	-\hbar
	\left\langle
		\gamma_1, \gamma_2
	\right\rangle^{X,red}_{0,2,\dd}
\end{equation*}
where the tags $vir$ and $red$ denote the fundamental class we use. In order to recover the usual Gromov-Witten invariant we have to use the non-reduced class in \refEq{EqDivisorEquation}. However the reduced class allows us to apply the trick from \cite{BMO}.

In our case $X$ is a symplectic variety, so it has trivial canonical class. Then the virtual dimension of $\MapModuli{0,2}{X,\dd}$ is 
\begin{equation*}
	-
	\left\langle
		K_X,
		\dd
	\right\rangle
	+
	\dim X
	+ n
	- 3
	=
	\dim X - 1.
\end{equation*}
Thus the degree of $\left[ \MapModuli{0,n}{X,\dd} \right]^{red}$ is $2\dim X - 2$ and the degree of $\left[ \MapModuli{0,n}{X,\dd} \right]^{red}$ is $2\dim X$.

Let $\gamma_1, \gamma_2 \in \EqCoHlgy{\T}{X}$ be non-localized classes. If $\gamma_1 \cup \gamma_2$ have compact support, then the support of $\ev_1^* \gamma_1 \cup \ev_2^* \gamma_2$ is also proper and the $2$-point reduced Gromov-Witten invariant
\begin{equation*}
	\left\langle
		\gamma_1, \gamma_2
	\right\rangle^{X,red}_{0,2,\dd}
	\in
	\EqCoHlgy{\T}{\pt}
\end{equation*}
lies in non-localized cohomology of a point. We get that for such classes the invariants are $\left\langle \gamma_1, \gamma_2 \right\rangle^X_{0,2,\dd}$ are non-zero by dimension argument only if we have
\begin{equation*}
	\deg \gamma_1
	+ 
	\deg \gamma_2
	\geq
	2 \dim X.
\end{equation*}

Possibility to get invariants in non-localized $\EqCoHlgy{\T}{\pt}$ shows why it is important to work with non-localized classes. We have finitely many terms to compute, it makes the problem more doable.

At the same time it shows disadvantages in using the fixed point basis in equivariant cohomology. By definition, it lies in the localized cohomology. If one replaces it via pushforwards $\iota_{p,*} \CohUnit{}$ of the units from the fixed points it is better, because these classes are non-localized. In this case we get that
\begin{equation*}
	\left\langle
		\iota_{q,*} \CohUnit{}, \iota_{p,*} \CohUnit{}
	\right\rangle^{X,red}_{0,2,\dd}
	\in
	\EqCoHlgyk{\T}{\pt}{2\dim X}
\end{equation*}
which means that we have to compute a polynomial of degree $\dim X$. There is a better choice.

Assume that $X^{\T}$ is finite and the conditions for the existence of the stable envelopes are satisfied. Then we can take 
\begin{align*}
	\gamma_1
	&= 
	\StabCEpt{-\C}{\DualPol{\epsilon}}{q},
	\\
	\gamma_2
	&= 
	\StabCEpt{\C}{\epsilon}{p}.
\end{align*}
As we said in the previous chapter, these classes are dual, so computing these $2$-point invariants exactly gives us the matrix of the purely quantum multiplication in the stable basis.

Recall that the degree of $\StabCEpt{\C}{\epsilon}{p}$ is $\dim X$. Then
\begin{equation*}
	\left\langle
		\StabCEpt{-\C}{\DualPol{\epsilon}}{q},
		\StabCEpt{\C}{\epsilon}{p}
	\right\rangle^{X,red}_{0,2,\dd}
	\in
	\EqCoHlgyk{\T}{\pt}{0}
	\simeq
	\mathbb{Q},
\end{equation*}
so the invariants are just numbers.

In particular, if $X^{\T} = X^{\A}$ these number can be computed via $\A$-equivariant localization instead of $\T$-equivariant one (as usual, we set $\A \subset \ker \hbar$ to be a subtorus acting by symplectomorphisms).

For a symplectic resolution $X$ this is convenient, since the stable envelope $\StabCEpt{\C}{\epsilon}{p}$ restricted to $\A$ vanish everywhere except the point $p$. Then we get

\begin{equation*}
	\left\langle
		\StabCEpt{-\C}{\DualPol{\epsilon}}{q},
		\StabCEpt{\C}{\epsilon}{p}
	\right\rangle^{X,red}_{0,2,\dd}
	=
	\resPol{\DualPol{\epsilon}}{q}
	\resPol{\epsilon}{p}
	\int\limits_
	{
		\left[
			\MapModuliFixed{\A}{0,(p,q)}{X,\dd}
		\right]^{red}
	}
	1,
\end{equation*}
where the integral is defined via $\A$-equivariant localization and $\MapModuliFixed{\A}{0,(p,q)}{X,\dd}$ is a subspace of the $\A$-fixed locus
\begin{equation*}
	\left(
		\MapModuli{0,2}{X,\dd}
	\right)^{\A}
	=
	\bigsqcup\limits_{p,q\in X^{\A}}
	\MapModuliFixed{\A}{0,(p,q)}{X,\dd}
\end{equation*}
which maps the first marked point to $p$ and the second marked point to $q$.

\subsubsection{Reduction to walls}

We reduced the computation of the quantum multiplication to the computation of $\A$-localized integrals
\begin{equation*}
	\int\limits_
	{
		\left[
			\MapModuliFixed{\A}{0,(p,q)}{X,\dd}
		\right]^{red}
	}
	1.
\end{equation*}

The space $\A$-fixed maps to $X$ has components of high dimension, sice we saw that even $\T$-invariant curves are not isolated in $X$. In principle, it is possible to compute the integral via deformation to the case which has only isolated $\A$-invariant projective curves. This appoach was used in \cite{BMO} for one of the most canonical example of the sypmlectic resolution, the Springer resolution.

Let $X$ be a symplectic resolution and let $\widetilde{X}$ be its sufficiently generic twistor deformation over $\AffSpace^2$ by a pair of line bundles:
\begin{equation*}
	\begin{tikzcd}
		X \arrow[r,hook] \arrow[d]
		&
		\widetilde{X} \arrow[d]
		\\
		0 \arrow[r,hook]
		&
		\AffSpace^2
	\end{tikzcd}.
\end{equation*}
By "sufficiently generic" we mean that 
\begin{enumerate}
	\item
		The generic fiber of $\widetilde{X} \to \AffSpace^2$ has no curves. By the discussion of the wall structure in twistor deformations, this means the fibers with curves lie over a union of finitely many lines $H \subset \AffSpace^2$ passing through the origin $0 \in \AffSpace^2$.
	\item
		If $H \subset \AffSpace^2$ is such a line passing through $0$, that there are curves over $H\setminus 0$, then curves there is $d\in \Hlgyk{X,\mathbb{Q}}{2}$ such that the degrees of all such curves are multiples of $\dd$. We consider $\dd$ up to scaling. Then we denote such a line by $H_\dd$ and refer to them as "walls" in $\AffSpace^2$.
\end{enumerate}

The main example of such a deformation is the family over a plane $\AffSpace^2 \simeq P \subset \mathcal{B}^{u}$ in the base of the universal symplectic deformation of $X$
\begin{equation*}
	\begin{tikzcd}
		X \arrow[r,hook] \arrow[d]
		&
		\widetilde{X} \arrow[r,hook] \arrow[d]
		&
		\widetilde{X}^{u} \arrow[d]
		\\
		0 \arrow[r,hook]
		&
		P \arrow[r,hook]
		&
		\mathcal{B}^{u}
	\end{tikzcd}
\end{equation*}
as long as $P$ intersects transversely all the walls. The walls in $P$ are exactly the intersections with the walls in $\mathcal{B}^{u}$.

\medskip

Let $L_0 \subset \AffSpace^2$ be any line passing through the origin $0 \in \AffSpace^2$ which is not a wall. We get a family $\widetilde{X}_0$ over $L_0$
\begin{equation*}
	\begin{tikzcd}
		X \arrow[r,hook] \arrow[d]
		&
		\widetilde{X}_0 \arrow[r,hook] \arrow[d]
		&
		\widetilde{X} \arrow[d]
		\\
		0 \arrow[r,hook]
		&
		L_0 \arrow[r,hook]
		&
		\AffSpace^2
	\end{tikzcd}
\end{equation*}

Then by the \refProp{PropRedClassViaFamily} the integral over the reduced virtual class can be computed via $\widetilde{X}_0$
\begin{equation*}
	\int\limits_
	{
		\left[
			\MapModuliFixed{\A}{0,(p,q)}{X,\dd}
		\right]^{red}
	}
	1
	=
	\int\limits_
	{
		\left[
			\MapModuliFixed{\A}{0,(\widetilde{p}_0,\widetilde{q}_0)}{\widetilde{X}_0,\dd}
		\right]^{vir}
	}
	1,
\end{equation*}
where $\widetilde{p}_0$ and $\widetilde{q}_0$ are the components of $\widetilde{X}_0^{\A}$ containing $p$ and $q$ respectively. They have exactly one point in the fiber over every point $\ALine$, they show how the $\A$-fixed points $p$, $q$ vary in the family $\widetilde{X}_0^{\A}$.

Now we can use the deformation invariance of the virtual fundamental class 
\begin{equation*}
	\left[ \MapModuliFixed{\A}{0,(\widetilde{p}_0,\widetilde{q}_0)}{\widetilde{X}_0,\dd} \right]^{vir}.
\end{equation*}
Parallel shifts along $L_0$ give us a quotient
\begin{equation*}
	\AffSpace^2
	\twoheadrightarrow
	\ALine
\end{equation*}
whose fiber over $0$ is $L_0$ and over another point $t\neq 0$ is a parallel line $L_t$. Let $\widetilde{X}_t$ be the family $\widetilde{X}$ restricted to $L_t$. Then the composition
\begin{equation*}
 	\widetilde{X}
 	\to
 	\AffSpace^2
	\twoheadrightarrow
	\ALine
\end{equation*}
gives a family whose fiber over $0$ is $\widetilde{X}_0$
\begin{equation*}
	\begin{tikzcd}
		\widetilde{X}_0 \arrow[r,hook] \arrow[d]
		&
		\widetilde{X} \arrow[d]
		&
		\widetilde{X}_t \arrow[l,hook'] \arrow[d]
		\\
		0 \arrow[r,hook]
		&
		\ALine
		&
		t \arrow[l,hook']
	\end{tikzcd}
\end{equation*}
In other words, it's a deformation of $\widetilde{X}_0$. The fixed components $\widetilde{p}_0$ deform to $\widetilde{p}_t \subset \widetilde{X}_t^{\A}$. By the deformation invariance of the virtual class
\begin{equation*}
	\int\limits_
	{
		\left[
			\MapModuliFixed{\A}{0,(\widetilde{p}_0,\widetilde{q}_0)}{\widetilde{X}_0,\dd}
		\right]^{vir}
	}
	1
	=
	\int\limits_
	{
		\left[
			\MapModuliFixed{\A}{0,(\widetilde{p}_t,\widetilde{q}_t)}{\widetilde{X}_t,\dd}
		\right]^{vir}
	}
	1.
\end{equation*}
First let us point that some effective curve classes $\dd$ are repserented only by curves over the origin $0\in \AffSpace^2$. Then the moduli space on the RHS is empty and the integral is zero.

From now on we assume that effective $\dd$ has representatives not only over the origin. By construction, this means that there is a wall $H_{\dd}$ which has curves of degree $\dd$ in the fibers over it. By construction $L_t$ intersects with the hyperplane $H_{\dd}$ at exactly one point and this intersection is transverse. Call the fiber of $\widetilde{X}_t$ over this point $X^{\dd}$ and denote by $p',q' \in \left( X^{\dd} \right)^{\A}$ the fibers of $\widetilde{p}_t$ and $\widetilde{q}_t$. Then by \refProp{PropRedClassViaFamily} we have

\begin{equation*}
	\int\limits_
	{
		\left[
			\MapModuliFixed{\A}{0,(\widetilde{p}_t,\widetilde{q}_t)}{\widetilde{X}_t,\dd}
		\right]^{vir}
	}
	1
	=
	\int\limits_
	{
		\left[
			\MapModuliFixed{\A}{0,(p',q')}{X^{\dd},\dd}
		\right]^{red}
	}
	1.
\end{equation*}
Combining all the equalities of integrals, we get that the integral of interest
\begin{equation}
	\int\limits_
	{
		\left[
			\MapModuliFixed{\A}{0,(p,q)}{X,\dd}
		\right]^{red}
	}
	1
	=
	\int\limits_
	{
		\left[
			\MapModuliFixed{\A}{0,(p',q')}{X^{\dd},\dd}
		\right]^{red}
	}
	1
\end{equation}
if there is a wall $H_\dd$, and zero otherwise. 

This formula relates a count of curves of degree $\dd$ with a count of curves in a generic fiber $X^{\dd}$ over the corresponding wall $H_\dd$.

The invariant curves in $X^{\dd}$ are much simpler as we will see in examples. For the spaces of our interest the moduli space $\MapModuliFixed{\A}{0,(p',q')}{X^{\dd},\dd}$ has only isolated points.

We summarize the reduction in the following statement

\begin{proposition} \label{PropRank1Reduction}
	Let $X$ be a symplectic resolution satisfying the assumptions in the definition of the stable envelopes and let $X^{\A}$ be finite. If for an effective curve class $\dd$ there is a wall $H_{\dd}$ in a sufficiently generic $2$-parametric deformation 
	\begin{equation*}
		\left\langle
			\StabCEpt{-\C}{\DualPol{\epsilon}}{q},
			\StabCEpt{\C}{\epsilon}{p}
		\right\rangle^{X,vir}_{0,2,\dd}
		=
		-\hbar
		\resPol{\DualPol{\epsilon}}{q}
		\resPol{\epsilon}{p}
		\int\limits_
		{
			\left[
				\MapModuliFixed{\A}{0,(p',q')}{X^\dd,\dd}
			\right]^{red}
		}
		1
	\end{equation*}
	for any $p,q \in X^{\A}$ and the corresponding $\A$-fixed points $p'$, $q'$ in a generic fiber $X^{\dd}$ over the wall $H_{\dd}$.
	
	If there is no wall of form $H_{\dd}$, then
	\begin{equation*}
		\left\langle
			\StabCEpt{-\C}{\DualPol{\epsilon}}{q},
			\StabCEpt{\C}{\epsilon}{p}
		\right\rangle^{X,red}_{0,2,\dd}
		=
		0.
	\end{equation*}
\end{proposition}

\begin{remark}
	Note that it is important for us that we are doing $\A$-localized computation, not $\T$-equivariant one. The torus $\T$ would scale the base and make the origin a distinguished point, the deformation invariance would be destroyed.
\end{remark}

\subsection{Computation of wall contributions}

\subsubsection{Wall deformations}

Now we restrict to the case $X = \Gr^{\ucolambda}_{\comu}$. Then the $2$-parameter deformation can be made by restricting the Beilinson-Drinfeld Grassmannian to a generic plane $P \simeq \AffSpace^2$ containing the origin in the base $\AffSpace^l$

\begin{equation*}
	\begin{tikzcd}
		\Gr^{\ucolambda}_{\comu} \arrow[r,hook] \arrow[d]
		&
		\widetilde{X} \arrow[r,hook] \arrow[d]
		&
		\GrBD^{\ucolambda}_{\comu} \arrow[d]
		\\
		0 \arrow[r,hook]
		&
		P
		\arrow[r,hook]
		&
		\AffSpace^l
	\end{tikzcd}
\end{equation*}

The walls in $P$ are the intesections of $P$ with the walls in $\AffSpace^l$, that is $H_{ij} = \left\lbrace t_i=t_j \right\rbrace$. As mentioned before, all fibers in the complement of these walls have no projective curves.

Now let us compute the contribution from a generic fiber over each wall.

Let $\left( t_1,\dots, t_l \right) \in H^{ij} \subset \AffSpace^l$ be a generic point in the wall $H_{ij}$, that is $t_i=t_j$ ($i<j$) and all other $t_k$ are pairwise distinct. Then we know that the fiber of the global convolution Grassmannian splits over this point in the following way
\begin{equation*}
	\GrBD^{\ucolambda}_{\left(t_1,\dots,t_l\right)}
	\simeq
	\Gr^{\colambda_i,\colambda_j}
	\times
	\prod_
	{
		\begin{smallmatrix}
			k\neq i
			\\
			k \neq j
		\end{smallmatrix}
	}
	\Gr^{\colambda_k}.
\end{equation*}

The restriction of the global line bundles $\EBundleTwisted{i}$, $\EBundleTwisted{j}$ are the pull-backs of $\EBundle{1}$ and $\EBundle{2}$ respectively from $\Gr^{\colambda_i,\colambda_j}$. For other $k$, $\EBundleTwisted{k}$ restricts to a pullback of $\EBundle{1} \simeq \OOO{1}$ from $\Gr^{\colambda_k}$. These are isomorphisms as $\A$-equivariant line bundles.

We are interested in the fiber of the global subspace $\GrBD^{\ucolambda}_{\comu, \left( t_1,\dots,t_l \right)} \subset \GrBD^{\ucolambda}_{\left(t_1,\dots,t_l\right)}$. We denote this fiber as $X^{ij}$:
\begin{equation*}
	X^{ij}
	\subset 
	\Gr^{\colambda_i,\colambda_j}
	\times
	\prod_
	{
		\begin{smallmatrix}
			k\neq i
			\\
			k \neq j
		\end{smallmatrix}
	}
	\Gr^{\colambda_k}.
\end{equation*}

From the period map we find that any algebraic curve $C$ in $X^{ij}$ must have fundamental class of form $c(e_i-e_j)$. This gives that the degree with respect to the restrictions of all line bundles $\EBundleTwisted{k}$ is zero unless $k=i$ or $k=j$. Since $\OOO{1}$ is ample on $\Gr^{\colambda}$ and generate the Picard group, this implies that the projection of $C$ to each factor of $\Gr^{\colambda_k}$ is constant. Moreover, the degrees with respect to $\EBundleTwisted{i}$ and $\EBundleTwisted{j}$ must be opposite. That means that the degree with respect to $\LBundle{2}$ on $\Gr^{\colambda_i,\colambda_j}$ is zero (as the restricition of $\EBundleTwisted{i}\EBundleTwisted{j}$), which similarly forces the projection to the last coordinate of $\Gr^{\colambda_i,\colambda_j}$ to be constant.

Now we compute the integrals over the moduli space of stable maps $\MapModuli{0,2}{X^{ij},\dd}$ via $\A$-equivariant localization.

The $\A$-fixed locus has a description, similar to the $\A$-fixed locus of $\Gr^{\ucolambda}_{\comu}$ in \refProp{PropSliceFixedLocus}:

\begin{proposition} \label{PropWallSliceFixedLocus}
	The fixed locus $\left( X^{ij} \right)^\A$ is the following disjoint union of finitely many points
	\begin{equation*}
			\left(
				\left(
					\Tfixed{\conu_i}, 
					\Tfixed{\conu_j}
				\right), 
				\Tfixed{\conu_1}, 
				\dots,
				\widehat{\Tfixed{\conu_i}},
				\dots,
				\widehat{\Tfixed{\conu_j}}, 
				\dots,
				\Tfixed{\conu_l}
			\right) 
	\end{equation*}
	where the hats denote omitted terms, $\left( \Tfixed{\conu_i}, \Tfixed{\conu_j} \right) \in \left( \Gr^{\colambda_i,\colambda_j} \right)^{\A}$ and $\Tfixed{\conu_k} \in \left( \Gr^{\colambda_k} \right)^{\A}$ for all $k$, $k \neq i$, $k \neq j$, and, finally, 
	\begin{equation*}
		\sum_{k\neq i} 
		\conu_k
		=
		\comu.
	\end{equation*}
\end{proposition}
\begin{proof}
	Omitted.
\end{proof}

The identification of the line bundles also allow us to compare the fixed points of $X^{ij}$ and the central fiber $X$, that is, to say which $\A$-fixed points go to which fixed points of $X$ as we approach the central fiber along the hyperplane.

For any $\T$-equivariant line bundle $\widetilde{\LL}$ on $\widetilde{X}$ the $\A$-weights on a point $p' \in \left( X^{ij} \right)^{\A}$ and its limit $p \in X ^{\A}$ must be the same. This is because $\A$-action commutes with $\LoopGm$ and the limit is taken by the $\LoopGm$. The $\A$-weights of $\EBundle{i}$ uniquely determine $\deltai{i}$ of the fixed point. 

Then the limit of 
\begin{equation*}
	p'
	=
	\left(
		\left(
			\Tfixed{\conu_i}, 
			\Tfixed{\conu_j}
		\right), 
		\Tfixed{\conu_1}, 
		\dots,
		\widehat{\Tfixed{\conu_i}},
		\dots,
		\widehat{\Tfixed{\conu_j}}, 
		\dots,
		\Tfixed{\conu_l}
	\right) 
\end{equation*}
is $p$ with
\begin{align}
	\deltapi{p}{k}
	&=
	\conu_k
	\text{ if }
	k\neq j
	\label{EqLimitPoints1}
	\\
	\deltapi{p}{j}
	&=
	\conu_j
	-
	\conu_i.
	\label{EqLimitPoints2}
\end{align}

\medskip

Similarly to the case of the slices one finds that $\A$-invariant curves in $\Gr^{\colambda_i,\colambda_j}$ are of form $C^{12}_{p,\wtalpha}$, where
\begin{equation*}
	p
	=
	\left( \Tfixed{\conu_i}, \Tfixed{\conu_j} \right)
	\in 
	\left( \Gr^{\colambda_i,\colambda_j} \right)^{\A}
\end{equation*}
and
\begin{align}
	\left\langle
		\conu_i,
		\wtalpha
	\right\rangle
	&> 0,
	\label{EqTangentWeightInHalfspace}
	\\
	\left\langle
		\conu_j,
		\wtalpha
	\right\rangle
	&= 0.
	\label{EqRightEndFixed}
\end{align}

Recall that $C^{12}_{p,\wtalpha}$ has two fixed points: $p$ and
\begin{equation*}
	q
	=
	\left( \Tfixed{\conu_i-\coalpha}, \Tfixed{\conu_j} \right)
\end{equation*}
and the tangent weight of $C^{12}_{p,\wtalpha}$ at $p$ and $q$ are $\wtalpha$ and $-\wtalpha$ respectively.

These curves are isolated. 

\subsubsection{Unbroken maps}

Now we need to discuss which stable $\A$-invariant maps give zero contributions. This needs the notion of broken maps by Okounkov and Pandharipande \cite{OP}. The classification of the unbroken maps is identical to the one in the case of the cotangent bundles to a flag variety, see \cite{Su}.

Recall that $X^{ij}$ is symplectic because it comes from a twistor deformation. We reduced the tangent-obstuction theory on the moduli stack $\MapModuli{0,2}{X^{ij},d}$ to make sure that the obstruction bundle doesn't have a non-vanishing cosection with zero $\A$-weight. However, there are two cases than an extra such cosection appears and has no way to cancel with deformation.

\medskip

Let $Z$ be an algebraic symplectic variety. Consider an $\A$-fixed stable map
\begin{equation*}
	\left[
		f\colon C \to Z
	\right]
	\in
	\left(
		\MapModuli{0,n}{Z,\dd}
	\right)^{\A}
\end{equation*}
\begin{itemize}
\item
	We say that $f$ has a \textbf{breaking subcurve} $\widetilde{C} \subset C$ if $\widetilde{C}$ is contracted by $f$ and the disconnected curve $C \setminus \widetilde{C}$ has at least two connected components which have positive degree under $f$.
\item
	We say that $f$ has a \textbf{breaking node} $s \in C$, if two non-contracted components $C_1,C_2 \subset C$ meet at $s$ and have tangent weights $a_1, a_2$ with $a_1+a_2 \neq 0$.
\end{itemize}

We call a stable $\A$-invariant map $f\colon C \to Z$ \textbf{broken}, if there's at least one broken subcurve or broken node. If a map $f$ is not broken, we call it \textbf{unbroken}. By combinatorial description of the components of $\left( \MapModuli{0,n}{Z,\dd} \right)^{\A}$ (see \cite{HKK+}) we get that in any component of $\left( \MapModuli{0,n}{Z,\dd} \right)^{\A}$ either all maps are broken or all maps are unbroken.

Both a broken subcurve and a broken node give rise to a non-vanishing cosection of zero $\A$-weight in the obstruction bundle, which implies vanishing of the reduced fundamental cycle on components of broken maps.

\medskip

Let us return back to the case $Z = X^{ij}$, $n=2$. Then there's a classification of all unbroken maps

\begin{proposition} \label{PropUnbrokenMaps}
	Let $\left[ f\colon C \to X^{ij} \right]$ be an unbroken map with two fixed points $c_+,c_- \in C$. Then there's exactly two cases
	\begin{enumerate}
	\item \label{IrredCase}
		$C\simeq \PLine$, $f$ is a cover over an $\A$-invariant curve $C_{p'q'}$, branched over $\A$-fixed points $p',q'$ of $C_{p'q'}$; the two marked points are sent to $p',q'$. 
	\item \label{RedCase}
		$C$ has two rational components $C_0$, $C_1$ meeting at a node. $C_0$ is contracted by $f$ to an $\A$-fixed point $p'$ and has both fixed points. $C_1$ is mapped by $f$ as $C$ in case \ref{IrredCase}.
	\end{enumerate}
\end{proposition}
\begin{proof}
	First let us note that two non-contracted components can't meet at a node since by \refEq{EqTangentWeightInHalfspace} the weights at this node can't sum to $0$, so the node would be broken.
	
	If there's no contracted component, then it must be the case \ref{IrredCase}.

	Then if there is a contracted curve, it can't have two nodes because then it would be a breaking subcurve. Then by stability it must contain both marked points and one node. The non-contracted component meeting at this node must connect two points and can't meet another component at the other point. We get the case \ref{RedCase}.
\end{proof}

The integral we want to compute is
\begin{equation*}
	\int\limits_
	{
		\left[
			\MapModuliFixed{\A}{0,(p',q')}{X^{\dd},\dd}
		\right]^{red}
	}
	1
\end{equation*}
for $\dd$ a multiple of $e_i - e_j$, and we see that we have two cases:
\begin{enumerate}
	\item
	If $p'\neq q'$, then there are only contributions from the case \ref{IrredCase} of \refProp{PropUnbrokenMaps}.
	\item
	If $p' = q'$, then there are only contributions from the case \ref{RedCase} of \refProp{PropUnbrokenMaps}.
\end{enumerate}

\begin{remark}
	The contribution of curves of type \ref{RedCase} in \refProp{PropUnbrokenMaps} is diagonal, and can be uniquely reconstructed from the equation
	\begin{equation*}
		D*1 = D.
	\end{equation*}
\end{remark}

\subsubsection{Off-diagonal case}

In this subsection we assume $p'\neq q'$.

\medskip

In this case there is a curve $C_{p'q'}$ connecting $p'$ and $q'$ if and only if
\begin{itemize}
	\item
		$p'$ and $q'$ project to the same point in
		\begin{equation*}
			\prod_{
				\begin{smallmatrix}
					k\neq i
					\\
					k \neq j
				\end{smallmatrix}
			}
			\Gr^{\colambda_k}.
		\end{equation*} 
	
	\item
		Let $p''$ and $q''$ be the projections to $\Gr^{\colambda_i,\colambda_j}$ of $p'$ and $q'$ respectively. Then there must be a coroot $\coalpha$ such that
		\begin{equation*}
			\sigmapi{p''}{1} 
			= 
			\sigmapi{q''}{1}
			-
			\coalpha.
		\end{equation*}
\end{itemize}

\begin{remark}
	The first condition also implies
	\begin{equation*}
		\sigmapi{p''}{2} 
		= 
		\sigmapi{q''}{2}.
	\end{equation*}
\end{remark}

\begin{remark}
	For the corresponding limit points $p,q \in X^{\A}$ we get the relations
	\begin{align}
		\notag
		\deltapi{q}{i} 
		&=
		\deltapi{p}{i} - \coalpha,
		\\
		\label{EqPQRelation}
		\deltapi{q}{j} 
		&=
		\deltapi{p}{j} + \coalpha,
		\\
		\notag
		\deltapi{q}{k}
		&=
		\deltapi{p}{k}
		\text{ for all }
		k,\; k \neq i,\; k \neq j,
	\end{align}	
	which is almost the same as the condition to have a non-zero off-diagonal contribution in the classical case \refEq{EqQIsAdjacentToP}. The only difference is that now we don't require $\wtalpha$ to be positive with respect to some chamber. If $p$ satisfies this condition, the corresponding $q$ also satisfies it. In the classical case the relation was asymmetric, and in the quantum case it becomes symmetric.
\end{remark}

\begin{remark}
	Given $p'$ and $\alpha$ such $q'$ is unique if it exists. It exists if and only if conditions \refEq{EqTangentWeightInHalfspace} and \refEq{EqRightEndFixed} are satisfied, or, equivalently, for the corresponding limit point $p \in X$ we have using \refEq{EqExistenceOfQ1} and \refEq{EqExistenceOfQ2}
\begin{align}
	\label{EqQuantumExistenceOfQ1}
	\left\langle
		\deltapi{p}{i},
		\wtalpha
	\right\rangle
	&= 1,
	\\
	\label{EqQuantumExistenceOfQ2}
	\left\langle
		\deltapi{p}{j},
		\wtalpha
	\right\rangle
	&= -1.
\end{align}	
	These conditions is the same as \refEq{EqExistenceOfQ1} and \refEq{EqExistenceOfQ2} in the classical case, again, modulo the requirement on positivity of $\wtalpha$. 
\end{remark}

The curve $C_{p'q'}$ connecting such $p'$ and $q'$ then has projection $C^{12}_{p'',\wtalpha}$ to $\Gr^{\colambda_i,\colambda_j}$ and projects to a point in $\Gr^{\colambda_k}$. The degree by \refProp{PropCHomologyClasses} is
\begin{equation*}
	\dfrac{\CorootScalar{\coalpha}{\coalpha}}{2} 
	\left( e_i - e_j \right).
\end{equation*}

\medskip

Let $f \colon C \xrightarrow{\sim} C_{p'q'} \subset X^{ij}$ be the $m$-fold cover of $C_{p'q'}$ mapping points $c_+,c_-\in C$ to $p'$ and $q'$ respectively. The degree of $f$ by \ref{PropCHomologyClasses} is 
\begin{equation*}
	\dd_f
	=
	m
	\dfrac{\CorootScalar{\coalpha}{\coalpha}}{2} 
	\left( e_i - e_j \right).
\end{equation*}

Then $\left[ f \right] \in \MapModuliFixed{\A}{0,(p',q')}{X^{ij},\dd_f}$ is an isolated point and the unique unbroken map of degree $\dd_f$ connecting $p'$ and $q'$. Let us figure out the $\A$-equivariant reduced fundamental cycle. It is $c_f$ times the fundamental class of $\left[ f \right]$. By $\A$-equivariant localization.
\begin{equation*}
	c_f
	=
	\dfrac{1}{m}
	\dfrac
	{
		1
	}
	{
		\Euler{\A}
		{
			N
			^{red}
			_{
				\left[ f \right] 
				/ 
				\MapModuli{0,(p',q')}
				{X^{ij},\dd_f}
			}
		}
	},
\end{equation*}
where the factor $m$ appears from the automorphisms of $f$ and the $N^{red}$ is the virtual normal bundle in the reduced tangent-obstruction theory. Namely,
\begin{equation*}
	N
	^{red}
	_{
		\left[ f \right] 
		/ 
		\MapModuli{0,(p',q')}
		{X^{ij},\dd_f}
	}
	=
	N
	^{vir}
	_{
		\left[ f \right] 
		/ 
		\MapModuli{0,(p',q')}
		{X^{ij},\dd_f}
	}
	+
	\left(
		\mathbb{C}\omega_{X^{ij}}
	\right)^\vee
\end{equation*}
as virtual $\A$-representations. Torus $\A$ preserves $\omega_{X^{ij}}$, so $\left( \mathbb{C}\omega_{X^{ij}} \right)^\vee$ is a trivial $\A$-representation. It is introduced to cancel a trivial cosection in the obstruction bundle, which enters with negative sign to $N^{vir}$. We want to show that the reduced virtual normal bundle has no trivial weights. For this we keep track where they appear in the usual virtual normal bundle $N^{vir}$.

Recall that usual $N^{vir}$ is the following\cite{HKK+}
\begin{align*}
	N^{vir}
	_{
		\left[ f \right] 
		/ 
		\MapModuli{0,(p',q')}
		{X^{ij},\dd_f}
	}
	&=
	-
	\Autf{C,c_+,c_-,f}
	+
	\Deff{C,c_+,c_-,f}
	-
	\Obsf{C,c_+,c_-,f}
	\\
	&=
	-
	\Autf{C,c_+,c_-}
	+
	\Deff{C,c_+,c_-}
	+
	\Deff{f}
	-
	\Obsf{f}
\end{align*}

Since $C\simeq \PLine$, the infinitesimal automorphisms $\Autf{C,c_+,c_-}$ are $1$-dimensional, generated by a vector field on $\PLine$ vanishing at both $c_+,c_-$. This has $\A$-weight $0$. The deformations vanish: $\Deff{C,c_+,c_-}=0$, there's only one rational curve with two marked points. Then we look at the deformations and obstructions of $f$:
\begin{equation*}
	\Deff{f}
	-
	\Obsf{f}
	=
	\CoHlgyk{C,f^* TX^{ij}}{0}
	-
	\CoHlgyk{C,f^* TX^{ij}}{1}.
\end{equation*}

The map $f$ induces an action on $C$ with respect to a cover of $\A$. With respect to this action the vector bundle $f^* TX^{ij}$ is equivariant.

We will need the following lemma\cite{MOk}.

\begin{lemma} \label{LemEulerOfAPair}
	Let $T$ be an $\A$-equivariant vector bundle over $C$ without zero weights at the fibers $T_{c_+}$, $T_{c_-}$ over the fixed points $c_+,c_- \in C$. Then
	\begin{equation*}
		\Euler{\A}
		{
			\CoHlgyk{C,T\oplus T^\vee}{0}
			-
			\CoHlgyk{C,T\oplus T^\vee}{1}
		}
		=
		(-1)^{\deg T + \rk T + z}
		\Euler{\A}{T_{c_+}}
		\Euler{\A}{T_{c_-}},
	\end{equation*}
	where $z$ is the number of zero weights in $\CoHlgyk{C,T\oplus T^\vee}{1}$. In particular, 
        \begin{equation*} 
            \CoHlgyk{C,T\oplus T^\vee}{0} - \CoHlgyk{C,T\oplus T^\vee}{1}
        \end{equation*}
        has no zero weights.
\end{lemma}

\begin{proof}
	By the splitting principle, it's enough to prove the statement for the case when $T$ is a line bundle: $T = L$. By symmetry of the expression in terms of exchanging $L \leftrightarrow L^\vee$, we assume that $L$ has non-negative degree.
	
	If $L$ is trivial with weight $\wtchi$, then $\deg L =0$, $\CoHlgyk{C,L\otimes L^\vee}{1} = 0$, so $z = 0$. Moreover, $\CoHlgyk{C,L\otimes L^\vee}{0}$ is two-dimensional with weights $\wtchi$ and $-\wtchi$. Then both sides are $-\left( \wtchi \right)^2$.
	
	Now let us assume the degree of $L$ is positive.

	Let the fibers of $TC$ at the fixed points $c_+$ and $c_-$ have weights $\wtomega$ and $-\wtomega$ respectively.
	
	It is well-known that then for any equivariant line bundle $L$ on $C$ of positive degree the weights of $\CoHlgyk{C,L}{0}$ are equidistant
	\begin{equation*}
		L_{c_+},
		L_{c_+} - \wtomega, 
		\dots, 
		L_{c_-} + \wtomega, 
		L_{c_-},
	\end{equation*}
	where $L_{c_+}$ and $L_{c_-}$ are weights of $L$ at $c_+$ and $c_-$. All weights have multiplicity $1$.
	
	For $L^\vee$ we have $\CoHlgyk{C,L^\vee}{0} = 0$.

	Using Serre duality, we find that for the line bundle $L^\vee$ of degree $<-1$ the weights of 
	\begin{equation*}
		\CoHlgyk{C,L^\vee}{1} \simeq \CoHlgyk{C,L\otimes K_C}{0}^*
	\end{equation*}
	are also equidistant and of multiplicity $1$:
	\begin{equation*}
		-L_{c_-}-\wtomega, 
		-L_{c_-}-2\wtomega,
		\dots
		-L_{c_+}+2\wtomega,
		-L_{c_+}+\wtomega.
	\end{equation*}

	Since $L^\vee\otimes K_C$ has negative degree, $\CoHlgyk{C,L}{1} = 0$.

	If a trivial weight is present in $\CoHlgyk{C,L}{0}$ it's also present in $\CoHlgyk{C,L^\vee}{1}$ by the assumption that $L_{c_+}$ and $L_{c_-}$ are not zero. Then the trivial weights cancel. All other weights in $\CoHlgyk{C,L^\vee}{1}$ have a corresponding weight in $\CoHlgyk{C,L}{0}$ with the opposite sign. So we cancel
	\begin{equation*}
		\deg L\otimes K_C + 1 - z = \deg L - 1 - z
	\end{equation*}
	terms, each gives $-1$ sign. This gives a sign factor
	\begin{equation*}
	(-1)^{\deg L - 1 - z} = (-1)^{\deg L + \rk L + z}.
	\end{equation*}
	
	The remaining terms in $\CoHlgyk{C,L}{0}$ are $L_{c_+}$, $L_{c_-}$ which give $\Euler{\A}{T_{c_+}} \Euler{\A}{T_{c_-}}$.
	
	This proves the lemma.

\end{proof}

Since $TX^{ij}$ has a symplectic structure, we can split it into $TX^{ij} \simeq T \oplus T^\vee$ for some equivariant bundle $T$. The tangent weights $T_{p'} X^{ij}$ and $T_{q'} X^{ij}$ are all non-zero because they are roots, then we are it he setting of \refLem{LemEulerOfAPair}. The lemma gives that the subspaces of weight $0$ in $\Deff{f}-\Obsf{f}$ cancel. 

The remaining $1$-dimensional $\Autf{C,c_+,c_-}$ of zero weight cancel with $\left( \mathbb{C}\omega_{X^{ij}} \right)^\vee$ introduced in the definition of $N^{red}_{ \left[ f \right] / \MapModuli{0,(p',q')} {X^{ij},\dd_f} }$ exactly for this cancellation.

The \refLem{LemEulerOfAPair} gives us even an expression for $c_f$ if we choose a splitting of $f^* TX^{ij}$.

Let $\mathfrak{a}_\mathbb{R}'$ be the real space associated to $\ker \wtalpha \subset\mathfrak{a} $ as in previous chapter. We fix a Weyl chamber $\C' \subset \mathfrak{a}_\mathbb{R}'$.

Then we have a splitting
\begin{equation*}
	\left.
		TX^{ij}
	\right\vert_{C_{p'q'}}
	=
	N^{\C'}_{C_{p'q'}/X^{ij}}
	\oplus
	N^{\wtalpha}_{C_{p'q'}/X^{ij}}
	\oplus
	N^{-\C'}_{C_{p'q'}/X^{ij}},
\end{equation*}
where $N^{\wtalpha}_{C_{p'q'}/X^{ij}}$ has all weights which are zero on $\ker \wtalpha$.

By previous consideration we know that
\begin{equation*}
	N^{\wtalpha}_{C_{p'q'}/X^{ij}}
	=
	TC_{p'q'}
	\oplus
	T^*C_{p'q'}
	\oplus
	V,
\end{equation*}
where $V$ is a trivial vector bundle pulled back from $\prod\Gr^{\colambda_k}$.
Let us choose a Weyl chamber $\widetilde{\C} \subset \mathfrak{a}_\mathbb{R}$ adjacent to $\C'$ (that is, we choose whether $\wtalpha$ or $-\wtalpha$ is positive). Then we can split $V$ with respect to the chamber $\widetilde{C}$
\begin{equation*}
	V = V^{\widetilde{C}} \oplus V^{-\widetilde{C}}.
\end{equation*}
This gives a splitting of $\left. TX^{ij} \right\vert_{C_{p'q'}}$
\begin{equation*}
	\left.
		TX^{ij}
	\right\vert_{C_{p'q'}}
	=
	T
	\oplus
	T^\vee
\end{equation*}
where
\begin{equation*}
	T
	=
	N^{-\C'}_{C_{p'q'}/X^{ij}}
	\oplus
	TC_{p'q'}
	\oplus
	V^{-\widetilde{C}}.
\end{equation*}

Then by \refLem{LemEulerOfAPair} we can compute the Euler class
\begin{align*}
	\Euler{\A}
		{
			\CoHlgyk{C,f^* TX^{ij}}{0}
			-
			\CoHlgyk{C,f^* TX^{ij}}{1}
		}
		&=
		(-1)^{\deg f^*T + \rk T + z}
		\Euler{\A}{f^*T_{c_+}}
		\Euler{\A}{f^*T_{c_-}}
		\\
		=
		(-1)^{\deg f^*T + \rk T + z + 1}
		&
		\Euler{\A}{N^{-\widetilde{\C}}_{p'/X^{ij}}}
		\Euler{\A}{N^{-\widetilde{\C}}_{q'/X^{ij}}},
\end{align*}
where we used that at one of points $p'$ and $q'$ the fiber of $T$ is the same as $N^{-\widetilde{\C}}$ and at another point is differs by a sign of the weight coming from $TC_{p'q'}$.

\medskip

First we find $z$. Let us first find all pairs $L, L^\vee$ of line subbundles in a larger bundle, the tangent bundle of $\GrBD^{\ucolambda}_{\left(t_1,\dots,t_l\right)}$

The pullback of the tangent bundle from $\Gr^{\colambda_k}$ are trivial with no trivial weights. So such pairs of line bundles must be in (the pullback of) $T\Gr^{\colambda_i,\colambda_j}$. Then the weights at the fixed points $L_{p''}$ and $L_{q''}$ are roots. Moreover, since the degree 
\begin{equation*}
	\deg L
	= 
	\int_{C_{p'',q''} } \Chern{\A}{L} 
	= 
	\dfrac{L_{p''}- L_{q''} }{\wtalpha}
\end{equation*}
is an integer, then the weights of $L_{p'}$ and $L_{q'}$ must differ by a multiple of $\wtalpha$. A zero weight if $\CoHlgyk{C,f^* L \otimes f^* L^\vee}{1}$ can appear only if $0$ is on the line connecting $L_{p'}$ and $L_{q'}$. Then this means that one of this weight must be $\alpha$, the other is $-\wtalpha$.

???The multiplicities of $\pm\wtalpha$ are both $1$ by \refCor{CorAMultiplisities}. Despite the computations of the multiplicities are done for $\Gr^{\colambda_i,\colambda_j}_{\comu}$, not for $\Gr^{\colambda_i,\colambda_j}$, the multiplicities for $\pm \wtalpha$ are transferable literally. In the language of Appendix in \cites{Da}, the only difference is if the Right Boundary Condition is satisfied or not. For the points of $\left( \Gr^{\colambda_i,\colambda_j} \right)^{\A}$ with $\left\langle \sigmapi{p''}{2}, \wtalpha \right\rangle = 0$, as in our case, the Right Boundary Condition holds as long as the Left Boundary Condition is satisfied. Thus the multiplicity of both $\pm \wtalpha$ at the tangent space at both $p'$ and $q'$ is $1$. Then there's at most one pair of subbundles $L, L^\vee$ which can contribute to $z$.

One can easily identify this pair of subbundles. One is the tangent space $TC_{p'q'} \subset TX^{ij}$, it has degree $2$. The symplectic form gives the dual subbundle of degree $-2$. In particular, this pair is a subbundle in $TX^{ij}$, not in the larger tangent space we considered. Finally, it gives exactly one zero in $\CoHlgyk{C,f^* L \otimes f^* L^\vee}{1}$, so $z=1$.

\medskip

Let us now find the degree of $T$. By $\A$-equivariant localization it is
\begin{equation*}
	\deg T
	=
	\int_{C_{p'q'}}
	\Chern{\A}{T}
	=
	\dfrac
	{
		\resZ
		{
			\Chern{\A}{T}
		}
		{
			p'
		}
		-
		\resZ
		{
			\Chern{\A}{T}
		}
		{
			q'
		}
	}
	{
		\wtalpha
	}
\end{equation*}
We know these Chern classes because we know the multiplicities of tangent $\A$-weigths at the limit points $p,q \in X$ from \refCor{CorAMultiplisities}. The weights $\wtchi$ coming from counting $\left\langle \deltapi{p}{k}, \wtchi \right\rangle = -1$ for $k\neq i$, $k \neq j$ cancel with similar contributions for $q$, because $\deltapi{p}{k} = \deltapi{q}{k}$ for such $k$. So we only have to take contributions from $\deltapi{p}{i}$, $\deltapi{p}{j}$ and similar for $q$. Denote the sums of these $\A$-weights $c_{i,j,p}$ and $c_{i,j,q}$ and get
\begin{equation*}
	\deg T 
	=
	\dfrac
	{
		c_{i,j,p}
		-
		c_{i,j,q}
	}
	{
		\wtalpha
	}.
\end{equation*}
The are coweights $\deltapi{p}{i}$, $\deltapi{q}{i}$, $\deltapi{p}{j}$ and $\deltapi{q}{j}$ are in the Weyl orbit of a miniscule coweights, so the conditions \refEq{EqPQRelation} imply
\begin{align*}
	\deltapi{q}{i} 
	&= 
	s_{\alpha}
	\deltapi{p}{i}
	\\
	\deltapi{q}{j} 
	&= 
	s_{\alpha}
	\deltapi{p}{j}.
\end{align*}
Since the weight multiplicities are determined by pairing with these coweights, we get
\begin{equation*}
	c_{i,j,q}
	=
	s_{\alpha}
	c_{i,j,p}
	=
	c_{i,j,p}
	-
	\left\langle
		c_{i,j,p},
		\coalpha
	\right\rangle
	\wtalpha
\end{equation*}
and
\begin{equation*}
	\deg T 
	=
	\left\langle
		c_{i,j,p},
		\coalpha
	\right\rangle.
\end{equation*}
Let us consider contributions to this sum of tangent weights which are in a pair of opposite roots $\pm\wtbeta$.
\begin{itemize}
	\item
		Assume first that
		\begin{equation*}
			\left\langle \wtbeta, \deltapi{p}{i} \right\rangle 
			= 
			\left\langle \wtbeta, \deltapi{p}{j} \right\rangle
			=
			0.
		\end{equation*}
		Then both $\pm\beta$ do not contribute to $c_{i,j,p}$.
	\item
		Now let the root $\wtbeta$ which satisfy
		\begin{equation*}
			\left\langle \wtbeta, \deltapi{p}{i} \right\rangle 
			= 
			\left\langle \wtbeta, \deltapi{p}{j} \right\rangle
			\neq
			0.
		\end{equation*}
		Without loss of generality we can assume $\left\langle \wtbeta, \deltapi{p}{i} \right\rangle = -1$. Then $\wtbeta$ contributes $2\left\langle \wtbeta, \coalpha \right\rangle \in 2 \mathbb{Z}$ to $\deg T$. The root $-\wtbeta$ does not contribute to $c_{i,j,p}$.
	\item
		Finally, the root $\wtbeta$ which satisfy
		\begin{equation*}
			\left\langle \wtbeta, \deltapi{p}{i} \right\rangle 
			= 
			-
			\left\langle \wtbeta, \deltapi{p}{j} \right\rangle
			\neq 0
		\end{equation*}
		Without loss of generality we can assume $\left\langle \wtbeta, \deltapi{p}{i} \right\rangle = -1$. The contributes of $\wtbeta$ to $\deg T$ is $\left\langle \wtbeta, \coalpha \right\rangle$ via coweight $\deltapi{p}{i}$ and there is a zero contribution via $\delta{p}{j}$. The second root $-\wtbeta$ has no contribution via $\deltapi{p}{i}$ and contributes $\left\langle -\wtbeta, \coalpha \right\rangle$ via $\deltapi{p}{j}$. Then the total contribution of the pair $\pm\wtbeta$ vanishes.
\end{itemize}
This proves that the contribution of all roots is even. Then the degree $\deg T$ is even, so $\deg f^* T = m \deg T$ is also even.

\medskip

The rank $\rk T$ is $\frac{1}{2} \dim X^{ij} = \frac{1}{2} \dim X$. Knowing $z$, $\deg T$ and $\rk T$ we find the Euler class
\begin{equation*}
	\Euler{\A}
	{
		\CoHlgyk{C,f^* TX^{ij}}{0}
		-
		\CoHlgyk{C,f^* TX^{ij}}{1}
	}
	=
	(-1)^{\dim X/2}
	\Euler{\A}{N^{-\widetilde{\C}}_{p'/X^{ij}}}
	\Euler{\A}{N^{-\widetilde{\C}}_{q'/X^{ij}}}.
\end{equation*}

\medskip

Finally, we get that for $p' \neq q'$ and effective curve class $\dd = m \left( e_i -e_j \right)$, such that the limits $p$ and $q$ are satisfying \refEq{EqPQRelation} for some $\coalpha$s, not necessary positive (or, equivalently, $p'$ and $q'$ are connected by a curve in the deformation to the generic fiber $X^{ij}$ on the wall associated to $e_i-e_j$) the contribution of the reduced Gromov-Witten invariants is
\begin{align*}
	\left\langle
		\StabCEpt{-\C}{\DualPol{\epsilon}}{q},
		\StabCEpt{\C}{\epsilon}{p}
	\right\rangle
	_{0,2,m \left( e_i -e_j \right)}
	^{X,red}
	&=
	\resPol{\DualPol{\epsilon}}{q}
	\resPol{\epsilon}{p}
	\int\limits_
	{
		\left[
			\MapModuliFixed
			{\A}{0,(p,q)}
			{X,m \left( e_i -e_j \right)}
		\right]^{red}
	}
	1
	\\
		&=
	\resPol{\DualPol{\epsilon}}{q}
	\resPol{\epsilon}{p}
	\int\limits_
	{
		\left[
			\MapModuliFixed
			{\A}{0,(p',q')}
			{X^{ij},m \left( e_i -e_j \right)}
		\right]^{red}
	}
	1
	\\
	&=
	\dfrac
	{
		(-1)^{\dim X/2}
	}
	{
		m
	}
	\dfrac
	{
		\resPol{\DualPol{\epsilon}}{q}
		\resPol{\epsilon}{p}
	}
	{
		\Euler{\A}{N^{-\widetilde{\C}}_{p'/X^{ij}}}
		\Euler{\A}{N^{-\widetilde{\C}}_{q'/X^{ij}}}
	}
	\\
	&=
	\dfrac{1}{m}
	\dfrac
	{
		\resPol{\epsilon}{q}
		\resPol{\epsilon}{p}
	}
	{
		\Euler{\A}{N^{-\widetilde{\C}}_{p'/X^{ij}}}
		\Euler{\A}{N^{-\widetilde{\C}}_{q'/X^{ij}}}
	}
	\\
	&=
	\dfrac{1}{m}
	\dfrac
	{
		\resPol{\epsilon}{q}
		\resPol{\epsilon}{p}
	}
	{
		\Euler{\A}{N^{-\widetilde{\C}}_{p/X^{ij}}}
		\Euler{\A}{N^{-\widetilde{\C}}_{q/X^{ij}}}
	}
	\\
	&=
	\sign{\tilde{\C}}{\epsilon}{q}{X}
	\sign{\tilde{\C}}{\epsilon}{p}{X}
	\dfrac{1}{m}
	\\
	&=
	\sigma^{\epsilon}_{p,q}
	\dfrac{1}{m},
\end{align*}
where $\sigma^{\epsilon}_{p,q}$ is exactly the same sign as appeared in \refEq{EqSignsInOmega}.

If $p\neq q $ and they don't satisfy \refEq{EqPQRelation} for these $i<j$ and any $\coalpha$, then the contribution is zero.

\medskip

We summarize the computation of off-diagonal terms. Let $p \neq q$ be related by \ref{EqPQRelation} for a coroot $\coalpha$ and $i<j$. Then the maps connecting $p$ and $q$ give the following off-
\begin{align*}
	D \qdot \StabCEpt{\C}{\epsilon}{p}
	&=
	-\hbar
	\sum\limits_{m=1}^\infty
	q^{m \cdot \dd\left(i,j,\alpha\right)}
	\left\langle
		D,
		m
		\frac{\CorootScalar{\coalpha}{\coalpha}}{2} 
		\left( e_i - e_j \right)
	\right\rangle
	\sigma^{\epsilon}_{p,q}
	\dfrac{1}{m}
	\StabCEpt{\C}{\epsilon}{q}
	+
	\dots
	\\
	&=
	-\hbar
	\left\langle
		D, 
		e_i - e_j
	\right\rangle
	\dfrac
	{q^{\dd\left(i,j,\alpha\right)}}
	{1-q^{\dd\left(i,j,\alpha\right)}}
	\dfrac{\CorootScalar{\coalpha}{\coalpha}}{2}
	\sigma^{\epsilon}_{p,q}
	\StabCEpt{\C}{\epsilon}{q}
	+
	\dots
	\\
	&=
	-\hbar
	\left\langle
		D, 
		e_i - e_j
	\right\rangle
	\dfrac
	{q^{\dd\left(i,j,\alpha\right)}}
	{1-q^{\dd\left(i,j,\alpha\right)}}
	\StabCEpt{\C}{\epsilon}{\OmegaOperator{ij}{\alpha,\epsilon} p}
	+
	\dots
\end{align*}
where $\dd\left(i,j,\alpha\right) = \frac{\CorootScalar{\coalpha}{\coalpha}}{2} \left( e_i - e_j \right)$, dots stay for terms with $\StabCEpt{\C}{\epsilon}{r}$ for $r \neq q$ and $\OmegaOperator{ij}{\alpha,\epsilon}$ is the operator defined in \refEq{EqOmegaAlphaOperDef}.

\subsubsection{Diagonal case}

Now let us consider the case $p'=q'$. The computation is quite similar to the off-diagonal case, so we will skip some details.

We know by \refProp{PropUnbrokenMaps} that the $\A$-fixed unbroken maps with both marked ponts mapped to $p' = q'$ are
\begin{equation*}
	f
	\colon
	C
	\to
	X^{ij}
\end{equation*}
where $C$ has two irreducible rational components
\begin{equation*}
	C = C_0 \cup C_1,
\end{equation*}
$C_0$ and $C_1$ meet at a node $s$, $C_0$ is contracted by $f$ to $p'=q'$, both marked points $c_+,c_-$ are in $C_0$, $C_1$ $m$-foldly covers $C_{p'r'}$ under $f$, where $r' \in \left( X^{ij} \right)^{\A}$. The tangent weight $\wtalpha$ of  $C_{p'r'}$ at $p'$ uniquely determines $r'$. 

The degree of $f$ by \ref{PropCHomologyClasses} is 
\begin{equation*}
	\dd_f
	=
	m
	\dfrac{\CorootScalar{\coalpha}{\coalpha}}{2} 
	\left( e_i - e_j \right).
\end{equation*}

For each pair $p',r'\in\left( X^{ij}\right)^{\A}$ connected by an $\A$ invariant curve, and every positive integer $m$ there is unique such map $f$. Thus $\left[ f \right]$ is an isolated point in $\MapModuliFixed{\A}{0,(p',p')}{X^{ij},\dd_f}$. We are interested in finding the reduced fundamental class at this point. It's $c_f$ times the class of the point $\left[ f \right]$. As previously,

\begin{equation*}
	c_f
	=
	\dfrac{1}{m}
	\dfrac
	{
		1
	}
	{
		\Euler{\A}
		{
			N
			^{red}
			_{
				\left[ f \right] 
				/ 
				\MapModuli{0,(p',p')}
				{X^{ij},\dd_f}
			}
		}
	},
\end{equation*}
where $N^{red}$ is the reduced virtual normal bundle, which is the same correction of the virtual normal bundle as before.

The usual virtual normal bundle to $\left[ f \right]$ is 
\begin{equation*}
	N^{vir}
	_{
		\left[ f \right] 
		/ 
		\MapModuli{0,(p',p')}
		{X^{ij},\dd_f}
	}
	=
	-
	\Autf{C,c_+,c_-}
	+
	\Deff{C,c_+,c_-}
	+
	\Deff{f}
	-
	\Obsf{f}
\end{equation*}

The infinitesimal automorphism bundle $\Autf{C,c_+,c_-}$ is now $2$-dimensional and comes from automorphisms of $C_1$ preserving the node $s$. One weight subspace is generated by rotations preserving $s$ and $f^{-1}(r')$. Similarly to the case $p'\neq q'$, this one has weight $0$. The new weight subspace is generated by parallel shifts in $\ALine\simeq C_1 \setminus \lbrace s \rbrace$. This one has the same weight as the weight of $f^* TC_{p'r'}$ at $f^{-1}(r')$, that is $-\wtalpha/m$.

The deformations $\Deff{C,c_+,c_-}$ come from the smoothing of the node $s$. The weight of this deformation is the sum of adjacent tangent weigths
\begin{equation*}
	0 + \wtalpha/m
\end{equation*}

Then we get
\begin{equation*}
	\Euler{\A}
	{
		-
		\Autf{C,c_+,c_-}
		+
		\Deff{C,c_+,c_-}
		+
		\left(
			\mathbb{C}\omega_{X^{ij}}
		\right)^\vee
	}
	=
	\dfrac{\wtalpha/m}{-\wtalpha/m}
	=
	-1.
\end{equation*}
In particular, all zero $\A$-weights cancel.

For the deformations and obstructions we have
\begin{equation*}
	\Deff{f}
	-
	\Obsf{f}
	=
	\CoHlgyk{C,f^* TX^{ij}}{0}
	-
	\CoHlgyk{C,f^* TX^{ij}}{1}.
\end{equation*}
To compute these sheaf cohomology, let us write the following exact sequence of sheaves
\begin{equation*}
	\begin{tikzcd}
		\OO_C
		\arrow[r]
		&
		\OO_{C_0} \oplus \OO_{C_1}
		\arrow[r]
		&
		\OO_s.
	\end{tikzcd}
\end{equation*}
Tensoring it with a vector bundle $Tf^* TX^{ij}$, we get
\begin{equation*}
	\begin{tikzcd}
		f^* TX^{ij}
		\arrow[r]
		&
		\left.
			f^* TX^{ij}
		\right\vert_{C_0}
		\oplus
		\left.
			f^* TX^{ij}
		\right\vert_{C_1}
		\arrow[r]
		&
		\left.
			f^* TX^{ij}
		\right\vert_{s}.
	\end{tikzcd}
\end{equation*}

The associated long exact sequence of sheaf cohomology gives
\begin{align*}
	\Deff{f}
	-
	\Obsf{f}
	&=
	\CoHlgyk
	{
		C_0,
		f^* TX^{ij}
	}{0}
	-
	\CoHlgyk
	{
		C_0,
		f^* TX^{ij}
	}{1}
	\\
	&+
	\CoHlgyk
	{
		C_1,
		f^* TX^{ij}
	}{0}
	-
	\CoHlgyk
	{
		C_1,
		f^* TX^{ij}
	}{1}
	\\
	&-
		\CoHlgyk
	{
		s,
		f^* TX^{ij}
	}{0}
\end{align*}
where we used that only these cohomologies don't vanish by the dimension of the support.

Moreover, $\left. f^* TX^{ij} \right\vert_{C_0}$ is a constant vector bundle with the same weight as $\left. f^* TX^{ij} \right\vert_{s}$, so 
\begin{equation*}
    \CoHlgyk{C_0,f^* TX^{ij}}{1} = 0
\end{equation*}
and $\CoHlgyk{C_0,f^* TX^{ij}}{0}$ cancels with $\CoHlgyk{s,f^* TX^{ij}}{0}$.

We get
\begin{equation*}
	\Deff{f}
	-
	\Obsf{f}
	=
	\CoHlgyk
	{
		C_1,
		f^* TX^{ij}
	}{0}
	-
	\CoHlgyk
	{
		C_1,
		f^* TX^{ij}
	}{1}
\end{equation*}

The map $f$ restricted to $C_1$ is exactly the same is the off-diagonal case. By applying the same arguments, we get

\begin{equation*}
	\Euler{\A}
	{
		\CoHlgyk
		{
			C_1,
			f^* TX^{ij}
		}{0}
		-
		\CoHlgyk
		{
			C_1,
			f^* TX^{ij}
		}{1}
	}
	=
	(-1)^{\dim X/2}
	\Euler{\A}{N^{-\widetilde{\C}}_{p'/X^{ij}}}
	\Euler{\A}{N^{-\widetilde{\C}}_{r'/X^{ij}}}
\end{equation*}
In particular, $\Deff{f} - \Obsf{f}$ has no zero weights.

Then the Euler class of the reduced virtual normal bundle is
\begin{equation*}
	\Euler{\A}
	{
		N
		^{red}
		_{
			\left[ f \right] 
			/ 
			\MapModuli{0,(p',p')}
			{X^{ij},\dd_f}
		}
	}
	=
	(-1)^{\dim X/2+1}
	\Euler{\A}{N^{-\widetilde{\C}}_{p'/X^{ij}}}
	\Euler{\A}{N^{-\widetilde{\C}}_{r'/X^{ij}}}
\end{equation*}

Finally, the contribution from $f$ to 
$
	\left\langle
		\StabCEpt{-\C}{\DualPol{\epsilon}}{p},
		\StabCEpt{\C}{\epsilon}{p}
	\right\rangle
	_{0,2,\dd_f}
	^{X,red}
$
is
\begin{align*}
	c_f
	\resPol{\DualPol{\epsilon}}{p}
	\resPol{\epsilon}{p}
	&=
	\dfrac
	{
		(-1)^{\dim X/2+1}
	}
	{
		m
	}
	\dfrac
	{
		\resPol{\DualPol{\epsilon}}{p}
		\resPol{\epsilon}{p}
	}
	{
		\Euler{\A}{N^{-\widetilde{\C}}_{p'/X^{ij}}}
		\Euler{\A}{N^{-\widetilde{\C}}_{r'/X^{ij}}}
	}
	\\
	&=
	\dfrac
	{
		(-1)^{\dim X/2+1}
	}
	{
		m
	}
	\dfrac
	{
		\resPol{\DualPol{\epsilon}}{p}
		\resPol{\epsilon}{p}
	}
	{
		\Euler{\A}{N^{-\widetilde{\C}}_{p/X^{ij}}}
		\Euler{\A}{N^{-\widetilde{\C}}_{r/X^{ij}}}
	}
	\\
	&=
	\dfrac
	{
		(-1)^{\dim X/2+1}
	}
	{
		m
	}
	\dfrac
	{
		\Euler{\A}{N_{p/X^{ij}}}
	}
	{
		\Euler{\A}{N^{-\widetilde{\C}}_{p/X^{ij}}}
		\Euler{\A}{N^{-\widetilde{\C}}_{r/X^{ij}}}
	}
	\\
	&=
	\dfrac
	{
		(-1)^{\dim X/2+1}
	}
	{
		m
	}
	\dfrac
	{
		\Euler{\A}{N^{\widetilde{\C}}_{p/X^{ij}}}
	}
	{
		\Euler{\A}{N^{-\widetilde{\C}}_{r/X^{ij}}}
	}
	\\
	&=
	-
	\dfrac{1}{m}
	\dfrac
	{
		\Euler{\A}{N^{-\widetilde{\C}}_{p/X^{ij}}}
	}
	{
		\Euler{\A}{N^{-\widetilde{\C}}_{r/X^{ij}}}
	}.
\end{align*}
where $p,r\in X^{\A}$ are limits of $p',r' \in \left( X^{ij} \right)^{\A}$ in the central fiber $X$.

In what follows we denote
\begin{equation*}
	\omega_{r,p} = 
	\dfrac
	{
		\Euler{\A}{N^{-\widetilde{\C}}_{p/X^{ij}}}
	}
	{
		\Euler{\A}{N^{-\widetilde{\C}}_{r/X^{ij}}}
	}
	\in
	\EqCoHlgy{\A}{\pt}_{loc}
\end{equation*}
the coefficients appearing here. They are the same factors $\omega_{p,q}$ as in the computation of the classical multiplication in \cites{Da} and by the same reasoning do not depend on a choice of $\widetilde{\C}$ as long as it is adjacent to the wall $\wtalpha = 0$.

\begin{remark}
	Let us show how one can derive this result form
	\begin{equation*}
	D * 1 = D 
	\Longleftrightarrow 
	D \qdot 1 = 0.
	\end{equation*}
	First, we have
	\begin{equation*}
		1 
		= 
		\sum_{p \in X^{\A}}
		\dfrac
		{
			\StabCEpt{\C}{\epsilon}{p}
		}
		{
			\resPol{\epsilon}{p}
		}
		\mod \hbar
	\end{equation*}
	in $\EqCoHlgy{\T}{X}_{loc}$ due to the off-diagonal vanishing and the definition of the stable envelopes.
	Then we have
	\begin{equation*}
		D \qdot 1
		=
		\sum_{p \in X^{\A}}
		\dfrac
		{
			D \qdot
			\StabCEpt{\C}{\epsilon}{p}
		}
		{
			\resPol{\epsilon}{p}
		}
		\mod \hbar
	\end{equation*}
	Then we see that contributions of a cover of an irreducible curve connecting $p'$ and $q'$ exactly cancels with a diagonal contribution of the same cover with additional rational component contracting to $q'$.
\end{remark}

\subsubsection{Quantum multiplication}

Let us define operators similar to \refEq{EqOmega0OperDef} and \refEq{EqOmegaCOperDef} used for the purely classical multiplication. Unfortunately, contributions from long and short roots give different contributions to the degree because of the factor $\dfrac{\CorootScalar{\coalpha}{\coalpha}}{2}$. This is why we split the operators into two groups

\begin{align*}
	\OmegaOperatorTilde{ij}{0,\mathrm{st}}
	\colon
	\EqCoHlgy{\A}{X^{\A}}
	&\to
	\EqCoHlgy{\A}{X^{\A}}_{loc}
	\\
	\CohUnit{p}
	&\mapsto
	\left[
		\sum_q
		\omega_{q,p}
	\right]
	\CohUnit{p}
\end{align*}
where the summation runs over all $q$ such that $p$ and $q$ satisfy the relation \ref{EqPQRelation} for a \textbf{short} coroot $\coalpha$ and $i<j$.
 
Similarly, we define
\begin{align*}
	\OmegaOperatorTilde{ij}{0,\mathrm{lg}}
	\colon
	\EqCoHlgy{\A}{X^{\A}}
	&\to
	\EqCoHlgy{\A}{X^{\A}}_{loc}
	\\
	\CohUnit{p}
	&\mapsto
	2
	\left[
		\sum_q
		\omega_{q,p}
	\right]
	\CohUnit{p}
\end{align*}
where the summation runs over all $q$ such that $p$ and $q$ satisfy the relation \ref{EqPQRelation} for a \textbf{long} coroot $\coalpha$ and $i<j$.

\begin{align*}
	\OmegaOperatorTilde{ij}{\mathrm{st},\epsilon}
	&=
	\OmegaOperatorTilde{ij}{0,\mathrm{st}}
	+
	\sum\limits_{\coalpha \textrm{ short}}
	\OmegaOperator{ij}{\alpha,\epsilon},
	\\
	\OmegaOperatorTilde{ij}{\mathrm{lg},\epsilon}
	&=
	\OmegaOperatorTilde{ij}{0,\mathrm{lg}}
	+
	\sum\limits_{\coalpha \textrm{ long}}
	\OmegaOperator{ij}{\alpha,\epsilon},
\end{align*}
where by our convention the shortest coroot is normalized as $\CorootScalar{\coalpha}{\coalpha} = 2$. Since we only consider root systems with minuscule weights, long roots appear in types $B$ and $C$ with norm squared equal to $4$. The coefficient $2$ in the definition of $\OmegaOperator{ij}{0,\mathrm{lg}}$ is exactly $\dfrac{\CorootScalar{\coalpha}{\coalpha}}{2}$ for the long coroots.

Now we can sum the computation of this chapter in one statement.

\begin{theorem}
	The purely quantum multiplication by a divisor $D \in \EqCoHlgyk{\T}{\Gr^{\ucolambda}_{\comu}}{2}$ is given by the following formula
	\begin{equation*}
		D \qdot \StabCEpt{\C}{\epsilon}{p}
		=
		-\hbar
		\sum\limits_{i<j}
		\left\langle
			D,
			e_i - e_j
		\right\rangle
		\StabCE{\C}{\epsilon}
		\left[
			\dfrac{q^{e_i - e_j}}{1-q^{e_i - e_j}}
			\OmegaOperatorTilde{ij}{\mathrm{st},\epsilon}
			\left(
				p
			\right)
			+
			\dfrac{q^{2(e_i - e_j)}}{1-q^{2(e_i - e_j)}}
			\OmegaOperatorTilde{ij}{\mathrm{lg},\epsilon}
			\left(
				p
			\right)
		\right]
	\end{equation*}
\end{theorem}
\begin{proof}
	This follows from the Divisor Equation \refEq{EqDivisorEquation} and the reduction in \refProp{PropRank1Reduction}. Then the identification of terms in operators $\OmegaOperator{ij}{\mathrm{sh},\epsilon}$, $\OmegaOperator{ij}{\mathrm{lg},\epsilon}$ and the curve contributions found is straightforward.
\end{proof}

\begin{corollary} \label{CorPurelyQuantumPart}
	The purely quantum multiplication by the first Chern classes $\ChernE{\T}{i}$ in $\EqCoHlgy{\T}{\Gr^{\ucolambda}_{\comu}}$ is the following
	\begin{align*}
		\ChernE{\T}{i} \qdot \StabCEpt{\C}{\epsilon}{p}
		=
		&
		-\hbar
		\sum\limits_{i<j}
		\StabCE{\C}{\epsilon}
		\left[
			\dfrac{q^{e_i - e_j}}{1-q^{e_i - e_j}}
			\OmegaOperatorTilde{ij}{\mathrm{st},\epsilon}
			\left(
				p
			\right)
			+
			\dfrac{q^{2(e_i - e_j)}}{1-q^{2(e_i - e_j)}}
			\OmegaOperatorTilde{ij}{\mathrm{lg},\epsilon}
			\left(
				p
			\right)
		\right]
		\\
		&+\hbar
		\sum\limits_{j<i}
		\StabCE{\C}{\epsilon}
		\left[
			\dfrac{q^{e_j - e_i}}{1-q^{e_j - e_i}}
			\OmegaOperatorTilde{ji}{\mathrm{st},\epsilon}
			\left(
				p
			\right)
			+
			\dfrac{q^{2(e_j - e_i)}}{1-q^{2(e_j - e_i)}}
			\OmegaOperatorTilde{ji}{\mathrm{lg},\epsilon}
			\left(
				p
			\right)
		\right]
	\end{align*}
\end{corollary}
\begin{proof}
	This follows from the definition of classes $e_i$.
\end{proof}

A direct computation shows that

\begin{lemma}
	\begin{equation*}
	\OmegaOperator{ij}{0}
	=
	\OmegaOperatorTilde{ij}{0,\mathrm{st}}
	+
	\OmegaOperatorTilde{ij}{0,\mathrm{lg}}
	+
	\CorootScalar{\colambda_i}{\colambda_j}
	\end{equation*}
\end{lemma}

For simply-laced types $ADE$ all roots are short, so $\OmegaOperator{ij}{\mathrm{lg},\epsilon} = 0$ and
\begin{equation*}
	\OmegaOperatorTilde{ij}{\epsilon}
	:=
	\OmegaOperatorTilde{ij}{\mathrm{st},\epsilon}
	=
	\OmegaOperator{ij}{\epsilon}
	-
	\CorootScalar{\colambda_i}{\colambda_j},
\end{equation*}
where
\begin{equation*}
	\OmegaOperator{ij}{\epsilon}
	:=
	\OmegaOperator{ij}{0,\mathrm{st}}
	+
	\sum\limits_{\alpha}
	\OmegaOperator{ij}{\alpha,\epsilon}
	=
	\OmegaOperator{ij}{\C,\epsilon}
	+
	\OmegaOperator{ij}{-\C,\epsilon}
\end{equation*}
for any Weyl chamber $\C$.

\begin{remark}
	This explains why we used tilde in the notation for the diagonal contribution operator $\OmegaOperatorTilde{ij}{0}$ in the quantum case: it will differ from the Cartan part of the Casimir operator by a constant shift $\CorootScalar{\colambda_i}{\colambda_j}$. 
\end{remark}

In what follows we use notation
\begin{equation*}
	\Kterm{ij}
	\left(
		p
	\right)
	= 
	\CorootScalar{\colambda_i}{\colambda_j}
	\CohUnit{p}.
\end{equation*}

Then we can write the following formula for the quantum multiplication

\begin{theorem} \label{ThmFullQuantumMultiplication}
	The quantum multiplication by the first Chern classes $\ChernE{\T}{i}$ in $\EqCoHlgy{\T}{\Gr^{\ucolambda}_{\comu}}$ for a simply-laced $\G$ is the following
	\begin{align*}
		\ChernE{\T}{i} * \StabCEpt{\C}{\epsilon}{p}
		&=
		\StabCE{\C}{\epsilon}
		\left[
			\Hterm{i}
			\left(
				p
			\right)
			+
			\hbar
			\sum\limits_{j \neq i}
			\dfrac
			{
				q^{e_i}
				\OmegaOperator{ij}{\C,\epsilon}
				\left(
					p
				\right) 
				+
				q^{e_j}
				\OmegaOperator{ij}{-\C,\epsilon}
				\left(
					p
				\right)
			}
			{
				q^{e_i} - q^{e_j}
			}
		\right]
		\\
		&-
		\hbar
		\StabCE{\C}{\epsilon}
		\left[
			\sum\limits_{j < i}
			\dfrac
			{
				q^{e_j}
				\Kterm{ij}
				\left(
					p
				\right)
			}
			{
				q^{e_i} - q^{e_j}
			}
			+
			\sum\limits_{j > i}
			\dfrac
			{
				q^{e_i}
				\Kterm{ij}
				\left(
					p
				\right)
			}
			{
				q^{e_i} - q^{e_j}
			}
		\right],
	\end{align*}
	where the rational function is assumed to be expanded into a power series in the region
	\begin{equation*}
		\left|
			q^n
		\right|
		\ll
		\left|
			q^n
		\right|
		\text{ iff }
		n<m.
	\end{equation*}
\end{theorem}

\begin{proof}
	Let us expand the fraction
	\begin{equation*}
		\hbar
		\dfrac
		{
			q^{e_i}
			\OmegaOperator{ij}{\C,\epsilon}
			\left(
				p
			\right) 
			+
			q^{e_j}
			\OmegaOperator{ij}{-\C,\epsilon}
			\left(
				p
			\right)
		}
		{
			q^{e_i} - q^{e_j}
		}
	\end{equation*}
	for $i<j$.
	\begin{align}
		\notag
		\hbar
		\dfrac
		{
			q^{e_i}
			\OmegaOperator{ij}{\C,\epsilon}
			\left(
				p
			\right) 
			+
			q^{e_j}
			\OmegaOperator{ij}{-\C,\epsilon}
			\left(
				p
			\right)
		}
		{
			q^{e_i} - q^{e_j}
		}
		&=
		\hbar
		\dfrac
		{
			q^{e_i}
			\left[
				\OmegaOperator{ij}{\C,\epsilon}
				\left(
					p
				\right) 
				+
				\OmegaOperator{ij}{-\C,\epsilon}
				\left(
					p
				\right) 
			\right]
			+
			\left[
				q^{e_j}
				-
				q^{e_i}
			\right]
			\OmegaOperator{ij}{-\C,\epsilon}
			\left(
				p
			\right)
		}
		{
			q^{e_i} - q^{e_j}
		}
		\\
		\notag
		&=
		\hbar
		\left[
			-
			\OmegaOperator{ij}{-\C,\epsilon}
			\left(
				p
			\right)
			+
			\dfrac
			{
				q^{e_i}
			}
			{
				q^{e_i} - q^{e_j}
			}
			\OmegaOperator{ij}{\epsilon}
			\left(
				p
			\right)
		\right]
		\\
		\notag
		&=
		-
		\hbar
		\OmegaOperator{ij}{-\C,\epsilon}
		\left(
			p
		\right)
		-
		\hbar
		\dfrac
		{
			q^{e_i-e_j}
		}
		{
			1 - q^{e_i-e_j}
		}
		\OmegaOperator{ij}{\epsilon}
		\left(
			p
		\right)
		\\
		\notag
		&=
		-
		\hbar
		\OmegaOperator{ij}{-\C,\epsilon}
		\left(
			p
		\right)
		-
		\hbar
		\dfrac
		{
			q^{e_i-e_j}
		}
		{
			1 - q^{e_i-e_j}
		}
		\OmegaOperatorTilde{ij}{\epsilon}
		\left(
			p
		\right)
		\\
		\label{EqExpansionIJ}
		&\phantom{=}\;
		-
		\hbar
		\dfrac
		{
			q^{e_i-e_j}
		}
		{
			1 - q^{e_i-e_j}
		}
		\Kterm{ij}
		\left(
			p
		\right).
	\end{align}
	Similarly for the case $j<i$:
	\begin{align}
		\notag
		\hbar
		\dfrac
		{
			q^{e_i}
			\OmegaOperator{ij}{\C,\epsilon}
			\left(
				p
			\right) 
			+
			q^{e_j}
			\OmegaOperator{ij}{-\C,\epsilon}
			\left(
				p
			\right)
		}
		{
			q^{e_i} - q^{e_j}
		}
		&=
		\hbar
		\dfrac
		{
			\left[
				q^{e_i}
				-
				q^{e_j}
			\right]
			\OmegaOperator{ij}{\C,\epsilon}
			\left(
				p
			\right) 
			+
			q^{e_j}
			\left[
				\OmegaOperator{ij}{\C,\epsilon}
				\left(
					p
				\right)
				+
				\OmegaOperator{ij}{-\C,\epsilon}
				\left(
					p
				\right)
			\right]
		}
		{
			q^{e_i} - q^{e_j}
		}
		\\
		\notag
		&=
		\hbar
		\left[
			\OmegaOperator{ij}{\C,\epsilon}
			\left(
				p
			\right)
			+
			\dfrac
			{
				q^{e_j}
			}
			{
				q^{e_i} - q^{e_j}
			}
			\OmegaOperator{ij}{\epsilon}
			\left(
				p
			\right)
		\right]
		\\
		&=
		\notag
		\hbar
		\OmegaOperator{ji}{-\C,\epsilon}
		\left(
			p
		\right)
		+
		\hbar
		\dfrac
		{
			q^{e_j-e_i}
		}
		{
			1 - q^{e_j-e_i}
		}
		\OmegaOperator{ij}{\epsilon}
		\left(
			p
		\right)
		\\
		\notag
		&=\hbar
		\OmegaOperator{ji}{-\C,\epsilon}
		\left(
			p
		\right)
		+
		\hbar
		\dfrac
		{
			q^{e_j-e_i}
		}
		{
			1 - q^{e_j-e_i}
		}
		\OmegaOperatorTilde{ij}{\epsilon}
		\left(
			p
		\right)
		\\
		\label{EqExpansionJI}
		&\phantom{=}\;
		+
		\hbar
		\dfrac
		{
			q^{e_j-e_i}
		}
		{
			1 - q^{e_j-e_i}
		}
		\Kterm{ij}
		\left(
			p
		\right),
	\end{align}
	where we used $\OmegaOperator{ij}{\C,\epsilon} = \OmegaOperator{ji}{-\C,\epsilon}$.
	
	Then we see that the first term in \refEq{EqExpansionIJ} and \refEq{EqExpansionJI} is the classical term in \refThm{ThmReformulatedClassicalMultiplication}. The second term in \refEq{EqExpansionIJ} and \refEq{EqExpansionJI} is the purely quantum part in \refCor{CorPurelyQuantumPart}, expanded in the correct region. The last term in both expressions is exactly compensated by terms with $\Kterm{ij}$.
\end{proof}
\section{Representation-Theoretic Reformulation}
\label{SecRepTheory}

\subsection{Trigonometric Knizhnik-Zamolodchikov connection}

First let us recall some constructions from representation theory. Let $\GLangDual$ be a semisimple complex connected Lie group and $\gLangDual$ be its Lie algebra. We denote the Cartan subalgebra as $\hLangDual \subset \gLangDual$ and the root subspaces with respect to $\hLangDual$ as $\grootsub{\coalpha}^\vee$.

Since $\gLangDual$ is simple, it admits a non-degenerate symmetric $\gLangDual$-invariant pairing
\begin{equation*}
	\CorootScalar{\bullet}{\bullet}
	\colon
	\gLangDual \otimes \gLangDual 
	\to 
	\mathbb{C}
\end{equation*}
which is unique up to scaling. Taking the dual map and identifying $\left( \gLangDual \right)^* \simeq \gLangDual$ via the same symmetric form, we get a $\gLangDual$-invariant element
\begin{equation*}
	\Omega \in \gLangDual \otimes \gLangDual.
\end{equation*}
The image of this element in the universal enveloping algebra $U\left( \gLangDual \right)$ is the quadratic Casimir element. We call $\Omega$ the Casimir operator.

We normalize $\CorootScalar{\bullet}{\bullet}$ by the condition that on the longest coroot $\wtalpha$ we have $\CorootScalar{\wtalpha}{\wtalpha} = 2$.

Since $\CorootScalar{\bullet}{\bullet}$ is $\hLangDual$-invariant, it pairs $\hLangDual$ with $\hLangDual$ and $\grootsub{\coalpha}^\vee$ with $\grootsub{-\coalpha}^\vee$. This gives a splitting
\begin{equation*}
	\Omega 
	= 
	\Omega_0 
	+ 
	\sum\limits_{\coalpha} 
	\Omega_{\coalpha},
\end{equation*}
where
\begin{align*}
	\Omega_0 
	&\in
	\hLangDual \otimes \hLangDual,
	\\ 
	\Omega_{\coalpha}
	&\in
	\grootsub{\coalpha}^\vee \otimes \grootsub{-\coalpha}^\vee.
\end{align*}

Later it is important for us that the definition of $\Omega$ implies the following normalization conditions
\begin{align} \label{EqOmegaNormalization}
	\CorootScalarNoComa{\Omega_{\coalpha}} 
	&= 
	1,
	\\
	\CorootScalarNoComa{\Omega_{0}} 
	&= 
	\rk \gLangDual.
\end{align}

Quite often the terms operators $\Omega_{\alpha}$ and $\Omega_0$ are written in the following form
\begin{align*}
	\Omega_0 
	= 
	\sum_i
	h_i \otimes h^i
	\\
	\Omega_{\alpha}
	=
	e_{\coalpha} \otimes e_{-\coalpha}
\end{align*}
for dual bases $h_i$, $h^i$ in $\hLangDual$ and a non-trivial $e_{\coalpha} \in \grootsub{\coalpha}^\vee$. However, by \refEq{EqOmegaNormalization} this assumes that
\begin{equation*}
	\CorootScalar{e_{\coalpha}}{e_{-\coalpha}}
	=
	1.
\end{equation*}
We use a \textbf{different} normalization later for $e_{\coalpha}$, so we prefer not to write it in this form.

Now we fix a Weyl chamber $\C$. Then we can define a "half-Casimir" operator
\begin{equation*}
	\Omega_\C 
	=
	\dfrac{1}{2}
	\Omega_0
	+
	\sum\limits_{\coalpha\gC{\C}0}
	\Omega_{\coalpha}
\end{equation*}

It is easy to see that
\begin{equation*}
	\Omega
	=
	\Omega_\C
	+
	\Omega_{-\C}.
\end{equation*}
Often this decomposition is made with respect to the dominant chamber $\C = +$, then $-\C=-$ is the antidomimant chamber, and
\begin{equation*}
	\Omega
	=
	\Omega_{+}
	+
	\Omega_{-}.
\end{equation*}

For our reasons it's better to keep $\C$ arbitrary.

Let $V_1, \dots, V_l$ be finite-dimensional complex representations of $\GLangDual$. Then for any
\begin{equation*}
	X \in \gLangDual
\end{equation*}
we define the operator $X^i$ acting on the tensor product
\begin{equation*}
	V_1 \otimes \dots \otimes V_l
\end{equation*}
via action of $X$ on $V_i$:
\begin{equation*}
	X^i
	\cdot
	v_1 \otimes \dots \otimes v_i \otimes \dots \otimes v_l
	=
	v_1 \otimes \dots \otimes X \cdot v_i \otimes \dots \otimes v_l
\end{equation*}
Similarly, for any
\begin{equation*}
	Y \in
	\gLangDual \otimes \gLangDual
\end{equation*}
we can define the operator $Y^{ij}$ acting on the same tensor product in such a way, that the first component of $Y$ acts on $V_i$ and the second component acts on $V_j$:
\begin{equation*}
	Y^{ij}
	\cdot
	v_1 \otimes \dots \otimes v_i \otimes \dots \otimes v_j \otimes \dots \otimes v_l
	=
	\sum_k
	v_1 \otimes \dots \otimes Y^{(1)}_k v_i \otimes \dots \otimes Y^{(2)}_k v_j \otimes \dots \otimes v_l,
\end{equation*}
where
\begin{equation*}
	Y = 
	\sum_k
	Y^{(1)}_k \otimes Y^{(2)}_k.
\end{equation*}

\medskip

Let us introduce formal variables $z_1,\dots,z_l$.

The trigonometric Knizhnik-Zamolodchikov connection is a collection of differential operators
\begin{equation} \label{EqTrigKZ}
	\nabla^{KZ,\C}_i
	=
	z_i\dfrac{\partial \phantom{z_i}}{\partial z_i}
	-
	H^i
	- 
	\hbar
	\sum_{j\neq i} 
	\dfrac
	{
		z_i \Omega_{\C}^{ij}+z_j \Omega_{-\C}^{ij}
	}
	{
		z_i - z_j
	}
\end{equation}
on the space
\begin{equation*}
	\mathbb{C}[\hbar](z_1,\dots,z_l)
	\otimes
	V_1 
	\otimes
	\dots 
	\otimes 
	V_l.
\end{equation*}

Here $H \in \hLangDual$ is any element in the Cartan subalgebra. These operators commute.

We think geometrially about this object as a connection vector bundle with fiber $V_1 \otimes \dots \otimes V_l$ over the base which is an open part of $\AffSpace^l$ with coordinates $z_1, \dots, z_l$. Commutativity of the operators means the flatness of the connection.

To combine the connections for various $H$ in one object, let us extend the base to an open subset of $\AffSpace^{l}\otimes \AffSpace \left( \gLangDual \right) \simeq \AffSpace^{l+r}$, where $r$ is the rank of $\gLangDual$. If we fix a basis a basis $h_{\overline{k}}$ in $\hLangDual$, then $\AffSpace^{l}\otimes \AffSpace \left( \gLangDual \right)$ has coordinates $z_1,\dots,z_l$ and $h^{\overline{1}},\dots,h^{\overline{r}}$, the vectors $h^{\overline{k}} \in \mathfrak{h}$ form the dual basis and can be identified with vectors in $\hLangDual$ via $\CorootScalar{\bullet}{\bullet}$. The connection in $z$-direction will be of form \refEq{EqTrigKZ} with
with
\begin{equation*}
	H^i = 	
	\sum_{\overline{k}}
	h^{\overline{k}}
	h_{\overline{k}}^i.
\end{equation*}

For the $a$-direction we use the operators
\begin{equation*}
	H_{\overline{k}}
	=
	h_{\overline{k}}
	=
	\sum_i
	h_{\overline{k}}^i.
\end{equation*}
These are not differential operators, but we need them for comparison. 

The $\hLangDual$-invariance of $\Omega_\C$ and commutativity of $\hLangDual$ implies that these operators commute.

The commutativity with the Cartan (which are the operators $H_{\overline{k}}$) allows us to restrict these operators from the whole tensor product $V_1 \otimes \dots \otimes V_l$ in the fiber to its weight subspace
\begin{equation*}
	V_1 
	\otimes 
	\dots 
	\otimes 
	V_l
	\left[
		\comu
	\right].
\end{equation*}

For further comparison it is convenient to add to the operators $\nabla^{KZ,\C}_i$ an extra term
\begin{equation*}
	\widehat{\nabla}^{KZ,\C}_i
	=
	\nabla^{KZ,\C}_i
	-
	\dfrac{\hbar}{2}
	\sum_{\overline{k}}
	h^i_{\overline{k}}
	H_{\overline{k}}.
\end{equation*}
Equivalently, one replaces $H^i$ in \refEq{EqTrigKZ} by 
\begin{equation*}
	\widehat{H}^i
	=
	H^i
	+
	\dfrac{\hbar}{2}
	\sum_{\overline{k}}
	h^i_{\overline{k}}
	H_{\overline{k}}.
\end{equation*}

After restriction to the subspace of weight $\comu$ this gives operators
\begin{equation*}
	\widehat{H}^i
	=
	H^i
	+
	\dfrac{\hbar}{2}
	\sum_{\overline{k}}
	h^i_{\overline{k}}
	\left\langle
		h_{\overline{k}},
		\comu
	\right\rangle
	=
	\sum_{\overline{k}}
	\left[
		h^{\overline{k}}
		+
		\dfrac{\hbar}{2}
		\CorootScalar{\comu}{h^{\overline{k}}}
	\right]
	h_{\overline{k}}^i.
\end{equation*}

The operators $\widehat{\nabla}^{KZ,\C}_i$ still commute because $\widehat{H}^i$ and $H^i$ differ by a shift in $\AffSpace \left( \hLangDual \right)$.

\medskip

To see explicitly the terms of the quantum multiplication in trigonometric KZ connection it is useful to expand the operator in a power series in the region
\begin{equation*}
	|z_1|
	\ll 
	|z_2|
	\ll 
	\dots
	\ll
	|z_l|,
\end{equation*}
i.e. write base change this operator as an operator on
\begin{equation*}
	\mathbb{C}((z_1/z_2 \dots z_{l-1}/z_l)) 
	\otimes 
	\Sym^*
	\left( 
		\mathfrak{h}^\vee 
		\oplus 
		\mathbb{C}\hbar 
	\right)
	\otimes
	V_1\
	\otimes
	\dots 
	\otimes 
	V_l
	[\comu]
\end{equation*}

Then the connection becomes

\begin{align*} 
	\notag
	\widehat{\nabla}^{KZ,\C}_i 
	= 
	z_i\dfrac{\partial \phantom{z_i}}{\partial z_i} 
	-
	\widehat{H}^i
	&-
	\hbar
	\left[
		\sum_{j<i}
		\Omega_{-\C}^{ji}
		-
		\sum_{j>i}
		\Omega_{-\C}^{ij}
	\right]
	\\
	&-
	\hbar
	\left[
		\sum_{j<i} 
		\dfrac
		{
			(z_j/z_i)
		}
		{
			1 - (z_j/z_i)
		}
		\Omega^{ij}
		-
		\sum_{j>i}
		\dfrac
		{
			(z_i/z_j)
		}
		{
			1 - (z_i/z_j)
		}
		\Omega^{ij}
	\right]
\end{align*}

where one treats fractions $\dfrac{x}{1-x}$ as power series $\sum\limits_{i=1}^{\infty} x^i$. We see that it has almost the same terms as the quantum multiplication. In particular, the first term in brackets in the expansion corresponds to the classical part in the multiplication by a divisor, and the second term in brackets is the quantum correction (modulo some gauge transform).

\subsection{Quantum connection}

Let us consider a trivial vector fiber bundle over the affine associated to $\EqCoHlgyk{\T}{X}{2}$ with the fiber $\QCoHlgy{\T}{X}$.

The quantum connection is the following collection of differential operators
\begin{equation} \label{EqQDE}
	\nabla^Q_{D}
	=
	\partial_D
	-
	D * ,
\end{equation}
where $D \in \EqCoHlgyk{\T}{X}{2}$ and the operators $\partial_D$ act as if formally $q^{\dd} = e^{\dd}$
\begin{equation*}
	\partial_D q^\dd 
	= 
	\left\langle
		D, \dd
	\right\rangle
	q^\dd
\end{equation*}
on $q^\dd$ with $\dd \in \Hlgyk{X, \mathbb{Z}}{2}$, and trivially on $q$-constant classes $\EqCoHlgy{\T}{X} \subset \QCoHlgy{\T}{X}$. The pairing is well-defined if since it vanishes for all constant equivariant classes $D$, $D \in \EqCoHlgy{\T}{\pt} \subset \EqCoHlgy{\T}{X}$.

Because of the vanishing of $\partial_D$ when $D \in \EqCoHlgy{\T}{\pt}$, we have that for these $D$ the operator $\nabla^Q_{D}$ just multiplies by $D$ and these directions are less interesting. This allows one to reduce base to
\begin{equation*}
	\EqCoHlgyk{\T}{X}{2}
	/
	\EqCoHlgyk{\T}{\pt}{2}
	=
	\CoHlgyk{X}{2},
\end{equation*}
by taking any subspace of $\EqCoHlgyk{\T}{X}{2}$ complementary to $\EqCoHlgyk{\T}{\pt}{2}$.

From \refProp{PropSecondCohomologyGeneration} we know that the classes $\ChernE{\T}{i}$ and a basis in $\EqCoHlgyk{\T}{\pt}{2}$ generate $\EqCoHlgyk{\T}{X}{2}$, possibly with some relations.

Taking $D_i = \ChernE{\T}{i}$ in \refEq{EqQDE} we have differential operators
\begin{equation*}
	\nabla^Q_{i}
	=
	q^{e_i}\dfrac{\partial \phantom{q^{e_i}}}{\partial q^{e_i}}
	-
	D_i *.
\end{equation*}
By \refThm{ThmFullQuantumMultiplication} their action on $\StabCEpt{\C}{\epsilon}{p}$ is the following
\begin{align*}
	\nabla^Q_{i} 
	\StabCEpt{\C}{\epsilon}{p}
	&=
	- D_i * \StabCEpt{\C}{\epsilon}{p}
	\\
	&=
	\StabCE{\C}{\epsilon}
	\left[
		-
		\Hterm{i}
		\left(
			p
		\right)
		-
		\hbar
		\sum\limits_{j \neq i}
		\dfrac
		{
			q^{e_i}
			\OmegaOperator{ij}{\C,\epsilon}
			\left(
				p
			\right) 
			+
			q^{e_j}
			\OmegaOperator{ij}{-\C,\epsilon}
			\left(
				p
			\right)
		}
		{
			q^{e_i} - q^{e_j}
		}
	\right]
	\\
	&+
	\hbar
	\StabCE{\C}{\epsilon}
	\left[
		\sum\limits_{j < i}
		\dfrac
		{
			q^{e_j}
			\Kterm{ij}
			\left(
				p
			\right)
		}
		{
			q^{e_i} - q^{e_j}
		}
		+
		\sum\limits_{j > i}
		\dfrac
		{
			q^{e_i}
			\Kterm{ij}
			\left(
				p
			\right)
		}
		{
			q^{e_i} - q^{e_j}
		}
	\right].
\end{align*}

Now we perform a gauge transform to get rid of terms with $\Kterm{ij}$. Take
\begin{equation*}
	\psi
	=
	-
	\hbar
	\sum_{i<j}
	\CorootScalar{\colambda_i}{\colambda_j}
	\ln
	\left(
		1
		-
		\dfrac
		{
			q^{e_i}
		}
		{
			q^{e_j}
		}
	\right)
	\in
	\QuantumPowerSeries.
\end{equation*}

Then if we do a gauge transform
\begin{equation*}
	\widehat{\nabla}^Q_{i}
	=
	e^{-\psi}
	\nabla^Q_{i}
	e^{\psi}
	=
	\nabla^Q_{i}
	+
	q^{e_i}\dfrac{\partial \psi}{\partial q^{e_i}}
\end{equation*}
and compute
\begin{equation*}
	q^{e_i}\dfrac{\partial \psi}{\partial q^{e_i}}
	=
	-\hbar
	\sum_{j<i}
	\CorootScalar{\colambda_i}{\colambda_j}
	\dfrac{q^{e_j}}{q^{e_i}-q^{e_j}}
	-\hbar
	\sum_{j>i}
	\CorootScalar{\colambda_i}{\colambda_j}
	\dfrac{q^{e_i}}{q^{e_i}-q^{e_j}},
\end{equation*}
we get
\begin{equation*}
	\widehat{\nabla}^Q_{i} 
	\StabCEpt{\C}{\epsilon}{p}
	=
	\StabCE{\C}{\epsilon}
	\left[
		-
		\Hterm{i}
		\left(
			p
		\right)
		-
		\hbar
		\sum\limits_{j \neq i}
		\dfrac
		{
			q^{e_i}
			\OmegaOperator{ij}{\C,\epsilon}
			\left(
				p
			\right) 
			+
			q^{e_j}
			\OmegaOperator{ij}{-\C,\epsilon}
			\left(
				p
			\right)
		}
		{
			q^{e_i} - q^{e_j}
		}
	\right].
\end{equation*}
Now the quantum connection has the form almost identical to the one of the trigonometric Knizhnik-Zamolodchikov connection.

%
%

\subsection{Identification}

Finally, we show that the quantum connection of $\Gr^{\ucolambda}_{\comu}$ can be identified with the trigonometric Knizhnik-Zamolodchikov connection. Our assumptions are that $\G$ is simply-laced and all $\lambda_i$ in $\ucolambda = \left( \colambda_1, \dots, \colambda_l \right)$

We consider the Knizhnik-Zamolodchikov connection on
\begin{equation*}
	V^{\ucolambda}_{\comu}
	=
	V_{\colambda_1}
	\otimes
	\dots
	\otimes 
	V_{\colambda_l}
	\left[ 
		\comu
	\right],
\end{equation*}
where $V_{\colambda_i}$ are irreducible finite dimensional representations with highest weight $\colambda_i$.

The trigonometric Knizhnik-Zamolodchikov operators act on
\begin{equation} \label{EqVHatDef}
	\widehat{V}^{\ucolambda}_{\comu}
	=
	\mathbb{C}((z_1/z_2 \dots z_{l-1}/z_l)) 
	\otimes 
	\Sym^*
	\left( 
		\mathfrak{h}^\vee 
		\oplus 
	\mathbb{C}\hbar 
	\right)
	\otimes 
	V^{\ucolambda}_{\comu}
\end{equation}

Let us start with writing explicitly the standard generators of $U_{\gLangDual}$. Fix a Weyl chamber $\C$, let $\coalpha_i$, $1\leq i \leq \rk \mathfrak{g}$ be the simple roots with respect to $\C$. Then on $e_{\coalpha_i}$, $e_{-\coalpha_i}$, $\wtalpha_i$ we have the standard relations
\begin{align*}
	\left[
		\wtalpha_i,
		\wtalpha_j
	\right]
	&=
	0
	\\
	\left[
		e_{\coalpha_i},
		e_{-\coalpha_i}
	\right]
	&=
	\wtalpha_i
	\\
	\left[
		\wtalpha_j,
		e_{\pm\coalpha_i}
	\right]
	&=
	\left\langle
		\wtalpha_j,
		\pm
		\coalpha_i
	\right\rangle
	e_{\pm\coalpha_i}.
\end{align*}
This defines $e_{\pm\coalpha_i}$ uniquely up to scaling via the exponent of the adjoint action of $\hLangDual$, in particular, $e_{\coalpha_i} \otimes e_{-\coalpha_i}$ is defined uniquely. Similarly, if we take $e_{\coalpha} \in \grootsub{\coalpha}^\vee$ such that
\begin{equation*}
	\left[
		e_{\coalpha},
		e_{-\coalpha}
	\right]
	=
	\wtalpha,
\end{equation*}
then this uniquely determines $e_{\coalpha} \otimes e_{-\coalpha}$. The space $\grootsub{\coalpha}^\vee \otimes \grootsub{-\coalpha}^\vee$ has dimension one, so $e_{\coalpha} \otimes e_{-\coalpha}$ is a multiple of $\Omega_\alpha$. To figure out the coefficient, look at
\begin{align*}
	\CorootScalar{\coalpha}{\coalpha}
	&=
	\CorootScalar
	{
		\left[
			e_{\coalpha},
			e_{-\coalpha}
		\right]
	}
	{\coalpha}
	\\
	&=
	\CorootScalar
	{
		e_{\coalpha}
	}
	{
		\left[
			e_{-\coalpha},
			\coalpha
		\right]
	}
	\\
	&=
	\CorootScalar
	{
		e_{\coalpha}
	}
	{
		2
		e_{-\coalpha}
	},
\end{align*}
and using the normalization \refEq{EqOmegaNormalization}, we get
\begin{equation*}
	e_{\coalpha} \otimes e_{-\coalpha}
	=
	\dfrac
	{
		\CorootScalar{\coalpha}{\coalpha}
	}
	{
		2
	}
	\Omega_\alpha.
\end{equation*}
In particular, in simply-laced case
\begin{equation*}
	e_{\coalpha} \otimes e_{-\coalpha}
	=
	\Omega_\alpha.
\end{equation*}

Then we pick a basis in $V_{\colambda_i}$. Since $\colambda_i$ is minuscule, the weights of this representation are $W\colambda_i$.

Let $v_{\conu_0}$ be the highest weight vector with respect to $\C$. Then for any $\conu \in W\colambda_i$ we have a sequence
\begin{equation*}
	\conu_0,
	\conu_1,
	\dots,
	\conu_i
	=
	\conu,
\end{equation*}
such that the difference $\conu_{j-1} - \conu_j$ is a simple root with respect to $\C$.
Then we define
\begin{equation*}
	v_{\nu}
	=
	e_{\conu_i - \conu_{i-1}}
	\dots
	e_{\conu_{1} - \conu_{0}}
	v_{\conu_{0}}.
\end{equation*}
If there are different such sequences, then the this vector doesn't depend on such a choice.

A straightforwad check shows that for simple $\coalpha_i$ and $v_{\conu_1}\otimes v_{\conu_2} \in V_{\colambda_1}\otimes V_{\colambda_2}$
\begin{equation*}	
	e_{\coalpha_i}
	\otimes
	e_{\coalpha_i}
	\cdot
	v_{\conu_1}\otimes v_{\conu_2}
	=
	\begin{cases}
		v_{\conu_1+\coalpha_i} \otimes v_{\conu_2-\coalpha_i}
		&\text{ if }
		\conu_1+\coalpha_i 
		\in 
		W \colambda_1
		\text{ and }
		\conu_2-\coalpha_i 
		\in 
		W \colambda_2
		\\
		0
		&\text{ otherwise.}
	\end{cases}
\end{equation*}

\medskip

On the side of quantum connection we fix the polarization
\begin{equation*}
	\resPol{\epsilon}{p}
	=
	\Euler{\A}{N^{-\C}_{p/X}}
\end{equation*}
and we omit it in the notation.

\medskip

Now we are ready to identify the sides.

There is an isomorphism
\begin{equation*}
	V^{\ucolambda}_{\comu}
	\xrightarrow{\sim}
	\Hlgy{\left( \Gr^{\ucolambda}_{\comu} \right)^{\T},\mathbb{C}}
\end{equation*}
given by sending $v_{\conu_1}\otimes \dots \otimes v_{\conu_l}$ to $\CohUnit{p}$, where $p\in\left( \Gr^{\colambda}_{\comu} \right)^{\T}$ is the following:
\begin{equation*}
	\deltapi{p}{i} = \conu_i
\end{equation*}
for all $i$.

Then we can get a isomorphism
\begin{equation*}
	\EqHlgy{\T}{pt}_{loc}
	\otimes
	V^{\ucolambda}_{\comu}
	\xrightarrow{\sim}
	\EqHlgy{\T}{\left( \Gr^{\ucolambda}_{\comu} \right)^{\T},\mathbb{C}}_{loc}
	\xrightarrow[\sim]{\StabC{\C}}
	\EqHlgy{\T}{\Gr^{\ucolambda}_{\comu},\mathbb{C}}_{loc},
\end{equation*}
which we denote by $S_{\C}$.

Then the following diagrams
\begin{equation*}
	\begin{tikzcd}
		\EqHlgy{\T}{pt}_{loc}
		\otimes
		V^{\ucolambda}_{\comu}
		\arrow[r,"\Omega^{ij}_\alpha"]
		\arrow[d,"S_{\C}"]
		&
		\EqHlgy{\T}{pt}_{loc}
		\otimes
		V^{\ucolambda}_{\comu}
		\arrow[d,"S_{\C}"]
		\\
		\EqHlgy{\T}{\Gr^{\ucolambda}_{\comu},\mathbb{C}}_{loc}
		\arrow[r,"\OmegaOperator{ij}{\alpha}"]
		&
		\EqHlgy{\T}{\Gr^{\ucolambda}_{\comu},\mathbb{C}}_{loc}
	\end{tikzcd}
\end{equation*}
\begin{equation*}
	\begin{tikzcd}
		\EqHlgy{\T}{pt}_{loc}
		\otimes
		V^{\ucolambda}_{\comu}
		\arrow[r,"\Omega^{ij}_0"]
		\arrow[d,"S_{\C}"]
		&
		\EqHlgy{\T}{pt}_{loc}
		\otimes
		V^{\ucolambda}_{\comu}
		\arrow[d,"S_{\C}"]
		\\
		\EqHlgy{\T}{\Gr^{\ucolambda}_{\comu},\mathbb{C}}_{loc}
		\arrow[r,"\OmegaOperator{ij}{0}"]
		&
		\EqHlgy{\T}{\Gr^{\ucolambda}_{\comu},\mathbb{C}}_{loc}
	\end{tikzcd}
\end{equation*}
commute for a simply-laced $\G$. The operators $\widehat{H}^i$ and $\Hterm{i}$ coincide as well.

\medskip

Sending $z_i$ to $q^{e_i}$ provides a morphism
\begin{equation*}
	\widehat{V}^{\ucolambda}_{\comu}
	\hookrightarrow
	\mathbb{C}((z_1/z_2 \dots z_{l-1}/z_l)) 
	\otimes 
	\EqCoHlgy{\T}{\pt}_{loc}
	\otimes
	V^{\ucolambda}_{\comu}
	\xrightarrow[\sim]{S_{\C}}
	\QCoHlgy{\T}{\Gr^{\ucolambda}_{\comu},\mathbb{C}}_{loc}
\end{equation*}
where $\widehat{V}^{\ucolambda}_{\comu}$ is from \refEq{EqVHatDef} and the last arrow becomes an isomophism if there is only one relation on $e_i$. We call this composition $S_{\C}$ as well because they are obviously related.

Then we summarize the computations in the main theorem.

\begin{theorem}
	\label{ThmKZequalsQ}
		For a simply-laced $\G$ the quantum connection for $\Gr^{\underline{\lambda}}_\mu$ matches with the trigonometric Knizhnik-Zamolodchikov connection. That is, the following diagram
	\begin{equation*}
		\begin{tikzcd}
			\widehat{V}^{\ucolambda}_{\comu}
			\arrow[r,"\widehat{\nabla}^{KZ,\C}_i"]
			\arrow[d,"S_{\C}"]
			&
			\widehat{V}^{\ucolambda}_{\comu}
			\arrow[d,"S_{\C}"]
			\\
			\EqCoHlgy{\T}{\Gr^{\ucolambda}_{\comu},\mathbb{C}}_{loc}
			\arrow[r,"\widehat{\nabla}^{Q}_i"]
			&
			\EqCoHlgy{\T}{\Gr^{\ucolambda}_{\comu},\mathbb{C}}_{loc}
		\end{tikzcd}
	\end{equation*}
	commutes.
\end{theorem}


\appendix
\section{\texorpdfstring{$\T$}{\textT}-equivariant geometry}
\label{AppendixA}

\subsection{Walls in equivariant paramenters}

We first need to explain what we mean by "walls". Define
\begin{equation*}
	\mathfrak{t}_\mathbb{Q} 
	= 
	\CocharLattice{\T}
	\otimes_{\mathbb{Z}}
	\mathbb{Q}
\end{equation*}

For each point $T \in \mathfrak{t}_{\mathbb{Q}}$ one can define the fixed point locus as $X^{T} = X^{nT}$, where $n\in\mathbb{Z}$ is any integer such that $nT \in \CocharLattice{\T} $. Then for a generic $T$ the fixed locus is $X^{\T}$. The locus where $X^{T}$ is larger is a union of hyperplanes called walls. Let us remind how positions of walls are related to the tangent weights of the fixed locus.

Let us first prove the following auxiliary statement

\begin{lemma}\label{LemInvSubvarietyHasFixedPoints}
Let $\T$ be a torus and $\pi\colon X\to X_0$ be a $\T$-equivariant proper morphism to an affine variety. Assume, moreover, $X_0$ has the unique $\T$-fixed point $x\in X_0$ and there is a $\T$-cocharacter $\colambda \colon \Gm \to \T$ contracting $X_0$ to $x$. Then for any $\T$-invariant closed subvariety $Y\subset X$ has a $\T$-fixed point.
\end{lemma}

\begin{proof}
The image $\pi \left( Y \right)$ is closed $\T$-invariant. By the action of $\colambda$ on any point, we get that $x \in \pi \left( Y \right)$. So $Y$ intersects non-trivially with $\pi^{-1} \left( x \right)$. The then $Y \cap \pi^{-1} \left( x \right)$ is non-empty and proper, invariant $\T$. Since $\T$ is a torus, there is a $\T$-fixed point in $Y \cap \pi^{-1} \left( x \right)$ (by a theorem due to Borel\cite{B}), and hence in $Y$.
\end{proof}

\begin{remark}
    We could not apply the theorem from \cite{B} immediately to $Y$ since $X$ is not proper, it only has a proper map to $X_0$. However, extra conditions on $X_0$ allowed us to reduce our case to the well-known result.
\end{remark}

Recall that we have such a proper map
\begin{equation*}
    m_{ \ucolambda } 
    \colon
    \Gr^{\ucolambda}_{\comu}
    \to
    \Gr^{\colambda}_{\comu}
\end{equation*}
to a $\T$-equivariant affine space. By \refProp{PropSlicesProperties} $\Gr^{\colambda}_{\comu}$ is contracted to the unique fixed point by $\LoopGm \subset \T$. So the statement holds for $X=\Gr^{\ucolambda}_{\comu}$.

Then there's an immediate corollary of \refLem{LemInvSubvarietyHasFixedPoints} and computation of tangent weights.

\begin{corollary} \label{CorTWallsPositions}
    Given a $T \in \mathfrak{t}_\mathbb{Q} $ the fixed locus $X^T$ is greater than $X^\T$ only if there is an affine root $\wtchi = \wtalpha + n\hbar$ such that $T \in \ker \wtchi$.
\end{corollary}

\begin{proof}
    Any component of $X^{T}$ is a closed $\T$-invariant smooth subvariety of $X$ Take any connected component $Z \subset X^{T}$ which is not a point. Then by \refLem{LemInvSubvarietyHasFixedPoints} it has a $\T$-fixed point $p$. The tangent space $T_p Z$ is not trivial and $T_p X$ has only weights of form $\wtalpha + n\hbar$. Thus the value of some affine $\wtalpha + n\hbar$ is zero on $T$ as stated.
\end{proof}

\begin{remark}
    The statement of this Corollary can be formulated in the way that the only walls in $\mathfrak{t}_\mathbb{Q}$ are of form $\wtalpha + n\hbar = 0$.
\end{remark}

The contraction of $\Gr^{\colambda}_{\comu}$ to the $\T$-fixed point is done by $\LoopGm$. Unfortunately, the corresponding cocharacter $\LoopGm \hookrightarrow \T$ can be on the wall, and, equivalently, the fixed locus is larger than $X^{\T}$. However, there exists a convenient perturbation of the cocharacter.

\begin{proposition} \label{PropGenericContraction}
    There exists a cocharacter $\cochi \in \mathfrak{t}_{\mathbb{Q}}$ such that 
    \begin{enumerate}
        \item \label{CondContraction}
            $\Gm$-action given by (a positive power of) $\cochi$ contracts $\Gr^{\colambda}_{\comu}$ to the unique $\T$-fixed point,
        \item \label{CondGeneric}
            $\cochi$ is not on the wall, i.e. $X^\cochi = X^{\T}$.  
    \end{enumerate}
\end{proposition}

\begin{proof}
    The global functions on $\Gr^{\colambda}_{\comu}$ is a finitely generated algebra graded by $\T$-weights. The condition \ref{CondContraction} is equivalent to saying that all weights of generators have positive value at $\cochi$. It's a finite intersection of open sets, hence an open subset of $\mathfrak{t}_{\mathbb{Q}}$. It's non-empty because the cocharacter $\LoopGm \hookrightarrow \T$ belongs to it. Finally, there are finitely many walls (hyperplanes) where condition \ref{CondGeneric} fails. They can't cover a non-empty open set in $\mathfrak{t}_{\mathbb{Q}}$, so there is $\cochi$ that satisfies both conditions.
\end{proof}

\subsection{Affine root subgroups}

Now we want to describe the fixed locus of the generic point on a wall, i.e. subvarieties $X^{\ker \wtchi} \subset X$.

For an explicit description of $X^{\ker \wtchi}$ we need the $\SL{2}$-subgroups of $\GK$ corresponding to the affine roots. One could define the $\SL{2}$-subgroups by referring to the fact that $\GK$ is a quotient of a Kac-Moody group, so it inherits such root subgroups from it. We provide a more explicit description.

For the group $\SL{2}$ let us denote the standard torus as $\SLtwoH$, the standard Borel and the opposite Borel subgroups as $\SLtwoB$ and $\SLtwoBop$, the corresponding unipotent radicals as $\SLtwoU$ and $\SLtwoUop$.

Fix an isomorphism $\SLtwoH \xrightarrow{\sim} \Gm$.

First, recall what the $\SL{2}$-subgroups corresponding to non-affine roots. Let $\wtalpha$ be a root of $\G$ and $\coalpha \colon \Gm \to \A$ the corresponding coroot. Then there is a homomorphism of algebraic groups

\begin{equation*}
	\iota_{\wtalpha}\colon \SL{2} \to \G
\end{equation*}

such that the diagram

\begin{equation*}
	\begin{tikzcd}
		\SL{2} \arrow[rr,"\iota_{\wtalpha}"] & & \G \\
		\SLtwoH \arrow[u,hook] \arrow[r,"\sim"] & \Gm \arrow[r,"\coalpha"] & \A \arrow[u,hook]
	\end{tikzcd}
\end{equation*}

commutes. This is unique up to a precomposition with an adjoint action of $\SLtwoH$ on $\SL{2}$.

By abuse of notation we use the same notation for
\begin{equation*}
	\iota_{\wtalpha} \colon \SL{2} \to \GK
\end{equation*}
given by postcomposing with the natural inclusion $\G \hookrightarrow \GK$.

One can construct from $\iota_{\wtalpha}$ similar maps for affine roots. Take
\begin{align*}
	Ad_{\Tcowt{t}{n\coomega}}
	\colon 
	\SL{2} 
	&\to 
	\SLF{2}{\K} 
	\\
	\begin{pmatrix}
		a & b \\
		c & d
	\end{pmatrix} 
	&\mapsto
	\begin{pmatrix}
		a & bt^n \\
		ct^{-n} & d
	\end{pmatrix}
\end{align*}
(even though $\Tcowt{t}{\coomega}$ is not a well-defined element of $\SLF{2}{\K}$, conjugation by it is well-defined).

Then define
\begin{equation*}
	\iota_{\wtalpha+n\hbar}
	=
	\iota_{\wtalpha}
	\circ
	Ad_{\Tcowt{t}{n\coomega}}.
\end{equation*}
This maps sends the positive nilpotents of $\mathfrak{sl}_2(2,\mathbb{C})$ to $t^n\grootsub{\wtalpha} = \grootsub{\wtalpha + n\hbar}$, as desired.

\subsection{Invariant curves in \texorpdfstring{$\Gr$}{\textGr}} 

We use the root subgroups to construct $\T$-invarint curves in a way, similar to the case of ordinary flag variety.

Now one can act by the image of $\SL{2}$ under $\iota_{\wtalpha+n\hbar}$ on a $\T$-fixed point $\Tfixed{\colambda}$.

Denote the orbit as
\begin{equation*}
	C_{\colambda,\wtalpha+n\hbar} 
	=
	\iota_{\wtalpha+n\hbar}
	\left(\SL{2}\right)
	\cdot
	\Tfixed{\colambda}.
\end{equation*}

The Cartan of $\SL{2}$ is in the stabilizer of $\Tfixed{\colambda}$ because the point is $\T$-fixed. By the computation of tangent weights of $T_{\Tfixed{\colambda}} \Gr$, the positive unipotent radical $\SLtwoU$ is in the stabilizer iff $n \geq \langle \colambda, \wtalpha \rangle$. Similarly, the negative unipotent radical $\SLtwoUop$ is in the stabilizer iff $n \leq \langle \colambda, \wtalpha \rangle$. This gives that the orbit is a point if $n = \langle \colambda, \wtalpha \rangle$, and the following isomorphism if $n\neq \langle \colambda, \wtalpha \rangle$
\begin{equation*}
	C_{\colambda,\wtalpha+n\hbar} 
	\simeq 
	\SL{2}/\B_\pm
	\simeq 
	\PLine,
\end{equation*}
where $\B_\pm$ is either $\SLtwoB$ or $\SLtwoBop$, depending if $n$ is greater or lesser than $\langle \colambda, \wtalpha \rangle$. I.e. $C_{\colambda,\wtalpha+n\hbar}$ are rational curves in $\Gr$. It's an easy fact that they are $\T$-invariant, and, moreover, all $\T$-invariant curves are of this form.

Curves $C_{\colambda,\wtalpha+n\hbar}$ have two $\T$-fixed points. One is, obviously, $\Tfixed{\colambda}$. The other is given by the image of (any representative in $\SL{2}$ of) the nontrivial element of the Weyl group of $\SL{2}$ ($N(\SLtwoH)/\SLtwoH \simeq \mathbb{Z}/2\mathbb{Z}$). A straightforward computation shows that it is $\Tfixed{s_{\wtalpha+n\hbar} \colambda}$ the image of $\colambda$ under the affine Weyl action:
\begin{equation*}
	s_{\wtalpha+n\hbar} \colambda 
	:=
	s_{\wtalpha} \colambda + n \coalpha 
	=
	\colambda 
	+ 
	\left( 
		n - 
		\langle \colambda, \wtalpha \rangle 
	\right)
	\coalpha.
\end{equation*}
I.e. $s_{\wtalpha+n\hbar}$ is the reflection preserving $\langle \bullet, \wtalpha \rangle = n$ hyperplane in the coweight space.

The tangent weight of $C_{\colambda,\wtalpha+n\hbar}$ for $n<\langle \colambda, \wtalpha \rangle$ at $\Tfixed{\colambda}$ is exactly $\wtalpha + n\hbar$ and it's $-(\wtalpha + n\hbar)$ at $\Tfixed{s_{\wtalpha+n\hbar} \colambda}$. For $n>\langle \colambda, \wtalpha \rangle$ the weights are exchanged.

\subsection{Some \texorpdfstring{$\T$}{\textT}-invariant curves in the resolutions of slices}

Similar $\T$-invariant curves in $\Gr^{ \ucolambda }_{\comu}$ can be constructed. Fix $i,j$ such that $0\leq i<j\leq l$. Consider the following homomorphism
\begin{align*}
	\iota_{\wtalpha+n\hbar}^{i,j} 
	\colon 
	\SL{2} 
	&\to 
	\G^{\times l} 
	\\
	A 
	&\mapsto 
	\left( 
		\UnitG, \dots, \UnitG,
		\iota_{\wtalpha_n\hbar}A,
		\dots, 
		\iota_{\wtalpha_n\hbar}A,
		\UnitG, \dots, \UnitG 
	\right),
\end{align*}
where $i$ is the number of the first nontrivial component and $j$ is the last one. This homomorphism and natural component-wise action of $\G^{\times l}$ on $\Gr^{\times l}$ gives an $\SL{2}$-action.

Take a $\T$-fixed point $p \in \Gr^{ \ucolambda }_{\comu} \subset \Gr^{\times l}$. In general, the orbit under $\iota_{\wtalpha+n\hbar}^{i,j} \left( \SL{2} \right)$ is not contained in $\Gr^{ \ucolambda }_{\comu}$, conditions are no longer satisfied. However, if 
\begin{equation*}
	\langle 
		\sigmapi{p}{i}, \wtalpha 
	\rangle 
	= 
	\langle 
		\sigmapi{p}{j}, \wtalpha 
	\rangle 
	= n
\end{equation*}
or
\begin{equation*}
	\langle 
		\sigmapi{p}{i}, \wtalpha 
	\rangle  
	= n
	\text{ and }
	l = j,
\end{equation*}
then the orbit is in $\Gr^{ \ucolambda }$. This is because the $i$th component (and $j$th in the first case) is fixed, and moreover
\begin{equation*}
	L_{k-1}
	\xrightarrow{\colambda_k}
	L_k
	\Longrightarrow
	\iota_{\wtalpha_n\hbar} (A)
	\cdot
	L_{k-1}
	\xrightarrow{\colambda_k}
	\iota_{\wtalpha+n\hbar} (A)
	\cdot
	L_k
\end{equation*}
for any $i<k\leq j$ and $A \in \SL{2}(\mathbb{C})$.

The stabilizer might fail to be greater than $\SLtwoH$, but if we require in addition
\begin{equation} \label{EqBopIsInStabilizer}
	\langle 
		\sigmapi{p}{k}, \wtalpha 
	\rangle 
	> n
	\text{ for all }
	i < k < j
\end{equation}
then $\SLtwoBop$ is in the stabilizer. Similarly for
\begin{equation} \label{EqBIsInStabilizer}
	\langle 
		\sigmapi{p}{k}, \wtalpha 
	\rangle 
	< n
	\text{ for all }
	i < k < j,
\end{equation}
we have $\SLtwoB$ is in the stabilizer.

In these cases of $\Tfixed{\colambda}$ in $\Gr^{\ucolambda}$ is $\PLine$ (or a point in a degenerated case $j = i+1$ which we ignore). Let us call this curve $C^{ij}_{p,\wtalpha+n\hbar} \subset \Gr^{\ucolambda}$. If \refEq{EqBopIsInStabilizer} holds, then the tangent vector to $C^{ij}_{p,\wtalpha+n\hbar}$ at $p$ corresponds exactly to a segment in the description of tangent vectors in $T_{p} \Gr^{\ucolambda}_{\comu}$ with weight $\wtalpha + n\hbar$. Similar for \refEq{EqBIsInStabilizer} and the weight equal to $-(\wtalpha + n\hbar)$.

The curve $C^{ij}_{p,\wtalpha+n\hbar}$ has two points. One is $p$ by construction. Another one is $q$ uniquely determined by conditions:
\begin{align*}
	\sigmapi{q}{k} 
	&=
	s_{\wtalpha+n\hbar} \sigmapi{p}{k}
	\text{ if }
	i \leq k \leq j,
	\\
	\sigmapi{q}{k} 
	&=
	\sigmapi{p}{k}
	\text{ otherwise}.
\end{align*}

Geometrically, this means that one can get the second fixed point by reflecting the part of the path between $i$th and $j$th vertex with respect to the hyperplane $\langle \bullet, \wtalpha \rangle = n$. See \refFig{FigIrredCurve} for an example. In this figure $\G = \PSL{3}$, the curve $C^{15}_{p,\wtalpha+\hbar}$ in $X = \Gr^{ \left(\coomega_2, \coomega_2, \coomega_2, \coomega_2, \coomega_2, \coomega_2 \right) }_{3\coomega_1}$ connects $p, q \in X^\T$ with $\deltap{p} = (\coomega_2,\coomega_2,\coomega_2,\coomega_2-\coomega_1, \coomega_2-\coomega_1,\coomega_2-\coomega_1)$ and $\deltap{q} = (\coomega_2,\coomega_2-\coomega_1,\coomega_2-\coomega_1,\coomega_2,\coomega_2,\coomega_2-\coomega_1)$.

\begin{figure}
	\centering
	\includegraphics{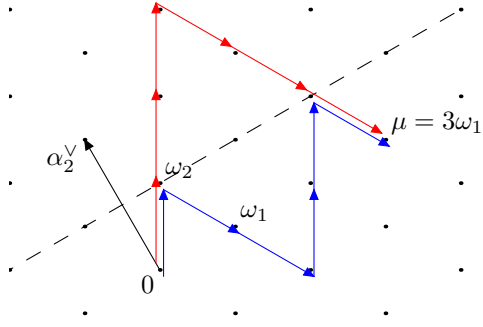}
	\caption{An example of two $\T$-fixed points connected by a curve.}
	\label{FigIrredCurve}
\end{figure}

The last thing we have to care about is if this orbit stays in $\Gr^{\ucolambda}_{\comu} \subset \Gr^{\ucolambda}$. If $j<l$ the last component $L_l$ is fixed and this is trivially true. If $j=l$ then we have three cases:
\begin{itemize}
	\item
		If $n = \langle \sigmapi{p}{j}, \wtalpha \rangle$ then $L_l$ is fixed and all $C^{ij}_{p,\wtalpha+n\hbar}$ is in $\Gr^{\ucolambda}_{\comu}$. This case is similar to the $j<l$ case.
	\item
		If $n<0$ and \refEq{EqBopIsInStabilizer} holds or $n>0$ and \refEq{EqBIsInStabilizer} holds, then the remaining unipotent radical acts on $L_l$ by a subgroup of $\GOne$. This means that open part of $C^{ij}_{p,\wtalpha+n\hbar}$ is inside $\Gr^{\ucolambda}_{\comu}$. Firstly, this part must be $\T$-invariant. Secondly, it must not contain $q$ (the fixed point of $C^{ij}_{p,\wtalpha+n\hbar}$ not equal to $p$) because the affine reflected $\comu$ is not equal to $\comu$. This means that the part of $C^{ij}_{p,\wtalpha+n\hbar}$  inside $\Gr^{\ucolambda}_{\comu}$ is $C^{ij}_{p,\wtalpha+n\hbar}\setminus \{q\}$, i.e. $\PLine \setminus \{\infty\} \simeq \ALine$.
	\item
		In other cases the remaining unipotent radical of $\SL{2}$ acts on $L_l$ moving it away from $\Gr_{\comu}$. Then $p$ is the only point of $C^{ij}_{p,\wtalpha+n\hbar}$  inside $\Gr^{\ucolambda}_{\comu}$.
\end{itemize}

\begin{remark}
	Not all $\T$-invariant curves (even irreducible ones) are of form $C^{ij}_{p,\wtalpha+n\hbar}$. However only these (and chains of them) contribute to the quantum multiplication by certain vanishing arguments.
\end{remark}

In the first two cases we'll still call the curve $C^{ij}_{p,\wtalpha+n\hbar}$ since it won't cause confusion.

\medskip

Now we claim that knowing these actions are enough to construct loci fixed by walls.

First remark in that the tangent vectors of $\T$-invariant curves $C^{ij}_{p,\wtalpha+n\hbar}$ are exactly the "segments" in the description of tangent weights. An easy check shows that all "segments" are tangent vectors to a curve of form $C^{ij}_{p,\wtalpha+n\hbar}$.

\begin{remark}
	The second case for $j=l$ (giving a $\ALine$-curve) is exactly the case when the corresponding tangent weight "segment" goes all the way to the right border and is not zero by the right boundary condition.
\end{remark}

\subsection{All \texorpdfstring{$\T$}{\textT}-invariant curves in the resolutions of slices}

The next crucial remark is that if $j_1\leq i_2$, then $\iota_{\wtalpha+n_1\hbar}^{i_1,j_1}$ and $\iota_{\wtalpha+n_2\hbar}^{i_2,j_2}$ commute. In particular, they commute for each pair of affine weights $\pm \left( \wtalpha + n\hbar \right)$ as long as \refEq{EqBIsInStabilizer} or \refEq{EqBopIsInStabilizer} are satisfied for each $i<j$ we consider. Simultaneous action on a fixed point $p \in X^{\T}$ gives an inclusion map
\begin{equation} \label{EqImagePA}
	\left(
		\PLine 
	\right)^m
	\times
	\ALine
	\hookrightarrow
	\Gr^{\ucolambda}_{\comu}
\end{equation} 
if there's a "segment" for weights $\pm \left( \wtalpha + n\hbar \right)$ going all the way to the right boundary, and
\begin{equation} \label{EqImageP}
	\left(
		\PLine 
	\right)^m
	\hookrightarrow
	\Gr^{\ucolambda}_{\comu}
\end{equation}
otherwise.

The image is a $\T$-invariant, and at a $\T$-fixed point $p$ the tangent space is equal to $T_p X^{\ker \wtalpha + n \hbar}$ (by the combinatorial description \refThm{ThmLatticeMultiplicities}; it's even more clear from the explicit description of tangent vectors in \cites{Da}). Thus the image is an open subvariety of the component of $X^{\ker \wtalpha + n \hbar}$ containing $p$. Now note that the image is proper over $\Gr^{\lambda}_{\comu}$: for \refEq{EqImageP} the image is already proper, and for \refEq{EqImagePA} the projection to $\Gr^{\lambda}_{\comu}$ is exactly the projection to $\ALine$. So the image is closed. This shows it the whole component of $X^{\ker \wtalpha + n \hbar}$. By \refLem{LemInvSubvarietyHasFixedPoints} any component of $X^{\ker \wtalpha + n \hbar}$ has this form.

The last touch is to keep track of $\T$ action on $X^{\ker \wtalpha + n \hbar}$. The weights with with $\T$ acts on tangent spaces of fixed points of $\PLine$ and $\ALine$ are $\pm (\wtalpha + n \hbar)$. 

Let $\wtchi$ be a non-zero character of $\T$. Then we denote by $\PLineWt{\wtchi}$ the projective line $\PLine$ on which $\T$ acts fixing points $0$ and $\infty$ in such a way that the weight of $T_0 \PLine$ is $\wtchi$. Similarly, we denote by $\ALineWt{\wtchi}$ the affine line $\ALine$ on which $\T$ acts fixing $0$ and scaling $T_0 \ALine$ by weight $\wtchi$.

Then the results on $\T$-walls can be summarized in the following theorem.

\begin{theorem} \label{ThmWallFixedLoci}
	\leavevmode
	\begin{enumerate}
		\item $X^{\ker \wtchi} = X^\T$ if $\wtchi$ is not a multiple of $\wtalpha + n\hbar$ for some root $\wtalpha$ and $n\in \mathbb{Z}$.

		\item
		Let $\wtchi = \wtalpha + n\hbar$, where $\wtalpha$ is a root and $n\in \mathbb{Z}$. Then the locus $X^{\ker \wtchi}$ as a $\T$-variety is isomorphic to a disjoint union of the following $\T$-varieties
		\begin{itemize}
			\item
			\begin{equation*} \label{EqInvarintSubspacePA}
				\left( \PLineWt{\wtchi} \right)^m \times \ALineWt{\wtchi}
			\end{equation*}
			if $0\leq n<\left\langle \wtalpha, \comu \right\rangle$,
				
			\item
			\begin{equation} \label{EqInvarintSubspaceP}
				\left( \PLineWt{\wtchi} \right)^m
			\end{equation}
			otherwise.
		\end{itemize}
	\end{enumerate}
\end{theorem}

\begin{proof}
Follows from preceding discussion. 
\end{proof}

\begin{remark}
In the statement of \refThm{ThmWallFixedLoci} $m=0$ is allowed, giving $\ALineWt{\wtchi}$ or a point.
\end{remark}

\begin{remark}
The number $m$ can be determined for each connected component of $X^{\ker \wtchi}$ by counting multiplicities of weights $\pm\left( \wtalpha + n\hbar \right)$ at any fixed point in this component.
\end{remark}

\begin{proposition} \label{PropCAreCoordinateLines}
    The coordinate lines in \refEq{EqInvarintSubspacePA} and \refEq{EqInvarintSubspaceP} are the curves of form $C^{ij}_{p,\wtalpha+n\hbar}$. And, conversely, any curve of form $C^{ij}_{p,\wtalpha+n\hbar}$ is a coordinate line in $X^{\ker \wtalpha+n\hbar}$.
\end{proposition}
\begin{proof}
    Follows directly from the construction.
\end{proof}

This information is enough to describe $\T$-invariant closed curves. An irreducible closed $\T$-invariant curve $C$ must be scaled with some $\T$-weight $\wtchi$ by $\T$. Then $C\subset X^{\ker \wtchi}$ and we know all the invariant curves in these simple spaces.

\begin{remark}
    Now we see that $\T$-invariant curves do in general come in families, not isolated. This contradicts one's expectation from well-known examples.
\end{remark}

\subsection{Homology}

Following Bia\l{}ynicki-Birula, let us consider the attractors in $X$ to the $\T$-fixed points. Since $X$ is not proper, action by a generic cocharacter may fail to have every point of $X$ is some of the attractors. However, is we use a cocharacter from \refProp{PropGenericContraction}, we can be sure that set-theoretically

\begin{equation*}
    X 
    = 
    \bigsqcup_{p \in X^{\T}} 
    \Attr_{\cochi}\left( p \right)
\end{equation*}

where $\Attr_{\cochi}\left( p \right)$ is the attractor to $p$ with respect to the $\Gm$-action given by (a positive power of) $\cochi$. $X$ is smooth, so by Bia\l{}ynicki-Birula Theorem we have an isomorphism of $\T$-varieties

\begin{equation*}
    \Attr_{\cochi}\left( p \right) 
    \simeq 
    \AffSpace \left( T^{+\cochi}_p X \right),
\end{equation*}
where $\AffSpace \left( T_p X\right)$ is the affine space associated to a complex vector space $T^{+\cochi}_p X \subseteq T_p X$ spanned by tangent vectors with weights positive with respect to $\cochi$. Similar to Morse theory, we define the \textbf{index} $\ind(p)$ of $p \in X^{\T}$ as 
\begin{equation*}
    \ind{p} = \dim T^{+\cochi}_p X.
\end{equation*}

Let $\overline{X} = X \cup \lbrace \infty \rbrace$ be the one-point compactification. Note that
\begin{equation*}
    \overline{X}^{\T}
    =
    X^\T \cup \lbrace \infty \rbrace.
\end{equation*}
For the sake of uniformity, we set
\begin{equation*}
    \Attr_{\cochi}\left( \infty \right)
    =
    \lbrace \infty \rbrace,
\end{equation*}
then
\begin{equation*}
    \overline{X} 
    = 
    \bigsqcup_{p \in \overline{X}^{\T}} 
    \Attr_{\cochi}\left( p \right).
\end{equation*}

This means $\overline{X}$ decomposed into subsets homeomorphic to powers of $\mathbb{C}$. It's not a CW-complex, since the closure of one cell is not, in general, contained in cells of smaller dimension. However, if we order the fixed points by attraction partial order $\lC{\wtchi}$ (see \cites{MOk}) with respect to $\wtchi$, then by definition of the order,

\begin{equation*}
    \overline{ \Attr_{\cochi}\left( p \right) }
    \subseteq
    \bigsqcup_{
    \begin{smallmatrix}
        q \in \overline{X}^{\T}
        \\
        q \leqC{\cochi} p   
    \end{smallmatrix}
    }
    \Attr_{\cochi}\left( q \right)
\end{equation*}

Pick any total order $<$ on $\overline{X}$ refining $\lC{\wtchi}$ and number the points with respect to it:
\begin{equation*}
    \infty = p_0 < p_1 < p_2 < \dots < p_m
\end{equation*}
Then we can filter $\overline{X}$
\begin{equation*} \label{EqFiltration}
    \lbrace \infty \rbrace = X_0 \subset X_1 \subset X_2 \subset \dots \subset X_m = \overline{X}
\end{equation*}
by closed (even compact since $\overline{X}$ is compact) subsets
\begin{equation*}
    X_i
    = 
    \bigsqcup_{
    \begin{smallmatrix}
        q \in \overline{X}^{\T}
        \\
        q < p_i   
    \end{smallmatrix}
    }
    \Attr_{\cochi}\left( q \right).
\end{equation*}

The filtration allow us to compute the Borel-Moore homology of $X$ in a way similar to CW case.
\begin{proposition}
    \leavevmode
    \begin{enumerate}
        \item 
            The integral Borel-Moore homology $\BMHlgy{X,\mathbb{Z}}$ is a free abelian group with even gradings only. The ranks are given by
            \begin{equation*}
                \rk \BMHlgyk{X,\mathbb{Z}}{2k} 
                = \# 
                \left\lbrace
                    p \in X^{\T}
                    \vert
                    \ind(p) = k
                \right\rbrace.
            \end{equation*}
        \item
            Moreover, $\BMHlgyk{X,\mathbb{Z}}{2k}$ has a basis
            \begin{equation*}
                A_p \in \BMHlgyk{X,\mathbb{Z}}{2\ind(p)}
            \end{equation*}
            enumerated by 
            \begin{equation*}
                p \in X^{\T},
                \quad
                \ind(p) = k,
            \end{equation*}
            such that the image of $A_p$ under the natural map
            \begin{equation*}
                \BMHlgy{X,\mathbb{Z}}
                \to
                \BMHlgy{X,X_p,\mathbb{Z}}
            \end{equation*}
            is the fundamental class
            \begin{equation*}
                \left[
                    \Attr_{\cochi}\left( p \right)
                \right]
            \end{equation*}
    \end{enumerate}
\end{proposition}

\begin{proof}
    Recall that
    \begin{equation*}
        \BMHlgy{X,\mathbb{Z}}
        =
        \Hlgy{\overline{X},\lbrace \infty \rbrace,\mathbb{Z}}.
    \end{equation*}
    The natural maps give
    \begin{align*}
        \Hlgy{X_i,X_{i-1},\mathbb{Z}}
        &\simeq
        \Hlgy{\Attr_{\cochi}\left( p_i \right),\overline{\Attr_{\cochi}\left( p_i \right)} \setminus \Attr_{\cochi}\left( p_i \right),\mathbb{Z}}
        \\
        &\simeq
        \BMHlgy{\Attr_{\cochi}\left( p_i \right),\mathbb{Z}}
        =
        \mathbb{Z}
        \left[
            \Attr_{\cochi}\left( p_i \right)
        \right].
    \end{align*}
    The initial step for us is
    \begin{equation*}
        \Hlgy{X_0,\lbrace \infty \rbrace,\mathbb{Z}} = 0
    \end{equation*}
    The final step is applying the spectral sequence to the filtration \refEq{EqFiltration} (or the exact sequence of triple inductively on $p_i \in \overline{X}$). Using, in addition, that
    \begin{itemize}
        \item 
            homology only in even degrees gives that there is no cancellation by differentials,
        \item
            an extension of a free abelian group is trivial
    \end{itemize}
    we get the statement of the proposition.
\end{proof}

Now we use that the Borel-Moore homology is dual to the ordinary homology. I.e. there is a perfect pairing
\begin{equation*}
    \BMHlgyk{X,\mathbb{Z}}{2\dim X - k} 
    \times 
    \Hlgyk{X,\mathbb{Z}}{k} 
    \to \mathbb{Z}
\end{equation*}
given by the intersection number. This allows us to analyze the second homology.

\begin{proposition} \label{PropCClassesGenerateH2}
	\leavevmode
	\begin{enumerate}
	\item
		The curves of form $C^{ij}_{p,\wtalpha+n\hbar}$ generate $\Hlgyk{X,\mathbb{Z}}{2}$ as an abelian group.
	\item
		The curves of form $C^{ij}_{p,\wtalpha+n\hbar}$ generate the submonoid $\HlgyEffk{X,\mathbb{Z}}{2} \subset \Hlgyk{X,\mathbb{Z}}{2}$ of effective curve classes.
	\end{enumerate}
\end{proposition}

\begin{proof}
    Take $p \in X^{\T}$ with $\ind(p) = \dim X - 1$. The attractor with respect to $\cochi$ to $p$ has dimension $\dim X - 1$, so the attractor $C_p$ with respect to $-\cochi$ has dimension $1$. Let's call the corresponding weight $\wtalpha + n \hbar$. Then $C_p$ is in the component of $X^{\ker \wtalpha + n \hbar}$ containing $p$. It's an open subset of a line in \ref{EqInvarintSubspacePA} or \ref{EqInvarintSubspaceP}. Moreover, because of the construction of $\cochi$ the closure is proper, so
    \begin{equation*}
        \overline{C}_p 
        = 
        C_p \cup \lbrace q\rbrace,
    \end{equation*}
    $q \in X^\T$. By the construction of the attraction order, $p \lC{\cochi} q$. Finally, since $\overline{C}_p$ is a coordiante line, \refProp{PropCAreCoordinateLines} gives us that 
    \begin{equation*}
        \overline{C}_p 
        = 
        C^{ij}_{p,\wtalpha+n\hbar} \cup \lbrace q\rbrace,
    \end{equation*}
    for some $i,j$.

    Now consider all $\overline{C}_p$ at the same time with the basis $A_p$, all points have index $\dim X - 1$. We have
    \begin{equation*}
        A_{p}
        \cup
        \overline{C}_p 
        = 1
    \end{equation*}
    and 
    \begin{equation*}
        A_{p}
        \cup
        \overline{C}_q 
        = o
    \end{equation*}
    for all $p \lC{\cochi} q$.
    The intersection pairing with these classes is triangular with $1$s on the diagonal. The pairing is prefect, so if $A_p$ is a basis, then $\overline{C}_q$ also form a basis. Using that $\overline{C}_q$ have form $C^{ij}_{p,\wtalpha+n\hbar}$ finishes the proof.
\end{proof}

\begin{proposition} \label{PropCurveDegrees}
	Let $C$ be an irreducible $\T$-invariant curve connecting $p,q \in X^{\A}$. Let $\A$ scale $T_p C$ with weight $\wtalpha$. Then the degree of $\LBundle{i}$ restricted to $C$ is
	\begin{equation*}
		\int\limits_{[C]} \ChernL{}{i} 
		=
		\dfrac
		{
			\CorootScalar
			{
				\sigmapi{p}{i} - \sigmapi{q}{i}
			}
			{
				\coalpha
			}
		}
		{
			2
		}.
	\end{equation*}
\end{proposition}

\begin{proof}
	By $\A$-equivariant localization and \refProp{PropLWeights} we have
	\begin{align*}
		\int\limits_{[C]} \ChernL{}{i} 
		&=
		\dfrac
		{
			\resZ
			{
				\ChernL{\A}{i} 
			}
			{
				p
			}
		}
		{
			\wtalpha
		}
		+
		\dfrac
		{
			\resZ
			{
				\ChernL{\A}{i} 
			}
			{
				q
			}
		}
		{
			-\wtalpha
		}
		\\
		&=
		\dfrac
		{
			\CorootScalar
			{
				\sigmapi{p}{i}
			}
			{
				\bullet
			}
			-
			\CorootScalar
			{
				\sigmapi{q}{i}
			}
			{
				\bullet
			}
		}
		{
			\wtalpha
		}.
	\end{align*}
	
	By dimension argument and using that the Chern classes are in non-localized cohomology, we have that this is in $\EqCoHlgyk{\A}{pt, \mathbb{Q}}{0} = \mathbb{Q}$. To find this number, we can pair both numerator and denominator with $\coalpha$. This gives the statement of the Corollary.
\end{proof}

Recall that $\CoHlgyk{X,\mathbb{Q}}{2}$ is generated by the first Chern classes by \refProp{PropSecondCohomologyGeneration}. This means that the linear map
\begin{align*}
	\mathbb{Q}^l
	&\twoheadrightarrow
	\CoHlgyk{X,\mathbb{Q}}{2}
	\\
	\left(
		z_1, \dots, z_l
	\right)
	&\mapsto
	\sum_i z_i \ChernE{}{i}.
\end{align*}

is surjective. Dualizing, we get an injection
\begin{equation} \label{EqDefinitionOfEClasses}
	\Hlgyk{X,\mathbb{Q}}{2}
	\hookrightarrow
	\mathbb{Q}^l
\end{equation}
We call the standard basis in the target $e_1, \dots, e_l$. By construction,
\begin{equation*}
    \left\langle
        \ChernE{}{i},
        e_j
    \right\rangle
    =
    \delta_{i,j}
\end{equation*}
where the RHS is the Kronecker delta.

The image in \refEq{EqDefinitionOfEClasses} is determined by the relations on Chern classes. For example, the relation
\begin{equation*}
    0 = \ChernL{}{l} = \sum_{i=1}^l \ChernE{}{i}
\end{equation*}
gives a relation for functions on $\CoHlgyk{X,\mathbb{Q}}{2}$
\begin{equation*}
    \sum_{i=1}^l 
    \Big. 
        e_i
    \Big\vert_{\CoHlgyk{X,\mathbb{Q}}{2}}
    = 0
\end{equation*}
which is a standard relation on weights in $A_l$ root system.

Later we identify the homology classes with a linear combination of $e_i$ in the image under \refEq{EqDefinitionOfEClasses}. However, one should keep in mind that there might be relations between $e_i$'s.

\begin{proposition} \label{PropCHomologyClasses}
	The homology class of $C^{ij}_{p,\wtalpha + n\hbar}$ is 
	\begin{equation*}
	\left[
		C^{ij}_{p,\wtalpha + n\hbar}
	\right]
	=
	\dfrac
	{
		\CorootScalar{\coalpha}{\coalpha}
	}
	{
		2
	}
	\sum_{k=i+1}^{j}
	\left\langle
		\deltapi{p}{k}, \wtalpha 
	\right\rangle
	e_{k}
	\end{equation*}
\end{proposition}
\begin{proof}
	From \refCor{PropCurveDegrees} we get
	\begin{equation*}
		\int\limits_
		{
			\left[
				C^{ij}_{p,\wtalpha + n\hbar}
			\right]
		}
		\ChernL{}{k} 
		=
		\left[
			\left\langle
				\sigmapi{p}{k}, \wtalpha 
			\right\rangle
			- n
		\right]
		\dfrac
		{			
			\CorootScalar
			{
				\coalpha
			}
			{
				\coalpha
			}
		}
		{
			2
		}
	\end{equation*}
	for $i<k<j$ and
	\begin{equation*}
		\int\limits_
		{
			\left[
				C^{ij}_{p,\wtalpha + n\hbar}
			\right]
		}
		\ChernL{}{k} 
		=
		0
	\end{equation*}
	otherwise.
	
	Then passing from $\LBundle{k}$ to $\EBundle{k}$ we get
	\begin{equation*}
		\int\limits_
		{
			\left[
				C^{ij}_{p,\wtalpha + n\hbar}
			\right]
		}
		\ChernE{}{k} 
		=
		\left\langle
			\deltapi{p}{k}, \wtalpha 
		\right\rangle
		\dfrac
		{
			\CorootScalar
			{
				\coalpha
			}
			{
				\coalpha
			}
		}
		{
			2
		}
	\end{equation*}
	for $i<k\leq j$ and
	\begin{equation*}
		\int\limits_
		{
			\left[
				C^{ij}_{p,\wtalpha + n\hbar}
			\right]
		}
		\ChernE{}{k} 
		=
		0
	\end{equation*}
	otherwise.
	Since pairing with $\ChernE{}{k}$ gives coordinates in $e_k$ by definition of $e_k$, we get the statement of the Proposition.
\end{proof}

Now we can explicitly find what the effective cone is in term of $e_k$.

\begin{corollary}
	\begin{equation*}
		\HlgyEffk{X,\mathbb{Z}}{2} 
		=
		\linspan_{\mathbb{Z}_{\geq 0} }
		\left\lbrace
			e_i - e_j
			\left\vert
				i<j
			\right.
		\right\rbrace
		\cap
		\Hlgyk{X,\mathbb{Z}}{2}.
	\end{equation*}
\end{corollary}
\begin{proof}
	By \refProp{PropCClassesGenerateH2} we know that the effective cone is generated by classes $ \left[ C^{ij}_{p,\wtalpha + n\hbar} \right] $. The classes in \refProp{PropCHomologyClasses} are sums of $e_{k+1}-e_{k}$ with positive coefficients, because
	\begin{equation*}
		\sum_{k=i+1}^j
		\left\langle
			\deltapi{p}{k}, \wtalpha 
		\right\rangle
		= 
		\left\langle
			\sigmapi{p}{j}, \wtalpha 
		\right\rangle
		-
		\left\langle
			\sigmapi{p}{i}, \wtalpha 
		\right\rangle
		=
		0
	\end{equation*}
	and
	\begin{equation*}
		\left\langle
			\sigmapi{p}{k}, \wtalpha 
		\right\rangle
		>
		\left\langle
			\sigmapi{p}{j}, \wtalpha 
		\right\rangle
		=
		\left\langle
			\sigmapi{p}{i}, \wtalpha 
		\right\rangle
	\end{equation*}
	for all $i<k<j$.
	
	The coefficients are in $\mathbb{Z}$ because
	\begin{equation*}
		\dfrac
		{
			\CorootScalar
			{
				\coalpha
			}
			{
				\coalpha
			}
		}
		{
			2
		}
		\in
		\mathbb{Z}
	\end{equation*}
	and
	\begin{equation*}
		\left\langle
			\deltapi{p}{k}, \wtalpha 
		\right\rangle
		\in
		\mathbb{Z}.
	\end{equation*}
	
	This proves the statement.
\end{proof}

\begin{remark}
	The classes $\ChernL{}{i}$ are closely related to the fundamental weights $\omega_i$ in the root system $A_l$. Then the classes $\ChernE{}{i}$ are relatives of $e_i$ in the same root system, that is the classes for which the roots are written as $e_i-e_j$. By "related" here we mean that in the generic case the cohomology classes do form a root system of type $A_l$ (there is only one relation) and these cohomology classes match with the weights under this identification.
\end{remark}

\section{Deformations of symplectic resolutions}
\label{AppendixB}

\subsection{Universal deformations}

Let us recall some basics on symplectic resolutions and their deformations. See \cite{K1} for more details.

A \textbf{symplectic resolution} is a smooth algebraic variety $X$ equipped with a closed non-degenerate 2-form $\omega$ such that the canonical map
\begin{equation*}
	X \to X_0 = \Spec \CoHlgyk{X,\OO_X}{0}
\end{equation*}
is a birational projective map.

Well-known examples of the symplectic resolutions are Du Val singularities, the Hilbert scheme points on a plane (or a Du Val singularity), cotangent bundles to flag varieties, Nakajima quiver varieties. As we mention in the main part of the paper, it's known when the transversal slices in the affine Grassmannian admit a symplectic resolution.

$X_0$ carries a natural structure of (possible singular) Poisson affine scheme. One can pose a question if affine Poisson varieties $\Gr^{\lambda}_{\comu}$ we study admit a symplectic resolution, i.e. a smooth symplectic $X$ with $X_0 = \Gr^{\lambda}_{\comu}$ such that the assumptions above hold. A result in \cite{KWWY} answers this question and we present their result later. Before it let us discuss properties of the symplectic resolutions in general.

\medskip

One rich source of symplectic varieties is the symplectic reduction. Let $\left( Z, \omega \right)$ be a symplectic variety and let a reductive group $\G$ have a Hamiltonian action on it. I.e. let there be a map
\begin{equation*}
	\mu
	\colon
	Z
	\to
	\mathfrak{g}^*,
\end{equation*}
where $\mathfrak{g} = \Lie \G$ and $\mu$ defines the infinitesimal action of $\mathfrak{g}$ by vector fields $\partial_\xi$, $\xi\in \mathfrak{g}$, via the Hamilton equations
\begin{equation*}
	\omega
	\left(
		\partial_\xi,
		\bullet
	\right)
	=
	d 
	\left\langle
		\mu,
		\xi
	\right\rangle.
\end{equation*}

The standard extra condition is that $\mu$ is $\G$-equivariant. The $\G$-fixed locus in $\mathfrak{g}^*$ is $\mathfrak{z}^* = \left[ \mathfrak{g}, \mathfrak{g} \right]^\perp$ (in particular, if $\G$ is a torus, then $\mathfrak{z}^* = \mathfrak{g}^*$). For any $\zeta \in \mathfrak{z}^*$ we have a $\G$-invariant subvariety
\begin{equation*}
	\mu^{-1}\left( \zeta \right)
	\subset
	Z.
\end{equation*}
If the quotient $X_\zeta = \mu^{-1}\left( \zeta \right)/\G$ exists (in GIT $Z$ is replaced by the semistable locus $Z^{ss}$ to achieve this), then it is naturally symplectic. Thus we have constructed a family of varieties over $\mathfrak{z}^*$
\begin{equation*}
	X \to \mathfrak{z}^*
\end{equation*}
which is fiberwise symplectic, with a fiber $X_\zeta$ over $\zeta$. Each fiber $X_\zeta$ is called the symplectic reduction of $Z$.

\medskip

The symplectic reduction with respect to a torus can be in certain sense inverted \cite{K1}. That is, one can reconstruct the family from a fiber if it is a symplectic resolution. The contruction is called twistor deformation. 

Let $\left( X, L \right)$ be a pair containing a symplectic resolution $X$ and a line bundle $L$ over it. Fix a connected subscheme $S\subset \ALine$ containing $0\in \ALine$. Then a twistor deformation $Z$ over the base $S$ is the pair $\left( \widetilde{X}, \widetilde{L} \right)$ over $S$ with the fiber $\left( X, L \right)$ over $0$ and a symplectic form $\omega_Z$ on the principal $\Gm$-bundle $Z\to \widetilde{X}$ associated to $\widetilde{L}$, that is $\Gm$-invariant, and the projection $Z\to \widetilde{X} \to S$ is (the base change from $\Lie \Gm \simeq \ALine$ to $S$ of) the moment maps for the $\Gm$-action on $Z$. In this case $X$ is the symplectic reduction of $Z$ as a fiber over $0$.

A theorem by Kaledin \cite{K3}, this deformation exists for complex varieties (even more generally, in characteristic $0$ case) formally locally, that is for $S = \Spec \mathbb{C}[[z]]$. It can be extended to the whole line $\ALine$.

The fiber $\widetilde{X}_s$ over a point $s\in S$ gets a symplectic form $\omega_s$ whose class in de Rham cohomology $\left[ \omega_s \right] \in \CoHlgyk{\widetilde{X}_s}{2}$ is $\left[\omega_0\right] + s \cdot \Chern{}{L} \in \left[ \omega_s \right] \in \CoHlgyk{X}{2}$ under identification $\CoHlgy{\widetilde{X}_s} \simeq \CoHlgy{X}$ via Guass-Manin connection.

One can do twistor deformations with respect to different line bundles. This gives the universal family of deformations \cite{KV} of $X$ as a symplectic variety:
\begin{equation*}
	\begin{tikzcd}
		X \arrow[r,hook]\arrow[d] & \widetilde{X}^{u} \arrow[d]
		\\
		0 \arrow[r,hook] & \mathcal{B}^{u}
	\end{tikzcd}
\end{equation*}
over the base which is a subscheme of the affine space associated to $\PicX{X} \otimes_\mathbb{Z} \mathbb{C} \simeq \CoHlgyk{X,\mathbb{C}}{2}$. $\mathcal{B}^{u}$ contains at least the formal neighborhood of $0$. In particular, the space of deformations of $\left( X, \omega \right)$ as a symplectic variety is finite-dimensional in contrary to the deformations of $X$ as a variety.

The map $\widetilde{X}^{u} \to \mathcal{B}^{u}$ is called the period map. Similarly to the case of $1$-parametric deformation, it sends the fiber $\widetilde{X}^{u}_s$ with a symplectic structure $\omega_s$ the de Rham class $\left[ \omega_s \right] \in \CoHlgyk{X}{2}$.

Now we are ready to describe the wall structure in $\CoHlgyk{X}{2}$. Fix a point $s \in \CoHlgyk{X}{2}$ and look at the symplectic variety in the fiber $\widetilde{X}^{u}_s$ with the natural algebraic symplectic form $\omega_s$. Let $C \subset \widetilde{X}^{u}_s$ be a smooth (algebraic, not topologic) curve. Then the restriction of $\omega_s$ to $C$ as a real manifold is trivial because $\omega_s$ is $\mathbb{C}$-linear and the tangent space to $C$ is $\mathbb{C}$-dependent. This gives
\begin{equation*}
	0
	=
	\int_C \omega_s
	=
	\left\langle
		\left[
			C
		\right],
		s
	\right\rangle,
\end{equation*}
where the last expression is the pairing of a homology class with a cohomology class. This gives that a class $\dd \in \Hlgyk{X,\mathbb{Z}}{2}$ can be represented by an algebraic curve only in the fibers over such $s \in \CoHlgyk{X,\mathbb{C}}{2}$ that
\begin{equation*}
	\left\langle
		\dd,
		s
	\right\rangle
	=
	0.
\end{equation*}
In other words, $s$ must lie on the hyperplane defined by $\dd$.

This implies that for a generic $s \in \CoHlgyk{X}{2} $ the fiber $\widetilde{X}^{u}_s$ has no algebraic curves.

This wall structure will be the key ingredient in the computation of the purely quantum part of the multiplication.


The resolutions of our interest $\Gr^{\ucolambda}_{\comu} \to \Gr^{\colambda}_{\comu}$ are conical, i.e. that there is a $\LoopGm$-action on $X$ such that the $X_0$ is scaled to a point by the induced $\LoopGm$-action. This allows to show that the universal deformation exists not only in a formal neighborhood, but in the whole affine space $\AffSpace \left( \CoHlgyk{X,\mathbb{C}}{2} \right)$.


\subsection{Beilinson-Drinfeld Grassmannian}

By definition, the affine Grassmannian is the moduli space of local data in a formal neighbourhood of a point. One can allow this point to vary and even have several such points. This gives rise to a moduli space called the Beilinson-Drinfeld affine Grassmannian \cite{BD}.

Let $\left( t_1, \dots, t_l \right) \in \AffSpace^l$ be $l$-tuple of points in $\ALine$. Then we can define the following moduli space similar to \refEq{EqGrAsModuli}.
\begin{equation*}
	\GrBD_
	{
		\left(
			t_1,
			\dots,
			t_l
		\right)
	}
	=
	\left\lbrace
		\left(
			\mathcal{P}, 
			\varphi 
		\right)
		\left\vert 
			\begin{matrix}
				\mathcal{P}
				\text{ is a }
				\G
				\text{-principle bundle on }
				\ALine,
				\\
				\varphi
				\colon 
				\mathcal{P}_0\vert_
				{
					\ALine 
					\setminus
					\left\lbrace
						t_1,
						\dots,
						t_l
					\right\rbrace
				} 
				\xrightarrow{\sim} 
				\mathcal{P}\vert_
				{
					\ALine 
					\setminus
					\left\lbrace
						t_1,
						\dots,
						t_l
					\right\rbrace
				}
				\\
				\text{ is a trivialization of }
				\mathcal{P}
				\text{ over }
				\ALine 
				\setminus
				\left\lbrace
					t_1,
					\dots,
					t_l
				\right\rbrace
			\end{matrix}
		\right.
	\right\rbrace
\end{equation*}
Varying $\left( t_1, \dots, t_l \right)$ we get a family over $\AffSpace^l$
\begin{equation*}
	\begin{tikzcd}
		\GrBD_
		{
			\left(
				t_1,
				\dots,
				t_l
			\right)
		} \arrow[r,hook] \arrow[d]
		&
		\GrBD \arrow[d]
		\\
		\left(
			t_1,
			\dots,
			t_l
		\right) \arrow[r,hook]
		&
		\AffSpace^l
	\end{tikzcd}
\end{equation*}
with a fiber $\GrBD_{ \left( t_1, \dots, t_l \right) }$. The family $\GrBD$ is called the Beilinson-Drinfeld Grassmannian.

If we set $t_1 = \dots = t_l = 0$, then $\GrBD_{ \left( 0, \dots, 0 \right) } \simeq \Gr$, by the Beauville-Laszlo Theorem\cite{BL} which allows us to pass to the formal neighborhoods. So the Beilinson-Drinfeld Grassmannian $\GrBD$ is a deformation of the usual affine Grassmannian $\Gr$.

The family $\GrBD$ is fiberwise Poisson. The Poisson structure again comes from consideration of Manin triples and can be found in \cite{EK}.

Here we discuss several deformations of familiar spaces associated to the affine Grassmannian $\Gr$. All of them will be deformations over $\AffSpace^l$ obtained by varying points $t_1, \dots, t_l$. We will use notation of form $\widetilde{X}$ for the family over $\AffSpace^l$ and $\widetilde{X}_{ \left( t_1, \dots, t_l \right) }$ for the fiber over $\left( t_1, \dots, t_l \right)$.

Now let us define a deformation of $\Gr_{\comu}$. It was defined as on orbit of $\Tfixed{\comu}$ under the action of $\GOne \subset \GK$. In geometric terms the action of $\GK$ meant a change of trivialization $\varphi$ by precomposing with a automorphisms of $\mathcal{P}_0\vert_{\PDisk}$, the subgroup $\GOne$ is the subgroup which doesn't change the behaviour of $\varphi$ at $\infty$. Informally, $\Gr_{\comu}$ is the moduli space of $\left( \mathcal{P}, \varphi \right)$ such that the trivialization $\varphi$ behaves like $\Tcowt{t}{\comu}$ as $t \to \infty$.

The analogue of $\GK$ is $\GRing{ R_{\left( t_1, \dots, t_l \right)} }$, where
\begin{equation*}
	R_
	{
		\left( 
			t_1,
			\dots, 
			t_l 
		\right)
	}
	=
	\mathbb{C}
	\left[
		t, 
		\dfrac{1}{t-t_1},
		\dots,
		\dfrac{1}{t-t_l}
	\right]
\end{equation*}
are the global functions on $\ALine \setminus \left\lbrace t_1, \dots, t_l \right\rbrace$. Then the analogue if $\GOne$ is $\GOneGlob{\left( t_1, \dots, t_l \right)}$, the subgroup of all elements in $\GRing{ R_{\left( t_1, \dots, t_l \right)} }$ whose limit as $t \to \infty$ exists and is $\UnitG$:
\begin{equation*}
	\GOneGlob{\left( t_1, \dots, t_l \right)}
	= 
	\left\lbrace 
		g(t)
		\in
		\GRing
		{ 
			R_
			{
				\left( 
					t_1, 
					\dots, 
					t_l 
				\right)
			} 
		}
		\bigg\vert \; 
		\lim_{t\to \infty} g(t) 
		= 
		\UnitG 
	\right\rbrace.
\end{equation*}
The analogue of $\Tfixed{\comu}$ is the equivalence class $\Tcowt{t-t_1}{\comu} \in \GRing{ R_{\left( t_1, \dots, t_l \right)} }$ under an identification similar to $\GK / \GO$ for $\Gr$. Note that the $\GOneGlob{\left( t_1, \dots, t_l \right)}$ doesn't change if we choose $\Tcowt{t-t_i}{\comu}$ instead of $\Tcowt{t-t_1}{\comu}$, because $\Tcowt{t-t_i}{\comu}\Tcowt{t-t_1}{-\comu} \in \GOneGlob{\left( t_1, \dots, t_l \right)}$.
The $\GOneGlob{\left( t_1, \dots, t_l \right)}$-orbit of $\Tcowt{t-t_1}{\comu}$ defines $\GrBD_{\comu, \left( t_1, \dots, t_l \right) } \subset \GrBD_{ \left( t_1, \dots, t_l \right) }$. This gives a family
\begin{equation*}
	\GrBD_{\comu}
	\to
	\AffSpace^l
\end{equation*}
with $\GrBD_{\comu, \left( 0, \dots, 0 \right)} \simeq \Gr_{\comu}$, so it is a deformation of $\Gr_{\comu}$.

Now we can discuss how $\Gr^{\colambda}$ deforms in this way. It was a condition that at $t=0$ the singularity is "not worse" than $\Tcowt{t}{\colambda}$. For every point $t \in \ALine$ we can do the natural restriction of $\left( \mathcal{P}, \varphi \right)$ to the formal disks about $t$
\begin{equation*}
	\rho_t
	\colon
	\GrBD_{\left( t_1, \dots, t_l \right)}
	\to
	\Gr
\end{equation*}
By definition this is trivial if $t \neq t_i$ for all $i$, $1\leq i \leq l$. For $t = t_i$ this allows to look at the local behavior of $L \in \GrBD_{\left( t_1, \dots, t_l \right)}$ by
\begin{equation*}
	\rho_{t_i}
	\left(
		L
	\right)
	\in
	\overline{\Gr^{\colambda}}.
\end{equation*}
In this case we say that $L$ has Hecke type $\leq \colambda$ at $t_i$.

One feature of $\GrBD$ is that points $t_i$ can split and merge as we move along $\AffSpace^l$. We want the Hecke type to add whenever points merge. So we assign to each $i$, $1\leq i \leq l$, a dominant weight $\colambda_i$ and consider a formal linear combination $\sum_i \colambda_i t_i$ with coefficients in cocharacters $\CocharLattice{\A}$. One can think of this as a divisor colored in cocharacters $\CocharLattice{\A}$. This divisor is "effective", i.e. all cocharacters are dominant. For $\G$ of rank $1$ this will be indeed an effective divisor.
We can sum the coefficients of the same point:
\begin{equation*}
	\colambda(t) 
	= 
	\sum\limits_{i,\: t_i = t}
	\colambda_i.
\end{equation*}
and write
\begin{equation*}
	\sum_i 
	\colambda_i t_i
	=
	\sum_{t\in \ALine}
	\colambda(t) t.
\end{equation*}

Then similarly to the case of usual $\Gr$ we write
\begin{equation*}
	L_1
	\xrightarrow{\sum_i \colambda_i t_i}
	L_2
\end{equation*}
for $L_1, L_2 \in \GrBD$ if
\begin{equation*}
	\rho_{t}
	\left(
		L_1
	\right)
	\xrightarrow{\colambda(t)}
	\rho_{t}
	\left(
		L_2
	\right)
\end{equation*}
for all $t \in \ALine$.

Then this defines $\GrBD^{\sum_i \colambda_i t_i}_{\left( t_1, \dots, t_l \right)} \subset \GrBD_{\left( t_1, \dots, t_l \right)}$ as
\begin{equation*}
	\GrBD^
	{
		\sum_i 
		\colambda_i t_i
	}		
	_
	{
		\left(
			t_1,
			\dots,
			t_l
		\right)
	}
	=
	\left\lbrace
		L \in \GrBD_{\left( t_1, \dots, t_l \right)}
		\left\vert 
			L_0
			\xrightarrow{\sum_i \colambda_i t_i}
			L
		\right.
	\right\rbrace,
\end{equation*}
where $L_0 \in \GrBD_{{\left( t_1, \dots, t_l \right)}}$ is the trivial element, the one with trivial $\G$-bundle and the identity as a trivialization.
 
As previously, one can check that it defines a flat family over $\AffSpace^l$
\begin{equation*}
	\GrBD^
	{
		\sum_i 
		\colambda_i t_i
	}
	\to
	\AffSpace^l
\end{equation*}
which is a deformation of $\Gr^{\sum \colambda_i}$.

The fibers $\GrBD^{\sum_i \colambda_i t_i}_{\left( t_1, \dots, t_l \right)}$ split into a product
\begin{equation} \label{EqBDGrSplits}
	\GrBD^{\sum_i \colambda_i t_i}_{\left( t_1, \dots, t_l \right)}
	\simeq
	\prod_{t \in \ALine}
	\Gr^{\colambda (t)}.
\end{equation}
All but finitely many factors are points, so this space is well-defined.

We are interested in the transversal slices $\Gr^{\colambda}_{\comu}$. Their deformation is
\begin{equation} \label{EqSliceDeformation}
	\GrBD^
	{
		\sum_i 
		\colambda_i t_i
	}
	_{\comu}
	=
	\GrBD^
	{
		\sum_i 
		\colambda_i t_i
	}
	\cap
	\GrBD_{\comu}
\end{equation}
for any splitting of $\colambda$ into a sum of dominant weights $\colambda_i$.

The fibers of \refEq{EqSliceDeformation} have properties, similar to the ones of $\Gr^{\colambda}_{\comu}$: they are affine, normal, Poisson. If all $\colambda_i$ are minuscule (or zero), then a generic fiber of \refEq{EqSliceDeformation} is smooth. The singular fibers (in particular, $\Gr^{\colambda}_{\comu}$) may appear only when $t_i = t_j$ for some $i$ and $j$. In other words, we have walls
\begin{equation*}
	H_{ij}
	=
	\left\lbrace
		t_i = t_j
	\right\rbrace
\end{equation*}
over which the fibers of \refEq{EqSliceDeformation} might be singular.

We can make a simultaneous resolution of \refEq{EqSliceDeformation} similar to resolutions of $\Gr^{\colambda}_{\comu}$. Let $\colambda_1, \dots, \colambda_l$ be an $l$-tuple of dominant minuscule weights which sum to $\colambda$ (assuming it exists). Look at the moduli space  
\begin{equation*}
	\GrBD^
	{
		\left(
			\colambda_1,
			\dots,
			\colambda_l
		\right)
	}		
	_
	{
		\left(
			t_1,
			\dots,
			t_l
		\right)
	}
	=
	\left\lbrace 
		\left(
			L_1,\dots,L_l
		\right) 
		\in 
		\GrBD^{\times l}_
		{
		\left(
			t_1,
			\dots,
			t_l
		\right)
		}
		\left\vert
			L_0
			\xrightarrow{\colambda_1 t_1}
			L_1 
			\xrightarrow{\colambda_2 t_2} 
			\dots 
			\xrightarrow{\colambda_{l-1} t_{l-1}}
			L_{l-1} 
			\xrightarrow{\colambda_{l} t_l}
			L_l
		\right.
	\right\rbrace
\end{equation*}
The family is
\begin{equation*}
	\GrBD^
	{
		\left(
			\colambda_1,
			\dots,
			\colambda_l
		\right)
	}
	\to
	\AffSpace^l
\end{equation*}
is a deformation of $\Gr^{\left( \colambda_1, \dots, \colambda_l \right)}$. All fibers are smooth.

The fibers $\GrBD^{ \left( \colambda_1, \dots, \colambda_l \right) }_{ \left( t_1, \dots, t_l \right) }$ have a product structure similar to \refEq{EqBDGrSplits}:
\begin{equation} \label{EqBDResGrSplits}
	\GrBD^
	{
		\left(
			\colambda_1,
			\dots,
			\colambda_l
		\right)
	}		
	_
	{
		\left(
			t_1,
			\dots,
			t_l
		\right)
	}
	\simeq
	\prod_{t \in \ALine}
	\Gr^{\ucolambda (t)},
\end{equation}
where $\ucolambda (t)$ is the subsequence of $\left( \colambda_1, \dots, \colambda_l \right)$ of all $\colambda_i$ such that $t_i = t$. In particular, over a generic point of a wall $H_{ij}$ (i.e. when $t_i=t_j$ and all other $t_k$ are distinct) we have the following fiber
\begin{equation*}
	\GrBD^
	{
		\left(
			\colambda_1,
			\dots,
			\colambda_l
		\right)
	}		
	_
	{
		\left(
			t_1,
			\dots,
			t_l
		\right)
	}
	=
	\Gr^{\colambda_i,\colambda_j}
	\times
	\prod_
	{
		\begin{smallmatrix}
			k\neq i
			\\
			k \neq j
		\end{smallmatrix}
	}
	\Gr^{\colambda_k}
\end{equation*}

There is a natural map which takes the last element $L_l$
\begin{equation}
	\widetilde{m}_
	{
		\left(
			\colambda_1,
			\dots,
			\colambda_l
		\right)
	}
	\to
	\GrBD^
	{
		\left(
			\colambda_1,
			\dots,
			\colambda_l
		\right)
	}
	\to
	\GrBD^
	{
		\sum_i 
		\colambda_i t_i
	}.
\end{equation}
Restricting it to a subspace
\begin{equation*}
	\GrBD^
	{
		\left(
			\colambda_1,
			\dots,
			\colambda_l
		\right)
	}
	_{\comu}
	=
	\widetilde{m}_
	{
		\left(
			\colambda_1,
			\dots,
			\colambda_l
		\right)
	}
	^{-1}
	\left(
		\GrBD_{\comu}
	\right)
\end{equation*}
we get a natural morphism
\begin{equation} \label{EqBDResolution}
	\GrBD^
	{
		\left(
			\colambda_1,
			\dots,
			\colambda_l
		\right)
	}
	_{\comu}
	\to
	\GrBD^
	{
		\sum_i 
		\colambda_i t_i
	}
	_{\comu}
\end{equation}
is the fiberwise affinization. So it's proper, birational and surjective and deforms the familiar resolution of slices
\begin{equation*}
	\Gr^
	{
		\left(
			\colambda_1,
			\dots,
			\colambda_l
		\right)
	}
	_{\comu}
	\to
	\Gr^{\sum \colambda_i}_{\comu}.
\end{equation*}
Since any fiber $\GrBD^{\sum_i \colambda_i t_i}_{\comu, \left( t_1, \dots, t_l\right)}$ over the complement of walls $H_{ij}$ is smooth, the map is an isomorphism in fiber over a generic point. So, outside of the planes $H_{ij}$ the spaces $\GrBD^{ \left( \colambda_1, \dots, \colambda_l \right) }_{ \comu, \left( t_1, \dots, t_l \right) }$ don't have any curves.

These constructions can be equipped with $\T$ action. $\A$ acts on these spaces by changing the trivializations by global multiplication by $
\A$. So it's a fiberwise action. Let $\LoopGm$ scale all $t_i$ simultaneously with weight $\hbar$. This acts on the whole families and $\AffSpace^l$, making all morphisms coming from
\begin{equation*}
	\GrBD \to \AffSpace^l
\end{equation*}
$\LoopGm$-equivariant. 

These $\A$-action and $\LoopGm$-action commute, so we can say that we have a $\T$-action. It specializes to the one defined before over $\left( t_1, \dots, t_l \right) = \left( 0, \dots, 0 \right)$, the only $\LoopGm$-fixed point of $\AffSpace^l$.

In this section we summarize essential facts about $A$-fixed points of the Beilinson-Drinfeld deformation
\begin{equation*}
	\Gr^
	{
		\left( 
			\colambda_1, 
			\dots, 
			\colambda_l 
		\right)
	}
	_{\comu}
	\to
	\AffSpace^l
\end{equation*}
and, more importantly, relation of this deformation to the bundles $\EBundle{i}$.

Let us start with the $\A$-fixed points. The $\A$-action is fiberwise, so we can compare the $\A$-fixed points for different fibers.

The family of $\A$-fixed points 
\begin{equation}
	\left( 
		\Gr^
		{
			\left( 
				\colambda_1, 
				\dots, 
				\colambda_l 
			\right)
		}
		_{\comu}
	\right)^{\A} 
	\to
	\AffSpace^l
\end{equation}
is a finite flat morphism, so it makes sense to say that $\A$-fixed points in one fiber correspond to $\A$-fixed points in another fiber. Often we use this fact by saying that an $\A$-fixed point $p'$ in some fiber has a limit $p$ in the fiber $\Gr^{\ucolambda}_{\comu}$ over $\left( 0, \dots, 0 \right)$. The $\A$-weights of the tangent bundle ant these points are the same. Finally, under the isomorphism \refEq{EqBDResGrSplits} there is a description of the $\A$-fixed points, similar to the case of $\Gr^{\ucolambda}_{\comu}$.

\medskip

Now we introduce the deformations of line bundles $\EBundle{i}$. Define the following projections
\begin{align*}
	\pi_i
	\colon
	\GrBD^{\times l}
	&\to
	\Gr
	\\
	\left(
		L_1, 
		\dots, 
		L_l 
	\right)
	&\mapsto
	\rho_{t_i}
	\left(
		L_i
	\right)
\end{align*}
for $1 \leq i \leq l$ and set
\begin{align*}
	\pi_0
	\colon
	\GrBD^{\times l}
	&\to
	\Gr
	\\
	\left(
		L_1, 
		\dots, 
		L_l 
	\right)
	&\mapsto
	\UnitG \cdot \GO
\end{align*}
for consistency. Then define
\begin{equation*}
	\EBundleTwisted{i}
	=
	\pi_i^* \OOO{1}
	/
	\pi_{i-1}^* \OOO{1}.
\end{equation*}
We are interested in the restriction of these line bundles to $\GrBD^{\ucolambda}_{ \comu } \subset \GrBD^{\times l}$ which we denote by the same symbols. It's easy to check that on the central fiber $\Gr^{\ucolambda}_{\comu}$ these line bundles restrict to $\EBundle{i}$ and come with a similar $\T$-equivariant structure.

The Beilinson-Drinfeld deformation $\GrBD^{\ucolambda}_{ \comu } \subset \GrBD^{\times l}$ is a twistor deformation of $\Gr^{\ucolambda}_{\comu}$ with respect to $\EBundle{1}, \dots, \EBundle{l}$ simultaneously. A shift in the variable $t_i$ is a deformation with respect to $\EBundle{i}$. Since there are relations between the bundles $\EBundle{i}$ (or, equivalently, the first Chern classes $\ChernE{}{i}$), some directions become trivial. For example, the relation
\begin{equation*}
	\sum_{i=1}^l 
	\ChernE{}{i} 
	= 0
\end{equation*}
gives that simultaneous shift of all $t_i$ by the same quantity gives an isomorphic fiber.

By the \refProp{PropSecondCohomologyGeneration} we know that $\ChernE{}{i}$ generate $\CoHlgyk{X,\mathbb{Q}}{2}$, so the universal base of deformations for $\Gr^{\ucolambda}_{\comu}$ is a quotient of the Beilinson-Drinfeld base $\AffSpace^l$ by shifts in "trivial" directions coming from the relations on $\ChernE{}{i}$. We will show it later.

Since we know that
\begin{equation*}
	\GrBD^{\ucolambda}_{ \comu }
\end{equation*}
is the period map, any $\dd \in \HlgyEffk{X,\mathbb{Z}}{2}$ represented by an algebraic curve in a fiber $\GrBD^{\ucolambda}_{ \comu, \left( t_1, \dots, t_l \right) }$ must pair trivially with $\sum_i t_i \ChernE{}{i}$. By the definition of $e_i$ over the generic point in the wall $H_{ij}$ we can have only classes
\begin{equation*}
	\dd = m \left( e_i - e_j \right)
\end{equation*}
for some $m$.

\subsection{Relation between families}

The Beilinson-Drinfeld affine Grassmannian is almost the universal family of deformations for $\Gr{\ucolambda}_{ \comu }$. There are "trivial" directions (e.g. along the big diagonal) in the Beilinson-Drinfeld base, but after taking the quotient along them we get the universal family. More precisely, by the definition of the universal family we have a pullback from the universal family for $X = \Gr^{\ucolambda}_{\comu}$

\begin{equation} \label{EqFamilyPullback}
    \begin{tikzcd}
        \GrBD^{\ucolambda}_{ \comu }
        \arrow[r]
        \arrow[d]
        &
        {X}^{u}
        \arrow[d]
        \\
        \AffSpace^l
        \arrow[r]
        &
        \AffSpace 
        \left( 
            \CoHlgyk{X,\mathbb{C}}{2} 
        \right)
    \end{tikzcd}
\end{equation}
and the relation between the two families is encoded in the map of bases. We claim that this map is a surjection.


\begin{theorem}
    The map between bases in \refEq{EqFamilyPullback} comes from a linear surjection. In other words, the universal family if the quotient of the Beilinson-Drinfeld family along some directions in $\AffSpace^l$.
\end{theorem}
\begin{proof} 
    The $\LoopGm$-actions on both families are compatible (both come from scaling the $t$ coordinate), so the map  
    \begin{equation} \label{EqBasesMap}
        \AffSpace^l
        \to
        \AffSpace 
        \left( 
            \CoHlgyk{X,\mathbb{C}}{2} 
        \right)
    \end{equation}
    is $\LoopGm$-equivaraint. Both affine spaces have coordinate rings generated by linear coordinates, and they have weight $\hbar$. This means that the map \refEq{EqBasesMap} comes from a \textit{linear} map
    
    \begin{equation} \label{EqVectorSpacesMap}
        \mathbb{C}^l
        \to
        \CoHlgyk{X,\mathbb{C}}{2}
    \end{equation}
    
    The final step is note that the map dual to \refEq{EqVectorSpacesMap} is injective
    \begin{equation*}
        \Hlgyk{X,\mathbb{C}}{2}
        \to
        \mathbb{C}^l.
    \end{equation*}

    This holds because the basis of $\Hlgyk{X,\mathbb{C}}{2}$ is the curve classes comes from the classes of curves $C^{ij}_{p,\wtalpha+n\hbar}$ by \refProp{PropCClassesGenerateH2}, and each such curve has a representative in the Beilinson-Drinfeld family for fibers exactly over the wall $t_i = t_j$. The linear functions defining such planes are linearly independent if and only if the corresponding homology classes of $C^{ij}_{p,\wtalpha+n\hbar}$ are linearly independent. This is exactly the injectivity of the dual map.
\end{proof}

\bibliography{Quant}

\end{document}